\documentclass{article}

\usepackage{microtype}
\usepackage{graphicx}
\usepackage{booktabs} 

\usepackage{hyperref}

\usepackage[accepted]{icml2023}

\usepackage{amsmath}
\usepackage{amssymb}
\usepackage{mathtools}
\usepackage{amsthm}

\usepackage[textsize=tiny]{todonotes}

\usepackage[utf8]{inputenc} 
\usepackage[T1]{fontenc}    
\usepackage{hyperref}       
\usepackage{url}            
\usepackage{booktabs}       
\usepackage{amsfonts}       
\usepackage{nicefrac}       
\usepackage{microtype}      
\usepackage{xcolor}         

\usepackage{amsmath, mathtools, amsfonts, amssymb, amsthm, bm, bbm, float,subcaption}
\usepackage{caption}
\usepackage[normalem]{ulem}
\newcommand{\stkout}[1]{\ifmmode\text{\sout{\ensuremath{#1}}}\else\sout{#1}\fi}
\usepackage{slashed}
\usepackage[shortlabels]{enumitem}
\usepackage[ruled, noend,algo2e]{algorithm2e}
\usepackage{color}
\usepackage[title]{appendix}
\usepackage{tikz}
\usetikzlibrary{positioning}
\usepackage{multirow}
\usepackage{thmtools}
\usepackage{thm-restate}

\usepackage{amsmath,amsfonts,bm}

\def\eqref#1{equation~\ref{#1}}

\def\ceil#1{\lceil #1 \rceil}
\def\floor#1{\lfloor #1 \rfloor}
\def\1{\bm{1}}

\def\ve{{\bm{e}}}

\def\vu{{\bm{u}}}
\def\vv{{\bm{v}}}

\def\vx{{\bm{x}}}

\def\mY{{\bm{Y}}}

\DeclareMathAlphabet{\mathsfit}{\encodingdefault}{\sfdefault}{m}{sl}
\SetMathAlphabet{\mathsfit}{bold}{\encodingdefault}{\sfdefault}{bx}{n}

\def\gE{{\mathcal{E}}}
\def\gF{{\mathcal{F}}}

\newcommand\bbR{\ensuremath{\mathbb{R}}} 
\newcommand\bbZ{\ensuremath{\mathbb{Z}}} 
\newcommand\bbN{\ensuremath{\mathbb{N}}} 
\newcommand\bbE{\ensuremath{\mathbb{E}}} 
\newcommand\bbP{\ensuremath{\mathbb{P}}} 
\newcommand\bbS{\ensuremath{\mathbb{S}}} 
\newcommand\calE{\ensuremath{\mathcal{E}}} 

\newcommand{\ep}{\epsilon} 
\newcommand{\ind}{\mathbbm{1}} 

\DeclarePairedDelimiter\abs{\lvert}{\rvert}%
\DeclarePairedDelimiterX{\norm}[1]{\lVert}{\rVert}{#1}
\DeclarePairedDelimiter{\brac}{\langle}{\rangle}
\newcommand{\xpt}[1]{\bbE\left[{#1}\right]}

\DeclarePairedDelimiterX{\infdivx}[2]{(}{)}{%
  #1\;\delimsize\|\;#2%
}

\newcounter{relctr}[section] 

\newcommand\labelrel[2]{%
  \begingroup
    \refstepcounter{relctr}%
    \stackrel{\textnormal{(\roman{relctr})}}{\mathstrut{#1}}%
    \originallabel{#2}%
  \endgroup
}
\AtBeginDocument{\let\originallabel\label}

\newtheorem{theorem}{Theorem}

\newtheorem{lemma}{Lemma}
\newtheorem{condition}{Condition}
\newtheorem{proposition}{Proposition}
\newtheorem{assumption}{Assumption}

\theoremstyle{definition}
\newtheorem{definition}{Definition}

\newboolean{showcomments}
\setboolean{showcomments}{true}
\newcommand{\lina}[1]{  \ifthenelse{\boolean{showcomments}}
	{ \textcolor{red}{(Lina says:  #1)}} {}  }

\theoremstyle{remark}
\newtheorem{remark}{Remark}
\allowdisplaybreaks

\usepackage{authblk}

\usepackage{thmtools}
\usepackage{thm-restate}
\usepackage[capitalise]{cleveref}

\crefname{condition}{Condition}{Condition}

\icmltitlerunning{Escaping saddle points in zeroth-order optimization: the power of two-point estimators}

\begin{document}

\twocolumn[
\icmltitle{Escaping Saddle Points in Zeroth-order Optimization:\\ The Power of Two-point Estimators}




\begin{icmlauthorlist}
\icmlauthor{Zhaolin Ren}{hhh}
\icmlauthor{Yujie Tang}{ppp}
\icmlauthor{Na Li}{hhh}
\end{icmlauthorlist}

\icmlaffiliation{hhh}{John A. Paulson School of Engineering and Applied Sciences, Harvard University}
\icmlaffiliation{ppp}{Department of Industrial Engineering \& Management at Peking University}

\icmlcorrespondingauthor{Zhaolin Ren}{zhaolinren@g.harvard.edu}
\icmlcorrespondingauthor{Na Li}{nali@seas.harvard.edu}

\icmlkeywords{Machine Learning, ICML}

\vskip 0.3in
]



\printAffiliationsAndNotice{}  

\begin{abstract}
Two-point zeroth order methods are important in many applications of zeroth-order optimization, such as robotics, wind farms, power systems, online optimization, and adversarial robustness to black-box attacks in deep neural networks, where the problem may be high-dimensional and/or time-varying. Most problems in these applications are nonconvex and contain saddle points. While existing works have shown that zeroth-order methods utilizing $\Omega(d)$ function valuations per iteration (with $d$ denoting the problem dimension) can escape saddle points efficiently, it remains an open question if zeroth-order methods based on two-point estimators can escape saddle points. In this paper, we show that by adding an appropriate isotropic perturbation at each iteration, a zeroth-order algorithm based on $2m$ (for any $1 \leq m \leq d$) function evaluations per iteration can not only find $\ep$-second order stationary points polynomially fast, but do so using only $\tilde{O}(\nicefrac{d}{m\ep^{2}\bar{\psi}})$ function evaluations, where $\bar{\psi} \geq \tilde{\Omega}(\sqrt{\ep})$ is a parameter capturing the extent to which the function of interest exhibits the strict saddle property.
\end{abstract}

\section{Introduction}
\label{section:introduction}

Two-point estimators, which approximate the gradient using two function evaluations per iteration, have been widely studied by researchers in the zeroth-order optimization literature, in convex \citep{nesterov2017random,duchi2015optimal,shamir2017optimal}, nonconvex \citep{nesterov2017random}, online \citep{shamir2017optimal}, as well as distributed settings \citep{tang2019texttt}. A key reason for doing so is that for applications of zeroth-order optimization arising in robotics \citep{li2022stochastic}, wind farms \citep{tang2020zeroth}, power systems \citep{chen2020model}, online (time-varying) optimization \citep{shamir2017optimal}, learning-based control \citep{malik2019derivative,li2021distributed}, and improving adversarial robustness to black-box attacks in deep neural networks \citep{chen2017zoo}, it may be costly or impractical to wait for $\Omega(d)$ (where $d$ denotes the problem dimension) function evaluations per iteration to make a step. This is especially true for high-dimensional and problems with time-varying noise. See \cref{appendix:related_work} for more discussion.

However, despite the advantages of zeroth-order methods with two-point estimators, there has been a lack of existing work studying the ability of two-point estimators to escape saddle points in nonconvex optimization problems. Since nonconvex problems arise often in practice, it is crucial to know if two-point algorithms can efficiently escape saddle points of nonconvex functions and converge to second-order stationary points (see \cref{definition:ep_approx_2nd_order_stationary_point} for a definition).

To motivate the challenges of escaping saddle points using two-point zeroth-order methods, we begin with a review of escaping saddle points using first-order methods. The problem of efficiently escaping saddle points in deterministic first-order optimization (with exact gradients) has been carefully studied in several earlier works \citep{jin2017escape,jin2018accelerated}. A key idea in these works is the injection of an isotropic perturbation whenever the gradient is small, facilitating escape from a saddle if a negative curvature direction exists even without actively identifying the direction. 
However, the analysis of efficient saddle point escape for stochastic gradient methods is often more complicated. In general, the behavior of the stochastic gradient near the saddle point can be difficult to characterize. Hence, strong concentration assumptions are typically made on the stochastic gradients being used, such as subGaussianity, boundedness of the variance or a bounded gradient estimator \citep{ge2015escaping,daneshmand2018escaping, xu2018first, fang2019sharp, roy2020escaping, vlaski2021second}, creating an analytical issue when such idealized assumptions fail to hold.

Indeed, though zeroth-order methods can be viewed as stochastic gradient methods, common zeroth order estimators, such as two-point estimators \citep{nesterov2017random}, are not subGaussian, and can have unbounded variance. For instance,  it can be shown that the variance of the two-point estimator is on the order of $\Omega(d \norm*{\nabla f(x)}^2)$ \citep{nesterov2017random}, with both a dependence on the problem dimension $d$ as well as on the norm of the gradient, which can be unbounded. Due to non-subGaussianity and unboundedness, it is tricky to bound the effect of such zeroth-order estimators and establish tight concentration inequalities that facilitate its escape near saddle points. In addition, the large variance of the zeroth-order estimator is also an issue in non-saddle regions, i.e. when the gradient is large. While this is not an issue to show function improvement in expectation, as we discuss later, this becomes an issue when guaranteeing high probability bounds.

Due to these difficulties, previous works on escaping saddle points in zeroth-order optimization have exclusively focused on approaches requiring $\Omega(d)$ function evaluations per iteration to accurately estimate the gradient \citep{jin2018local, bai2020escaping,vlatakis2019efficiently}, or in some cases negative curvature directions \citep{zhang2022zeroth, lucchi2021second} or the Hessian itself~\citep{balasubramanian2022zeroth}, reducing in a sense the zeroth-order problem back to a first-order one. However, as explained earlier, two-point zeroth-order algorithms are important for high-dimensional and/or time-varying problems in many applications areas. This raises an important question:

\begin{center}
    \textbf{Can two-point zeroth-order methods escape saddle points and reach approximate second order stationary points efficiently?}
\end{center}

\textbf{Our Contribution.} In this work, we show that by adding an appropriate isotropic perturbation at each iteration, a zeroth-order algorithm based on \emph{any} number $m$ of pairs ($1\leq m\leq d$) of function evaluations per iteration can not only find $(\ep,\sqrt{\ep})$-second order stationary points (cf. the definition later in \cref{definition:ep_approx_2nd_order_stationary_point}) polynomially fast, but do so using only $\tilde{O}(\nicefrac{\operatorname{polylog}(\frac{1}{\delta})d}{\ep^{2.5}})$ function evaluations, with a probability of at least $1 - \delta$. In particular, this proves that using a single two-point zeroth-order estimator at each iteration (with appropriate perturbation) suffices to efficiently escape saddle points in zeroth-order optimization, with high probability. Moreover, for functions that are $(\ep,\psi,O(\sqrt{\ep}))$ strict-saddle (see \cref{definition:ep_gamma_strict_saddle} for a definition of strict saddle functions), our results become $\tilde{O}(\nicefrac{\operatorname{polylog}(\frac{1}{\delta})d}{\psi \ep^{2}})$, which is a significant improvement when $\psi \gg \ep$; strict saddle functions have been identified as an important class of functions in nonconvex optimization, with several well-known examples such as tensor decomposition~\citep{ge2015escaping}, dictionary learning and phase retrieval~\citep{sun2015nonconvex}. A comparison of our results with existing zeroth-order and first-order methods is shown in \cref{tab:convergence_comparison}. We also provide numerical results in \cref{main:simulations} showing that our proposed two-point algorithm requires fewer total function evaluations to converge than zeroth order methods that use $2d$ function evaluations per iteration, for a nonconvex test function proposed in \cite{du2017gradient}.

To overcome the theoretical challenges that were discussed earlier, we i) first show, via a careful analysis, that zeroth order methods can make function value improvement across iterates with large gradients with high probability, even when only a single two-point estimator (which can have significant variance at large gradients) is used per iteration.  ii) Second, near saddle points, we overcome issues caused by the unbounded variance and non-subGaussinity of zeroth-order gradient estimators by developing new technical tools, including novel martingale concentration inequalities involving Gaussian vectors, to tightly bound such terms. In turn, this allows us to show that the noise emanating from the zeroth-order estimators will not overwhelm the effect of the additional isotropic perturbative noise, facilitating escape along negative curvature directions. To the best of our knowledge, both analyses are novel, and may be independent contributions on their own.

\begin{table*}[t]
    \centering
    \begin{tabular}{c|c|c|c}
    \multicolumn{2}{c}{} & Iteration Complexity & Fun. Evaluations. per iter \\
    \hline
     \multirow{2}{*}{First-order} & \cite{jin2017escape}  (deterministic)   & $\tilde{O}\left(\frac{1}{\ep^2} \right)$ &  --- \\
         & \cite{fang2019sharp} (SGD) & $\tilde{O}\left(\frac{1}{\ep^{3.5}}\right)$ & --- \\
    \hline
    \multirow{7}{*}{Zeroth-order} & \cite{jin2018local} & $\tilde{O}\left(\frac{1}{\ep^2} \right)$ & $\tilde{O}\left( \frac{d^2}{\ep^3}\right)$ \\
    & \cite{bai2020escaping} & $\tilde{O}\left(\frac{1}{\ep^2} \right)$ & $\tilde{O}\left(\frac{d^2}{\ep^{8}} \right)$ \\
        & \cite{vlatakis2019efficiently} & $\tilde{O}\left(\frac{1}{\ep^2} \right)$ & $\tilde{O}\left(d\right)$ \\
        & \cite{balasubramanian2022zeroth} & $\tilde{O}\left(\frac{1}{\ep^{1.5}} \right)$ & $\tilde{O} \left(\frac{d}{\ep^{2}} + \frac{d^4}{\ep} \right)$ \\
        & \cite{lucchi2021second}$^\dagger$ & $\tilde{O}\left(\frac{1}{\ep^2} \right)$ & $\tilde{O}\left(\frac{d}{\ep^{2/3}} \right)$ \\
        & \cite{zhang2022zeroth} & $\tilde{O}\left(\frac{1}{\ep^2}\right)$ & $\tilde{O}(d)$ \\
        & \cref{algorithm:ZOPGD} (this paper, $1\leq m \leq d $)$\ddag$ & $\tilde{O}\left(\frac{d}{\ep^{2} \bar{\psi} m} \right)$ & $2m$
    \end{tabular}
    \caption{Selected comparison of convergence results to $(\ep, O(\sqrt{\ep})$-second order stationary points in smooth, nonconvex functions; for $^\dagger$, the convergence is to $(\ep,\ep^{2/3})$-second order stationary points. For $^\ddag$, the term $\bar{\psi}$ in the denominator is (i) $\psi$ when the function $f$ is $(\ep,\psi, O(\sqrt{\ep}))$-strict saddle for a $\psi > O(\sqrt{\ep})$ (see \cref{definition:ep_gamma_strict_saddle} for a definition) and (ii) $O(\sqrt{\ep})$ if otherwise.}
    \label{tab:convergence_comparison}
\end{table*}

\textit{Related Work.}
Due to space considerations, we defer a full discussion of related work to \cref{appendix:related_work}.

\section{Problem Setup}
We make the following assumptions on the class of functions $f: \bbR^d \to \bbR$ which we consider.

\begin{assumption}[Properties of $f$]
We suppose that $f: \bbR^d \to \bbR$ satisfies the following properties:
\begin{enumerate}
\item $f$ is twice-differentiable and lower bounded, i.e. 
$f^* := \min_x f(x) > -\infty.$
\item $f$ is $L$-gradient Lipschitz, i.e. 
$$\norm*{\nabla f(x) - \nabla f(y)} \leq L \norm*{x - y} \ \ \forall x,y \in \bbR^d.$$
\item $f$ is $\rho$-Hessian Lipschitz, i.e. 
$$\norm*{\nabla^2 f(x) - \nabla^2 f(y)} \leq \rho \norm*{x-y} \ \ \forall x,y \in \bbR^d.$$
\end{enumerate}
\end{assumption}
In our work, we focus on finding approximate second order stationary points, defined below.
\begin{definition}
\label{definition:ep_approx_2nd_order_stationary_point}
A point $x \in \bbR^d$ is an $(\ep,\varphi)$-second order stationary point if
$$\norm*{\nabla f(x)} < \ep, \quad \mbox{ and } \quad \lambda_{\min}(\nabla^2 f(x)) > -\varphi.$$
\end{definition}
We define an $(\ep,\varphi)$-approximate saddle point as follows.
\begin{definition}
\label{definition:ep_saddle_point}
A point $x \in \bbR^d$ is an $(\ep,\varphi)$-approximate saddle point, if 
\begin{align*}
    \norm*{\nabla f(x)} < \ep, \quad \mbox{ and } \quad \lambda_{\min}(\nabla^2 f(x)) \leq -\varphi.
\end{align*}
\end{definition}
Following past convention \citep{jin2019nonconvex}, we will focus in particular on escaping $(\ep,\sqrt{\rho \ep})$-saddle points. For notational simplicity, in following text, we refer to $(\ep,\sqrt{\rho \ep})$-saddle points simply as $\ep$-saddle points and $(\ep,\sqrt{\rho \ep})$-second order stationary points as $\ep$-second order stationary points.
 Beyond the definition of $\ep$-approximate saddle points above, it is known that many nonconvex functions with saddle points, such as orthogonal tensor decomposition \citep{ge2015escaping}, phase retrieval and dictionary learning \citep{sun2015nonconvex}, satisfy what is known as a \emph{strict saddle} condition \citep{ge2015escaping}. For the Hessians of the saddle points of such functions, there is always a strict negative eigenvalue whose magnitude is bounded from below. We provide a precise definition below.
\begin{definition}
\label{definition:ep_gamma_strict_saddle}
A twice-differential function $f(x)$ is $(\ep,\psi, \varrho)$-strict saddle for any $\psi > \varrho > 0$, if for any point $x$, either
\begin{enumerate}
    \item $\norm*{\nabla f(x)} \geq \ep$ holds,
    \item or when $\norm*{\nabla f(x)} < \ep$ holds, either 
    \begin{enumerate}
        \item $\lambda_{\min}(\nabla^2 f(x)) \leq -\psi$, or
        \item $\lambda_{\min}(\nabla^2 f(x)) > - \varrho$.
    \end{enumerate}
\end{enumerate}
\end{definition}

In our work, we consider the following batch symmetric two-point zeroth-order estimator. 

\begin{definition}[(Batch) two-point zeroth-order estimator with perturbation]
\label{definition:batch_2pt_estimator}
We define a $m$-batch two-point zeroth order estimator as follows:
\begin{align}
    \label{eq:ZO_estimator_m}
    g_{u}^{(m)}(x) := \frac{1}{m} \sum_{i=1}^m\frac{f(x + uZ_i) - f(x - uZ_i)}{2u} Z_i,
\end{align}
where $Z_i \stackrel{i.i.d}{\sim} N(0,I)$, and $u > 0$ is a smoothing radius.
\end{definition}
Such $2m$ zeroth-order gradient estimators have frequently been studied in zeroth-order optimization works (see e.g. \cite{nesterov2017random}). To facilitate efficient escape from saddle points, our proposed \cref{algorithm:ZOPGD} adds isotropic perturbation at each iteration.

\begin{algorithm}[t]
\caption{Zeroth-order perturbed gradient descent (ZOPGD)}
\label{algorithm:ZOPGD}
\DontPrintSemicolon
\SetAlgoNoLine
\SetKwInOut{Input}{input}
\Input{$x_0$, horizon $T$, step-size $\eta$, smoothing radius $u$, perturbation radius $r$, batch size $m$}
\For{step $t = 0,\dots,T$}{
Sample $Z^{(m)} = \{Z_{t,i}\}_{i=1}^m \sim N(0,I)$ to compute $g_{u}^{(m)}(x_t))$, defined in \cref{eq:ZO_estimator_m}. \\
Update $x_{t+1} = x_t - \eta \left(g_{u}^{(m)}(x_t) + Y_t \right),$ \ \ where $Y_t \sim N(0, \frac{r^2}{d} I)$
}
\end{algorithm}

We now state an informal version of our main result, and follow that with a few remarks. 
\begin{theorem}[Main result, informal version of ~\cref{theorem:convergence_full}]
\label{theorem:convergence_main}
Consider running \cref{algorithm:ZOPGD}. Let $\tilde{O}$ hide polylogarithmic terms in $\delta$ and other parameters. Suppose $\delta \in (0,1/e]$. Suppose $\sqrt{\rho \ep} \leq \min \{1, L\}$\footnote{In our paper, we focus on the case $\sqrt{\rho \ep} \leq L$; otherwise, by the $L$-Lipschitz assumption, $\lambda_{\min}(\nabla^2 f(x)) \geq -L$ for all $x \in \bbR^d$, which implies $\ep$-first order stationary points are also $\ep$-second order stationary points.}, such that $\bar{\psi} \leq \min\{1,L\}$, where
\small
\begin{align}
    \bar{\psi}\! :=\! \begin{cases}
\min\{\psi,1,L\} \!&\! \mbox{if }\exists \psi\!  >\! \sqrt{\rho \ep} \mbox{ s.t. } f(\cdot) \mbox{ is } (\ep,\psi, \sqrt{\rho \ep})\mbox{-strict saddle }  \\
\sqrt{\rho \ep}\! &\! \mbox{if otherwise}.
\end{cases} \label{eq:psi_bar_definition}
\end{align}
\normalsize
Suppose \small
$$u\! =\! \tilde{O}\left(\frac{\min\left\{\sqrt{\ep}, \sqrt{r}\right\}}{\sqrt{\rho}d} \right), \ 
r \!=\! \tilde{O}\left(\ep\right), \ \eta \! =\! \tilde{O}\left(\frac{m \bar{\psi}}{d \max\{L,L^2 \}} \right),$$
\normalsize
Then, in 
\begin{align*}
{T} &{= \tilde{\Omega} \left(\frac{(f(x_0) - f^*)}{\eta \ep^2 } + \frac{\rho^2 (f(x_0) -  f^*)}{\eta \bar{\psi}^4} \right)} \\
&{= \tilde{\Omega}\left( \frac{d \max\{L,L^2\} \rho^2 (f(x_0) -  f^*)}{m \bar{\psi} \ep^2}\right)}
\end{align*}
iterations (with each iteration using $2m$ function evaluations), with probability at least $1 - \delta$, at least half the iterates are $(\ep,\sqrt{\rho \ep})$-second-order stationary points.
\end{theorem}
\begin{remark}
As the choice of $\eta$ in \cref{proposition:func_decrease_contradiction} (\cref{appendix:function_decrease}) and \cref{theorem:convergence_full} (\cref{appendix:main_result}) respectively imply, the $\tilde{\Omega}\left(\frac{f(x_0) - f^*}{ \eta \ep^2}\right)$ term in the sample complexity comes from the large gradient iterations (\cref{proposition:func_decrease_contradiction}), whereas the $\tilde{\Omega}\left(\frac{\rho^2 (f(x_0) -  f^*)}{ \eta  \bar{\psi}^4}\right)$ term comes from the escape saddle point phase. 
\end{remark}
\begin{remark}
As a corollary of \cref{theorem:convergence_main}, for functions $f$ which are $(\ep,\psi, \sqrt{\rho \ep})$ strict saddle, assuming that $\psi \geq \sqrt{\rho \ep}$, the sample complexity of our algorithm scales as $\tilde{\Omega}\left(\frac{d \max\{L^2,L\} (f(x_0) - f^*)}{m \ep^2 \psi} \right)$, which scales as $\tilde{\Omega}\left(\frac{d}{m \ep^2} \right)$ when $\psi$ is of size $\Omega(1)$. Thus, in this setting, for two-point estimators, where $m = 1$, the dependence on $d$ and $\ep$ in our sample complexity (as measured by function evaluations) matches that achieved by the algorithms in \cite{vlatakis2019efficiently,zhang2022zeroth}, which have to use $2d$ function evaluations per iteration to estimate the gradient.
\end{remark}

\textbf{Comparison to gradient-based methods.} For first-order escape saddle point algorithms, standard perturbation-based methods (without acceleration) can find a $(\ep,O(\sqrt{\ep}))$-second-order stationary point using $\tilde{O}(1/\ep^2)$ iterations for deterministic GD \citep{jin2019nonconvex}, while for standard SGD the best-known rates are slower at $\tilde{O}(1/\ep^{3.5})$ \citep{fang2019sharp}. In contrast, our sample complexity (as measured by the total number of function evaluations) is $\tilde{O}\left(\frac{d}{\ep^{2} \bar{\psi}} \right)$, where $\bar{\psi}$ is defined in \cref{eq:psi_bar_definition}. The extra (linear) dependence on $d$ is typical for zeroth-order algorithms (see e.g. \cite{nesterov2017random}); intuitively, gradient calculation for $d$-dimensional functions requires $O(d)$ calculations agnostically, so it makes sense that zeroth-order algorithms requires $d$ times more iterations. For general non strict-saddle functions, our dependence on $\ep$ sits between that of the deterministic methods and SGD methods, and suggests the benefit of a specialized treatment of zeroth-order methods over considering them simply as a subclass of SGD methods. Moreover, for $(\ep,\psi,\sqrt{\rho \ep})$- strict-saddle functions where $\psi = \Omega(1)$, our sample complexity becomes $\tilde{O}(\frac{d}{\ep^2})$, with an $\ep$ dependence that matches that of the best existing sample complexity for non-accelerated first-order escape saddle point methods \cite{jin2017escape}

\textbf{Comparison to existing zeroth-order methods.} As \cref{tab:convergence_comparison} suggests, our sample complexity significantly outperforms that of \cite{jin2018local}, \cite{bai2020escaping}, \cite{balasubramanian2022zeroth}, and also that in \cite{lucchi2021second}, which is a random search method. We note that the sample complexity in \cite{vlatakis2019efficiently,zhang2022zeroth} outperform our method, with a function evaluation complexity of $\tilde{O}\left(\frac{d}{\ep^2} \right)$. However, for for $(\ep,\psi,\sqrt{\rho \ep})$- strict-saddle functions where $\psi = \Omega(1)$, our sample complexity becomes $\tilde{O}(\frac{d}{\ep^2})$, which matches the sample complexity in \cite{vlatakis2019efficiently,zhang2022zeroth}. Moreover, a key limitation of their methods is a requirement to use $\Omega(d)$ function evaluations to estimate the gradient at each iteration, which may not be practical in realistic applications when $d$ is large. In contrast, our method supports any number of function evaluations at each iteration between $1$ to $d$. Moreover, numerically, we found that for a test nonconvex function proposed in \cite{du2017gradient}, our method (with two-point estimators) takes fewer function evaluations to escape saddle points and converge to the global minimum than the methods in \cite{vlatakis2019efficiently,zhang2022zeroth}; see \cref{main:simulations} for details.



\section{Proof strategy and key challenges in the zeroth-order setting}

Broadly speaking, our proof include two major parts, i) characterizing the progress made in iterations when the gradient is large (which we can define to be iterations $t$ where $\norm*{\nabla f(x_t)} \geq \ep$) (Section~\ref{subsec:large-gradient}), ii) and iterations when we are at an $\ep$-approximate saddle point (where progress may be made along the negative eigendirection of the Hessian matrix) (Section~\ref{subsec:small gradient}). While the approach is similar to the first-order case (e.g. \cite{jin2019nonconvex}), the zeroth-order setting brings forth several unique challenges.   In the rest of this section, we explain these challenges, sketch out our high-level proof outlines, and provide statements of the main technical results. Due to limited space, we defer the full proof to the Appendix.



\subsection{Showing function decrease when gradients are large}\label{subsec:large-gradient}
\textbf{Challenge.} Due to the noise in two-point (or $2m$ where $m$ is a small constant) zeroth-order gradient, even when the gradient is large, it may not always be possible to make progress at each iteration, especially when $m< d$ is used in the gradient estimation equation in \cref{eq:ZO_estimator_m}. While it is tempting to use an expectation-based argument to handle this issue, it is known that expectation-based function decrease arguments are insufficient for the purpose of escaping saddle points (see e.g. Proposition 1 in \cite{ziyin2021sgd}). We tackle this issue by using high-probability arguments instead; we note that achieving these high-probability bounds is highly nontrivial due to the large variance of the two-point zeroth-order estimator (scaling with $d$ times the squared norm of the gradient). Hence, any single iteration of the zeroth-order method may in fact lead to a function increase rather than decrease.

\textbf{High-level proof outline.} \textbf{(i)} We first characterize the function value change for our proposed algorithm (\cref{lemma:function_decrease_tighter_decomp}). \textbf{(ii)} Next, we tackle the issue of the possibility that the function value might increase for any given iteration. The key idea here is that across any small consecutive number of iterations, there will be one iteration where the zeroth-order estimator is sufficiently aligned with the gradient direction~(\cref{lemma:some_t_Z_bdd_away_from_0} in \cref{appendix:function_decrease}). \textbf{(iii)} Along with a series of other technical results in \cref{appendix:function_decrease}, we then show that the function makes sufficient progress across the duration of the algorithm, with high probability (\cref{proposition:func_decrease_contradiction_informal}). To more concretely illustrate the key analytical challenge, we next introduce the following function decrease lemma, proved in \cref{appendix:function_decrease}. 

\begin{restatable}[Function decrease for batch zeroth-order optimization]{lemma}{funcDecreaseMain}
\label{lemma:function_decrease_tighter_decomp}
Suppose at each time $t$, the algorithm performs the update step (with batch-size parameter $1 \leq m \leq d$)
\begin{align*}
x_{t+1} = x_t - \eta \left(g_u^{(m)}(x_t) + Y_t\right),
\end{align*}
where
\begin{align*}
    g_u^{(m)}(x_t) = \frac{1}{m} \sum_{i=1}^m\frac{f(x_t + uZ_{t,i}) - f(x_t - uZ_{t,i})}{2u} Z_{t,i},
\end{align*}
where each $Z_{t,i}$ is drawn i.i.d from $N(0, I)$, $u > 0$ is the smoothing radius, and $Y_t \sim N(0,\frac{r^2}{d} I)$ with $r > 0$ denoting the perturbation radius.

Then, there exist absolute constants $c_1>0, C_1 \geq 1$ such that, for any $T\in\bbZ^+$ and $T  \geq \tau > 0$, $\alpha > 0$ and $\delta \in( 0,1/e]$, upon defining $\mathcal{H}_{0,\tau}(\delta)$ to be the event on which the inequality
\small
\begin{align}
& \ f(x_{\tau}) -f(x_0) \\
\leq &\ 
-\frac{3\eta}{4} \sum_{t=0}^{\tau-1} \frac{1}{m} \sum_{i=1}^m \abs*{Z_{t,i}^\top \nabla f(x_t)}^2  \\
& \ + \left(\frac{\eta}{\alpha} + \frac{c_1 L\eta^2 \chi^3 d}{m}\right) \sum_{t=0}^{\tau-1} \norm*{\nabla f(x_t)}^2
\nonumber \\
&
\ + \tau \eta u^4 \rho^2\cdot c_1 d^3 \left(\log\frac{T}{\delta}\right)^3
+ \tau L \eta^2 u^4 \rho^2\cdot c_1 d^4 \left(\log\frac{T}{\delta}\right)^4 
\nonumber \\
& \ + \eta  c_1r^2(\alpha + \eta L)\log\frac{T}{\delta} + \tau c_1 L\eta^2  r^2  \label{eq:func_decrease_first_lemma}
\end{align}
\normalsize
is satisfied (where $\chi \coloneqq \log(C_1 dm T/\delta)$), we have
\small
$$
\bbP(\mathcal{H}_{0,\tau}(\delta))\geq 1-\frac{(\tau + 4)\delta}{T}, \qquad \bbP(\cap_{\tau = 1}^{\tau'}\mathcal{H}_{0,\tau}(\delta))\geq 1-\frac{5\tau'\delta}{T}
$$
\normalsize
for any $0 \leq \tau' \leq T$.
\end{restatable}

Our goal is to show that we can arrive at a contradiction $f(x_T) < \min_x f(x)$ when there is a large number of steps at which $\norm*{\nabla f(x_t)} \geq \ep$ (\cref{proposition:func_decrease_contradiction_informal}). As we can see from \cref{eq:func_decrease_first_lemma}, this implies that we need to prove a lower bound of the form
\small
\begin{align}
\label{eq:Z_dot_nabla_sq_geq_nabla_sq_roughly-main}
\sum_{t=0}^{T-1} \frac{1}{m} \sum_{i=1}^m \norm*{Z_{t,i}^\top \nabla f(x_t)}^2 \!\geq\! \Omega \left(\frac{1}{\alpha} + \frac{c_1 L\eta \chi^3 d}{m} \right) \sum_{t=0}^{T-1} \norm*{\nabla f(x_t)}^2 
\end{align}
\normalsize
for some $\alpha$ which is not too large (an example would be picking $\alpha$ such that it only scales logarithmically in the problem parameters). 
However, it is tricky to prove such a lower-bound in the zeroth-order setting. In particular, for small batch-sizes $m$, $\frac{1}{m} \sum_{i=1}^m \norm*{Z_{t,i}^\top \nabla f(x_t)}^2$ could be small even as $\norm*{\nabla f(x_t)}^2$ is large; this is because for each $i \in [m]$,  $Z_{t,i}$ could have a negligible component in the $\nabla f(x_t)$ direction. This necessitates a more delicate analysis to prove a bound similar to
\cref{eq:Z_dot_nabla_sq_geq_nabla_sq_roughly-main}. Due to space reasons, we defer our more detailed proof approach outline to \cref{appendix:function_decrease} (see the discussion immediately following \cref{lemma:function_decrease_tighter_decomp}). The results in \cref{appendix:function_decrease} culminates in the following result which limits the number of large-gradient. 
\begin{proposition}
[Bound on number of iterates with large gradients, informal version of Proposition~\ref{proposition:func_decrease_contradiction}] \label{proposition:func_decrease_contradiction_informal}
Let $\delta \in (0,1/e]$ be arbitrary. Letting $\tilde{O}$ hide polylogarithmic dependencies on $\delta$ (and other parameters), consider choosing $u$, $r$, $\eta$ and $T$ such that
\small
\begin{align*}
&u = \tilde{O}\left(\frac{\sqrt{\ep}}{\sqrt{\rho}d} \right),
\
r = O(\ep), \ \eta = \tilde{O}\left(\frac{m}{d L} \right), \\
& T = \tilde{\Omega}\left( \frac{\left(\left(f(x_0) - f^*\right) + \ep^2/L) \right)}{\eta \ep^2}\right).
\end{align*}
\normalsize
Then, with probability at least $1 - O(\delta)$, there are at most $T/4$ iterations for which $\norm*{\nabla f(x_t)} \geq \ep$.
\end{proposition}

\subsection{Making progress near saddle points}\label{subsec:small gradient}
\textbf{Challenge.} The noise in two-point zeroth-order estimators makes the analysis around $\epsilon-$approximate saddle points challenging, because the concentration properties of the (non-subGaussian) noise are hard to characterize. Intuitively, a noisier estimator might facilitate easier escape from saddle point. However, without an appropriate concentration bound, the noise may behave in unpredictable ways, preventing escape from saddle regions. Previous analysis of saddle point escape using stochastic estimators typically requires these estimators to satisfy subGaussian properties \citep{jin2019nonconvex,fang2019sharp}, which zeroth-order estimators do not satisfy.

\textbf{High-level proof outline.} \textbf{(i)} We first prove a technical result showing that the travelling distance of the iterates can be bounded in terms of the function value decrease (i.e., Improve or Localize, \cref{lemma:improve_or_localize_exact_informal}). 
\textbf{(ii)} Next,  at any $\ep$-saddle point, we consider a coupling argument and define two sequences running near-identical zeroth-order dynamics, differing only in the sign of their perturbative term along the minimum eigendirection of $H$ , which denotes the Hessian of the saddle (\cref{lemma:saddle_diff_term_expression}). Using \cref{lemma:improve_or_localize_exact_informal} in point (i), if we assume for contradiction that the two sequences both ``get stuck'' and make little function value progress, the dynamics of the difference between the two sequences will remain small as both sequences remain close to the saddle point. \textbf{iii)} However, since the perturbation vectors of the two sequences differ in the (most) negative direction of $H$, the norm of the the difference of the two sequences will grow exponentially so long as \textit{a).} the sequences remain close to the saddle point (and thus the Hessian has a negative curvature direction) and \textit{b).} the effect of the zeroth-order stochastic noise can be controlled. This leads to a contradiction, implying that sufficient function decrease must have been made (\cref{proposition:saddle_point_func_decrease} in \cref{appendix:saddle_point_decrease}). \textbf{(iv)} To show that the zeroth-order stochastic noise can be controlled, we prove one technical result (\cref{proposition:nSG_adapted_to_ZO}), providing a concentration bound for the product of (possibly unbounded) subGaussian random vectors that scales linearly with the dimension $d$. This enables us to control the effect of the zeroth-order noise near saddle points, and is essential in showing that the eventual sample complexity scales linearly with $d$.

We provide a more detailed proof sketch below, where we elaborate more on our analytical challenges and ideas. 
We first introduce an informal statement of a key technical result that bounds, with high probability, the travelling distance of the iterates in terms of the function value decrease.

\begin{lemma}[Improve or Localize, informal version of Lemma~\ref{lemma:improve_or_localize_exact}]
\label{lemma:improve_or_localize_exact_informal}
Consider the perturbed zeroth-order update \cref{algorithm:ZOPGD}. Let $\delta \in (0,1/e]$ be arbitrary. Consider any $T_s = \tilde{\Omega}\left( \frac{1}{m} \log(1/\delta)\right)$, and any $t_0 \geq 0$. For any $F > 0$, suppose $f(x_{T_s+t_0}) - f(x_{t_0}) > - F,$ i.e. $f(x_{t_0}) - f(x_{T_s+t_0}) < F$. Letting $\tilde{O}$ hide polylogarithmic terms involving $\delta$, suppose 
\small
\begin{align*}
& u = \tilde{O}\left(\frac{\min\left\{\sqrt{\ep}, \sqrt{r}\right\}}{\sqrt{\rho}d} \right),
\
r = \tilde{O}\left(\min\left\{\ep, \frac{F}{\eta T_s} \right\} \right), \\
&\eta = \tilde{O}\left(\frac{m \sqrt{\rho \ep}}{d L} \right).
\end{align*}
\normalsize
Then, with probability at least $1 - O\left(\frac{T_s \delta}{T} \right)$ (here $T \geq T_s$ denotes the total number of iterations), for each $\tau \in \{0,1,\dots,T_s\}$, we have that
\begin{align*}
     &\norm*{x_{t_0 + \tau} - x_{t_0}}^2 \leq \phi_{T_s}(\delta,F), \quad \mbox{where } \\
     &\phi_{T_s}(\delta,F) \! = \! \tilde{O}\left(\max\left\{T_s,\frac{d}{m} \right\} \right)\eta F + \tilde{O}(\eta^2 \ep^2).  
\end{align*}
\end{lemma}
Intuitively, the above result shows that if little function value improvement has been made, then the algorithm's iterates have not moved much, such that it remains approximately in a saddle region if it started out in a saddle region. Next, \Cref{lemma:saddle_diff_term_expression} formally introduces the coupling we have mentioned, setting the stage for the rest of our arguments. For notational convenience, in this section, unless otherwise specified, we will assume that the initial iterate $x_0$ is an $\ep$-saddle point.

\begin{restatable}{lemma}{saddleDiffExpression}
\label{lemma:saddle_diff_term_expression}
Suppose $x_0$ is an $\ep$-approximate saddle point. Without loss of generality, suppose that the minimum eigendirection of $H := \nabla^2 f(x_{0})$ is the $e_1$ direction (i.e. the first basis vector in $\bbR^d$), and let $\gamma$ to denote $-\lambda_{\min}(\nabla^2 f(x_0))$ (note $\gamma \geq \sqrt{\rho \ep}$). Consider the following coupling mechanism, where we run the zeroth-order gradient dynamics, starting with $x_{0}$, with two isotropic noise sequences, $Y_t$ and $Y_t'$ respectively, where $(Y_t)_1 = -(Y_t)_1'$, and $(Y_t)_j = (Y_t)_j'$ for all other  $j \neq 1$. Suppose that the sequence $\{Z_{t,i}\}_{t \in T, i \in [m]}$ is the same for both sequences. Let $\{x_t\}$ denote the sequence with the $\{Y_t\}$ noise sequence, and let the $\{x_t'\}$ denote the sequence with the $\{Y_t'\}$ noise sequence, where $x_{0}' = x_{0}, $ and
\small
\begin{align*}
 & \ x_{t+1}' \\
 = & \ x_t'\! -\! \eta\!\left(\!\frac{\sum_{i=1}^m \left(Z_{t,i}Z_{t,i}^\top \nabla f(x_t') \!+\! \frac{u}{2}Z_{t,i}Z_{t,i}^\top \tilde{H}'_{t,i} Z_{t,i}\right)}{m}\! +\! Y_t'\!\right)\!\!,
\end{align*}
\normalsize
and 
$\tilde{H}_{t,i}' := \frac{H_{t,i,+}' - H_{t,i,-}'}{2}$, with $H_{t,i,+}' = \nabla^2 f(x_t' + \alpha_{t,i,+}' uZ_i')$ for some $\alpha_{t,i,+}' \in [0,1]$, and $H_{t,i,-}' = \nabla^2 f(x_t - \alpha_{t,i,-}' uZ_i')$ for some $\alpha_{t,i,-}' \in [0,1]$.
Then, for any $t \geq 0$,
{\small
\begin{align*}
 & \ \hat{x}_{t+1} \\
    := &  \ x_{t+1} - x_{t+1}' \\
   = & \ -\underbrace{\eta \sum_{\tau = 0}^t (I \!-\! \eta H)^{t - \tau} \hat{\xi}_{g_0}(\tau)}_{W_{g_0}(t+1)} \!-\! \underbrace{\eta \sum_{\tau = 0}^t (I\! -\! \eta H)^{t - \tau} (\bar{H}_{\tau}\! -\! H)\hat{x}_{\tau}}_{W_{H}(t+1)} 
   \\
   &- \underbrace{\eta \sum_{\tau = 0}^t (I - \eta H)^{t - \tau}\hat{\xi}_u(\tau)}_{W_u(t+1)} -\underbrace{\eta \sum_{\tau = 0}^t (I - \eta H)^{t - \tau} \hat{Y}_{\tau}}_{W_{p}(t+1)}
\end{align*}
}
where 
{\small
\begin{align*}
    &\xi_{g_0}(t) = \frac{1}{m}\sum_{i=1}^m (Z_{t,i}Z_{t,i}^\top\! - \!I) \nabla f(x_t), \\
    & \xi'_{g_0}(t) = \frac{1}{m} \sum_{i=1}^m (Z_{t,i}
(Z_{t,i})^{\top} \!-\! I) \nabla f(x_t'), \\
& \hat{\xi}_{g_0}(t) = \xi_{g_0}(t) - \xi'_{g_0}(t), \  \xi_{u}(t) = \frac{1}{m} \sum_{i=1}^m \frac{u}{2} Z_{t,i}Z_{t,i} \tilde{H}_{t,i} Z_{t,i}, \\
& \xi_{u}'(t) = \frac{1}{m} \sum_{i=1}^m \frac{u}{2} Z_{t,i}Z_{t,i} \tilde{H}_{t,i}' Z_{t,i},  \hat{\xi}_u(t) = \xi_u(t) - \xi_u'(t), \\
& \hat{Y}_t = Y_t - Y_t', \quad \bar{H}_t = \int_{0}^1 \nabla^2 f(a x_t + (1 -a)x_t')  da.
\end{align*}
}
\end{restatable}
Our goal is to show that the dominating term in the evolution of the difference dynamics comes from the $W_p$ term involving the additional perturbation. To this end, we need to bound the remaining terms, $W_{g_0}, W_H, W_u$. A key technical challenge is to find a precise concentration bound for the $W_{g_0}(t+1)$ term, where
{\small
\begin{align*}
& \ -W_{g_0}(t+1) \\
= &  \ \eta\! \sum_{\tau = 0}^t (I \!-\! \eta H)^{t - \tau}\!\! \left(\!\frac{\sum_{i=1}^m (Z_{\tau,i} Z_{\tau,i}^\top\! -\! I) (\nabla f(x_{\tau})\! -\! \nabla f(x_{\tau}'))\!}{m}\!\right)\!.  \end{align*}
}

For the simplicity of discussion, we assume for the time being that $m = 1$, and drop the $i$ index in the subscript of $Z_{\tau,i}$.
Since $\bbE[Z_{\tau} Z_{\tau}^\top] = I$, heuristically, assuming that $Z_{\tau} Z_{\tau}^\top - I$ satisfies ``nice'' concentration properties, utilizing the independence of the $Z_{\tau}$'s across time and the fact that $(I - \eta H) \preceq (1 + \eta \gamma) I$, we would like to show that with high probability,
\tiny
\begin{align}
    &\ \norm*{W_{g_0}(t)} \nonumber \\
    \lesssim & \
    \eta \sqrt{\sum_{\tau = 0}^{t-1} (1 + \eta \gamma)^{2(t- 1 -\tau)} \bbE\left[\norm*{\left(Z_{\tau}Z_{\tau}\! -\! I\right) (\nabla f(x_{\tau})\! -\! \nabla f(x_{\tau}'))}^2 \mid \gF_{\tau-1} \right]} \label{eq:key_claim_W_g_0}
\end{align}
\normalsize
where  $\gF_{\tau-1}$ is a sigma-algebra containing all randomness up to and including iteration $\tau-1$, such that $x_{\tau}$ and $x_{\tau}'$ are both in $\gF_{\tau-1}$, but $Z_{\tau}$ is not. Then, assuming that \cref{eq:key_claim_W_g_0} holds, since
\small
\begin{align*}
&\ \bbE\left[\norm*{\left(Z_{\tau}Z_{\tau} - I\right) (\nabla f(x_{\tau}) - \nabla f(x_{\tau}'))}^2 \mid \gF_{\tau-1} \right]\\
= & \  O(d) \norm*{\nabla f(x_{\tau}) - \nabla f(x_{\tau}')}^2, 
\end{align*}
\normalsize
it follows that
\small
$$
    \norm*{W_{g_0}(t)}\! \leq\! \eta \sqrt{O(d)\sum_{\tau = 0}^{t-1} (1 \!+\! \eta \gamma)^{2(t\!-\! 1\! -\!\tau)}  \norm*{\nabla f(x_{\tau})\! -\! \nabla f(x_{\tau}')}^2  } 
$$
\normalsize
With this bound on $\norm*{W_{g_0}(t)}$, we eventually prove in \cref{proposition:saddle_point_func_decrease} in  \cref{appendix:saddle_point_decrease} that our algorithm escapes any $\epsilon-$saddle point with constant probability and that the $O(d)$ term appearing in the square root term above will eventually lead to an $O(d)$ dependence in the sample complexity\footnote{For general $1 \leq m \leq d$, there will also be an $O(1/m)$ dependence in the sample complexity.}.
We note that the $O(d)$ dimension dependence matches that of the best-known existing upper bound for finding first-order stationary points in smooth nonconvex zeroth-order optimization \citep{nesterov2017random}, and has been conjectured to be the best possible dimension dependence for general smooth nonconvex zeroth-order optimization \citep{balasubramanian2022zeroth}.
\paragraph{Key technical challenge} 
\label{paragraph:key_technical_challenge}
The key challenge in the above argument is to show that an equation in the form of \cref{eq:key_claim_W_g_0} could in fact hold. At first glance, that an inequality such as \cref{eq:key_claim_W_g_0} should hold is rather non-obvious --- this is because while the variable $(Z_{\tau} Z_{\tau} - I)(\nabla f(x_{\tau}) - \nabla f(x_{\tau}')) \mid \gF_{\tau-1}$ is mean-zero, it is subExponential rather than subGaussian. In fact, even in the subGaussian case, given a sequence of random vectors $\vx_{0},\dots, \vx_{t-1}$, such that each $\bbE[\vx_{\tau} \mid \gF_{\tau - 1}] = 0$, and that each $\vx_{\tau} \mid \gF_{\tau - 1}$ is norm-subGaussian with parameter $\sigma_{\tau} \in \gF_{\tau-1}$ (which is an appropriate generalization of subGaussianity for vectors, proposed in \cite{jin2019short}), proving a concentration inequality of the form 
$\norm*{\sum_{\tau = 0}^{t-1} \vx_{\tau}} \approx \tilde{O} \left(\sqrt{\sum_{\tau = 0}^{t-1} \sigma_{\tau}^2}\right)$
is a very delicate matter. In our case, the analogue of $\vx_{\tau}$ is $(I-\eta H)^{t-1-\tau}(Z_{\tau} Z_{\tau}\! -\! I)(\nabla f(x_{\tau})\! -\! \nabla f(x_{\tau}'))$, while the analogue of
$\sigma_{\tau}^2$ is $(1 + \eta \gamma)^{2(t- 1 -\tau)} \bbE\left[\!\norm*{\left(Z_{\tau}Z_{\tau} \!-\! I\right) (\nabla f(x_{\tau})\! -\! \nabla f(x_{\tau}'))}^2 \!\mid\! \gF_{\tau-1} \!\right]$. Existing techniques (cf. \cite{tropp2015introduction,jin2019short}) rely crucially on subGaussian properties that allow for each $\tau$ the moment-generating function $\bbE[e^{\theta \mY_{\tau}} \mid \gF_{\tau - 1}]$ to be defined for any fixed (and non-random) $\theta > 0$, where $\mY_{\tau}$ takes the form 
\small
$$     \mY_{\tau} = \begin{bmatrix}
    0 & \vx_{\tau}^\top \\
    \vx_{\tau} & 0
    \end{bmatrix},$$
    \normalsize
such that $\bbE[\mY_{\tau} \mid \gF_{\tau - 1}] = 0$ (since $\bbE[\vx_{\tau} \mid \gF_{\tau - 1}] = 0$), and the eigenvalues of $\mY_{\tau}$ are $\pm \norm*{\vx_{\tau}}$. In the case when $\vx_{\tau}$ is merely subExponential, the Moment Generating Function (MGF), $\bbE[e^{\theta \mY_{\tau}} \mid \gF_{\tau - 1}]$, will no longer be well-defined at any fixed (and non-random) $\theta > 0$. This poses a challenge in our setting, since $\vx_{\tau}$ takes the form 
$(I - \eta H)^{t-1-\tau} (Z_{\tau}Z_{\tau}^\top- I)(\nabla f(x_{\tau}) - \nabla f(x_{\tau}'))$, which is subExponential rather than subGaussian. While it may be possible to force $(I - \eta H)^{t-1-\tau} (Z_{\tau}Z_{\tau}^\top- I)(\nabla f(x_{\tau}) - \nabla f(x_{\tau}'))$ to be sub-Gaussian, say by normalizing $Z_{\tau}$ to have norm $\sqrt{d}$ (note any bounded random vector is also subGaussian), such that $\norm*{(Z_{\tau}Z_{\tau}^\top- I)g}^2 \leq d^2 \norm*{g}^2$ for any vector $g \in \bbR^d$, a careful examination of the argument in \cref{proposition:saddle_point_func_decrease} would show that this results in a $O(d^2)$ rather than $O(d)$ dependence in the sample complexity, incurring a heavy price on the overall sample complexity (extra factor of $d$) if $d$ is large. 

\paragraph{Our solution} To overcome the issue, we build on the following observation: with high probability, for any vector $g \in \bbR^d$, $\abs*{Z_{\tau}^\top g}$ is bounded within some log factor of $\norm*{g}$. On the event $\{\abs*{Z_{\tau}^\top g} = \tilde{O}(\norm*{g})\}$, the variable $$(Z_\tau Z_\tau^\top - I)g = Z_\tau (Z_\tau^\top g) - g \approx Z_\tau \norm*{g} - g$$
behaves approximately like a subGaussian random vector since $Z_\tau \sim N(0,I_d)$. 
Based on this intuition, after some careful analysis, we can show that $(Z_{\tau}Z_{\tau}^\top- I)(\nabla f(x_{\tau}) - \nabla f(x_{\tau}')) \mid \gF_{\tau - 1}$ is subGaussian on the event that $\abs*{Z_{\tau}^\top \nabla f(x_{\tau})}$ is bounded within some log factor of $\norm*{\nabla f(x_{\tau})}$, which happens with high probability. This then allows us to show that on this event, the corresponding MGF is well-defined for all fixed $\theta > 0$, enabling us to prove a concentration inequality of the form \cref{eq:key_claim_W_g_0}. This intuition is crystallized in the following proposition, which proves a more general bound than what we strictly need. For notational simplicity, we introduce the function $\operatorname{lr}(x) \coloneqq \log\left( x\log(x)\right)$.

\begin{restatable}{proposition}{nSGadaptedToZO}
\label{proposition:nSG_adapted_to_ZO}
Let $\gF_t,\,t\geq -1$ be a filtration. Let $(Z_t)_{t\geq 0}$ be a sequence of random vectors following the distribution $N(0,I)$ such that $Z_t\in \gF_t$ and is independent of $\gF_{t-1}$, and let $(v_t)_{t\geq 0}$ be a sequence of random vectors such that $v_t\in\gF_{t-1}$. For each $\tau\geq 0$, let
\small
\begin{align*}
    W_{\tau} = \sum_{t=0}^{\tau - 1} M_{t} (Z_t Z_t^\top - I) v_t, 
\end{align*}
\normalsize
where each $M_t$ is a deterministic matrix of appropriate dimension. Then, there exist some absolute constants $ c', C > 0$ such that for any $\tau \in \bbZ^+$ and $\delta \in(0, 1/e]$, the following statements hold:
\small
\begin{enumerate}
    \item For any $\theta>0$, with probability at least $1 - \delta$, we have
    \small
$$
    \norm*{W_{\tau}}\! \leq\! c' \theta \sum_{t=0}^{\tau -1} \norm*{M_t}_{2}^2 d (\operatorname{lr}(C \tau/\delta))^2 \norm*{v_t}^2\! +\! \frac{1}{\theta} \log(C d\tau/\delta).
$$
\normalsize
    \item For any $B > b > 0$, with probability at least $1 - \delta$, 
    \begin{align*}
        \mbox{either }\ 
        & \sum_{t=0}^{\tau - 1} \norm*{M_t}_{2}^2 d (\operatorname{lr}(C\tau/\delta))^2 \norm*{v_t}^2 \geq B, \mbox{ or} \\
        \norm*{W_{\tau}} \leq & \  \sqrt{\!\max\left\{\sum_{t=0}^{\tau - 1}\norm*{M_t}_{2}^2 d (\operatorname{lr}(C\tau/\delta))^2 \norm*{v_t}^2, b \right\}} \\
        & \  \times  c' \sqrt{\left(\log(C\tau d/\delta) + \log(\log(B/b)+1)\right)}
    \end{align*}
\end{enumerate}
\normalsize
Moreover, as is clear from the bounds above, we may pick $C \geq 1$ such that $\log\left(\frac{C}{\delta}\right) \geq 1, \forall \delta \in (0,\frac{1}{e}]$.
\end{restatable}



With this result, along with a series of other technical results in \cref{appendix:saddle_point_decrease}, we can show that the algorithm makes a function decrease of $F$ with $\Omega(1)$ probability near an $\ep$-saddle point (\cref{proposition:saddle_point_func_decrease} in \cref{appendix:saddle_point_decrease}). Armed with \cref{proposition:saddle_point_func_decrease}, as well as \cref{proposition:func_decrease_contradiction_informal}, the main result in \cref{theorem:convergence_main} then follows. The complete detailed analysis can be found in \cref{sec: escaping saddle appendix} (escaping saddle point) and \cref{appendix:main_result} (main result).

\section{Simulations}
\label{main:simulations}

We test the performance of our proposed algorithm with two-point estimators (ZOPGD-2pt) against existing zeroth-order benchmarks using the \emph{octopus function} (proposed in \cite{du2017gradient}).\footnote{Our code can be found at \url{https://github.com/rafflesintown/escape-saddle-points-2pt}} It is known that the octopus function defined on $\bbR^d$, which chains $d$ saddle points sequentially, takes exponential (in $d$) time for exact gradient descent to escape; it has thus emerged as a popular benchmark to evaluate algorithms that seek to escape saddle points. In our experiments, we compare the performance of our two-point estimator algorithm (ZOPGD-2pt) with PAGD (Algorithm 1 in \cite{vlatakis2019efficiently}) and ZO-GD-NCF (see \cite{zhang2022zeroth}), which are the only two existing zeroth-order algorithms that have (a) a $\tilde{O}(\nicefrac{d}{\ep^2})$ sample complexity for escaping saddle points (with the latter algorithm yielding the tightest bounds), and (b) performed the best empirically on escaping saddle points (see the simulation results in \cite{zhang2022zeroth}). Both PAGD and ZO-GD-NCF have to use $2d$ function evaluations per iteration to estimate the gradient while our algorithm only needs to use $2$ function evaluations. We plot the function value against the number of function evaluations. 

We tested the algorithms for $d = 10$ and $d = 30$. To account for the stochasticity in the algorithms, for each algorithm, we computed the average and standard deviation over 30 trials, and plotted the mean trajectory with an additional band that represents $1.5$ times the standard deviation. For our algorithm's
hyperparameters, we picked
\begin{align*}
    \eta = \frac{1}{4dL}, u = 10^{-2}, r = 0.05,  m = 1.
\end{align*}
Note $m = 1$ corresponds to using a two-point estimator.
For PAGD, we used the hyperparameters listed in their paper, and for ZO-GD-NCF, we used the code from their Neurips submission. 
For initialization, we chose a random $x_0$ near the saddle point at the origin, drawn from $N(0, 10^{-3} I_{d \times d})$

As we see in \cref{fig:comparison_octopus}, our algorithm reaches the global minimum of the octopus function in significantly fewer function evaluations than PAGD and ZO-GD-NCF (approximately 2.5 times faster than ZO-GD-NCF, and approximately 3 times faster than PAGD), despite our algorithm only using $2$ function evaluations per iteration compared to $2d$ function evaluations per iteration for both PAGD and ZO-GD-NCF. 
This suggests that in addition to our theoretical convergence guarantees, there can also be empirical benefits to using two-point estimators versus existing $2d$-point estimators in the zeroth-order escaping saddle point literature.

\begin{figure}[t]
\centering
\includegraphics[width=.95\linewidth]{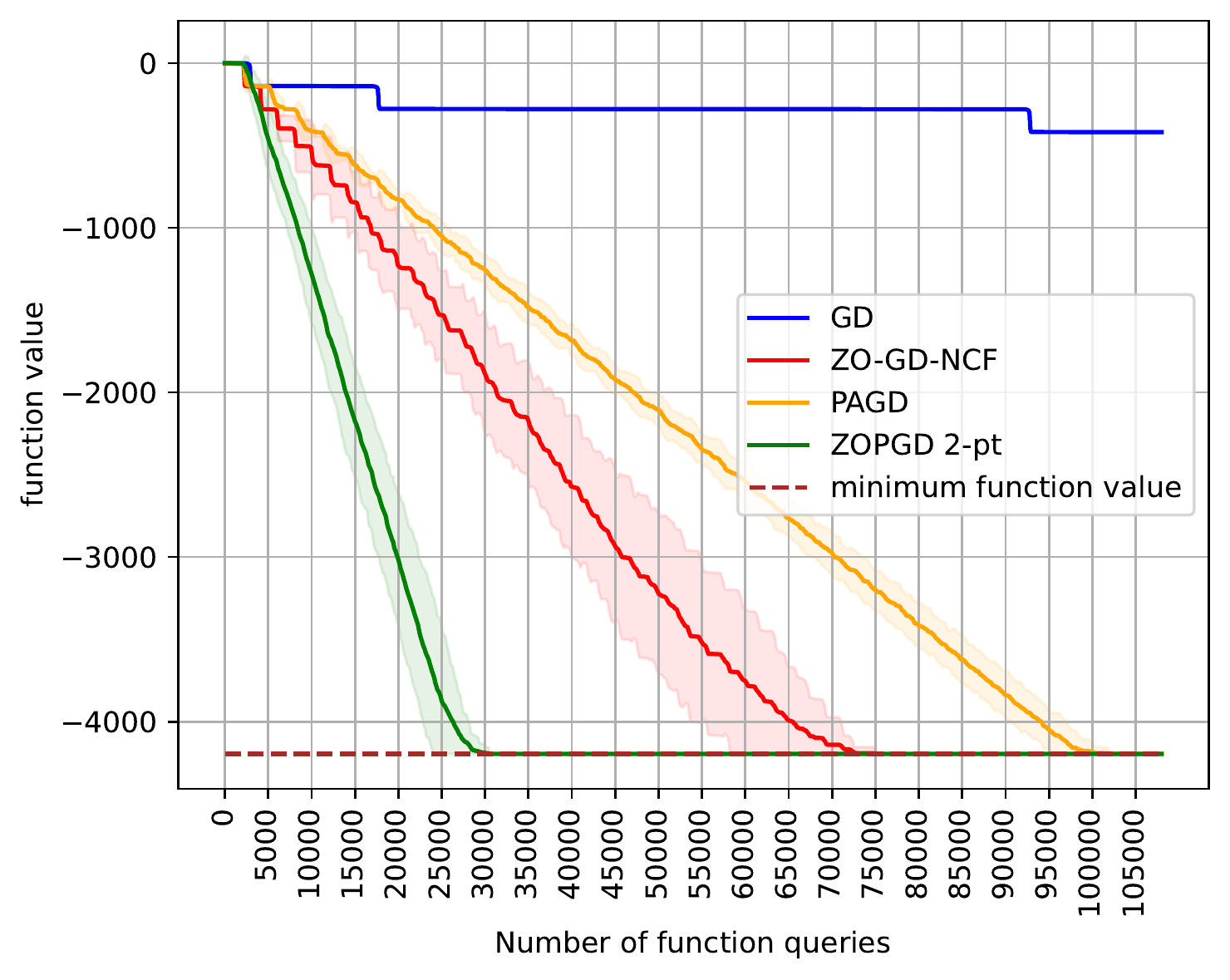}
    \label{fig:d=30_main}
\caption{Performance on toy octopus function, with $d = 30$}
\label{fig:comparison_octopus_main}
\end{figure}

\section{Conclusion}

In this paper, we proved that using two function evaluations per iteration suffices to escape saddle points and reach approximate second order stationary points efficiently in zeroth-order optimization. Along the way, we also gave the first analysis of high-probability function change using two (or more)-point zeroth-order gradient estimators, as well as a novel concentration bound for sums of subExponential (but not subGaussian) vectors which are each the products of Gaussian vectors. These technical contributions may be of independent interest to researchers working in zeroth-order optimization as well as general stochastic optimization. Finally, we provided numerical evidence supporting the theoretical convergence results.

\section{Acknowledgements}
This work is supported by NSF CAREER: ECCS-1553407, NSF CNS: 2003111, NSF AI institute: 2112085, and ONR YIP: N00014-19-1-2217.

\newpage
\bibliography{references}
\bibliographystyle{icml2023}

\newpage

\newpage

\onecolumn

\appendix

\section{Related Work}
\label{appendix:related_work}

\textbf{Two-point methods in zeroth-order optimization.}
Two-point (or in general $2m$-point, where $1\leq m< d$ with $d$ being the problem dimension) estimators, which approximate the gradient using two (or $2m$) function evaluations per iteration, have been widely studied by researchers in the zeroth-order optimization literature, in convex \citep{nesterov2017random,duchi2015optimal,shamir2017optimal}, nonconvex \citep{nesterov2017random}, online \citep{shamir2017optimal}, as well as distributed settings \citep{tang2019texttt}. A key reason for doing so is that for applications of zeroth-order optimization arising in robotics \citep{li2022stochastic}, wind farms \citep{tang2020zeroth}, power systems \citep{chen2020model}, online (time-varying) optimization \citep{shamir2017optimal}, learning-based control \citep{malik2019derivative,li2021distributed}, and improving adversarial robustness to black-box attacks in deep neural networks \citep{chen2017zoo}, it may be costly or impractical to wait for $\Omega(d)$ (where $d$ denotes the problem dimension) function evaluations per iteration to make a step. This is especially true for high-dimensional and/or time-varying problems. 
Indeed, for high-dimensional problems, two-point estimators can make swift progress even in the initial stage compared to $2d$-point estimator, and can reach a higher-quality solution if computation is limited \citep{tang2020distributed, chen2017zoo}. For instance, consider the work in \citep{chen2017zoo}, which studies the use of zeroth-order estimators to perform black-box attacks on deep neural networks, in order to identify (and then defend against) adversarial images that may lead to misclassification. In the paper, the authors employed two-point zeroth-order estimators, due to the high computational cost of using $2d$ function evaluations per iteration for hundreds of iterations (here $d$ is the dimension of an image, which in this case is over 20000). The authors showed empirically that their two-point estimators worked well; however there over no accompanying theoretical results.  

For online or time-varying environments, two-points estimators also often preferable. Since zeroth-order methods are often used in physical systems whose environment drifts or changes over time, this leads naturally to a \textit{time-varying} or \textit{online} optimization. For these problems, $2d$-point estimators will not produce a good estimation because the underlying function can drift to a very different problem while waiting for the $2d$ function evaluations. Indeed, the fewer function evaluations an optimization procedure needs, the faster it can catch up with the time-varying environment. In fact, for online optimization, it has been shown that two points estimator is optimal for convex Lipschitz functions~\citep{shamir2017optimal}. Thus, two-point estimators are a natural fit for time-varying online optimization problems.

\textbf{Saddle point escape with access to deterministic gradient.} While standard gradient descent can escape saddle points asymptotically \citep{lee2019first,panageas2019first}, it is known that standard gradient descent may take exponential time to escape saddle points \citep{du2017gradient}. Hence, when access to deterministic gradient is available, research has centered on escaping saddle points with adding perturbation \citep{jin2017escape}, momentum/acceleration based methods \citep{jin2018accelerated, sun2019heavy,staib2019escaping}, or gradient-based robust Hessian power/curvature exploitation methods \citep{zhang2021escape, adolphs2019local}. In addition, there has also been work on escaping saddle points devoted to specific optimization settings, such as constrained optimization \citep{mokhtari2018escaping, avdiukhin2019escaping}, optimization of weakly convex functions \citep{huang2021escaping}, bilevel optimization \citep{huang2022efficiently}, as well as on general manifolds \citep{sun2019escaping,criscitiello2019efficiently,han2020escape}. 

\textbf{Saddle point escape in stochastic gradient descent (SGD).} In practice, only stochastic gradient estimators are available in many problems. While SGD may converge to local maxima in worst-case scenarios \citep{ziyin2021sgd}, under assumptions such as bounded variance or subGaussian noise, there have been many works that have studied the problem of saddle point escape in SGD \citep{ge2015escaping, daneshmand2018escaping, xu2018first,jin2019nonconvex,vlaski2021second}. The best existing rate (without considering momentum/variance reduction techniques) appears to belong to that of \cite{fang2019sharp}, which converges to $\ep$-second order stationary points using $\tilde{O}(1/\ep^{3.5})$ stochastic gradients. While zeroth-order gradient estimators may also be viewed as stochastic gradients, they typically do not satisfy the bounded/subGaussian noise assumptions that are assumed in these works, making a direct comparison inappropriate. Escaping saddle point via momentum methods in SGD has also been studied \citep{wang2021escaping, antonakopoulos2022adagrad}; while we do not consider incorporating momentum in our works, this may be interesting future work. A number of papers has also considered the specialized setting of escaping saddle points in nonconvex finite-sum optimization \citep{reddi2018generic,liang2021escaping}, with many considering the case where variance-reduction is used  \citep{ge2019stabilized, li2019ssrgd}. While the finite-sum problem is quite different from our problem, the variance reduction approach considered in these works may be a relevant future direction. The saddle point escape problem has also been studied in other specific settings such as compressed optimization \citep{avdiukhin2021escaping},  distributed optimization \citep{vlaski2021distributed}, or in the overparameterization case \citep{roy2020escaping}.

\textbf{Saddle point escape with zeroth-order information.} The problem of escaping saddle points in zeroth-order optimization has been studied less often, and we have already listed all known works comparable to our work in the introduction \citep{bai2020escaping,vlatakis2019efficiently,balasubramanian2022zeroth}; a more detailed comparison of these works with our results has been provided in the discussion following the statement of our main result \cref{theorem:convergence_main}. We would like to mention that \cite{roy2020escaping} also includes a convergence result of $\tilde{O}\left(\frac{d^{1.5}}{\ep^{4.5}} \right)$ for the case with noisy function evaluations, which is incomparable to our existing work which focuses on the case with exact function evaluation. In addition, \cite{roy2020escaping} also makes a subGaussian assumption on the estimator noise, which zeroth-order estimators in our paper do not satisfy. Nonetheless, considering the extension to noisy function evaluations will make for important future work.

\textbf{Zeroth-order optimization.} Our work rests on a line of research in zeroth-order optimization which focuses on constructing gradient estimators using zeroth-order function values ~\citep{flaxman2004online,duchi2015optimal,nesterov2017random,shamir2017optimal,larson2019derivative}. As we have discussed, for smooth nonconvex functions, it is known that two-point zeroth-order estimators suffice to find first-order $\ep$-stationary points using $\tilde{O}(d/\ep^2)$ function evaluations \citep{nesterov2017random}. Our work studies the more complicated problem of reaching $\ep$-second order stationary points, attaining a rate of $\tilde{O}(d/\ep^{2.5})$. 

\section{Proof Roadmap}

We begin by introducing several key concentration inequalities in \cref{appendix: inequalities} which we will frequently use in our proofs. We then describe in detail (and prove) the sequence of results that lead up to \cref{proposition:func_decrease_contradiction} in \cref{appendix:function_decrease}, showing that there cannot be too many iterations with large gradients. Next, we describe the saddle point argument in detail, and prove \cref{proposition:saddle_point_func_decrease} in \cref{appendix:saddle_point_decrease}. Finally, we combine these results and prove our main result \cref{theorem:convergence_full} (whose informal version is \cref{theorem:convergence_main}) in \cref{appendix:main_result} 

Throughout our proofs, absolute constants, as denoted by e.g. $(c, c', C)$, may change from line to line. However, within the same proof, for clarity, we try to index different constants differently. We assume $d\geq 2$ and $m\leq d$.

\paragraph{Notations.} We shall denote the conditional expectation and conditional probability by $\bbE_{\gF}[\cdot]=\bbE[\cdot\mid \gF]$ and $\bbP_{\gF}(\cdot)=\bbP(\cdot\mid \gF)$ where $\gF$ is a sigma-algebra.

\section{Concentration inequalities}\label{appendix: inequalities}

This section serves to introduce several probabilistic results which will be useful for our main proofs in subsequent sections. We first introduce subGaussian, subExponential and norm-subGaussian random vectors in \cref{subsec:subGaussian_etc}. Next, in \cref{subsec:nSG_nSE_conce_bounds}, we provide concentration bounds for norm-subGaussian and subExponential random vectors. We then prove a novel concentration inequality involving products of subGaussian random vectors in \cref{subsec:novel_concentration_inequality}. We conclude by stating some concentration bounds for \cref{subsec:sub_weibull} random variables.

\subsection{subGaussian, subExponential and norm-subGaussian random vectors}
\label{subsec:subGaussian_etc}

We first define subGaussian and subExponential random vectors. A detailed reference for these concepts can be found in \cite{vershynin2018high}.

\begin{definition}[subGaussian and subExponential random vectors]
\label{definition:sG_sE_definition}
A random vector $\vx \in \bbR^d$ is
$\sigma$-subGaussian (SG($\sigma$)), if there exists $\sigma > 0$ such that for any unit vector $g \in \bbS^{d-1}$,
\begin{align*}
\bbE \left[\exp(\lambda \brac*{g, \vx - \bbE[\vx]}) \right] \leq \exp(\lambda^2 \sigma^2/2) \quad \forall \lambda \in \bbR.
\end{align*}
Meanwhile, a random vector $\vx \in \bbR^d$ is
$\sigma$-subExponential (SE($\sigma$)), if there exists $\sigma > 0$ such that for any unit vector $g \in \bbS^{d-1}$,
\begin{align*}
\bbE \left[\exp(\lambda \brac*{g, \vx - \bbE[\vx]}) \right] \leq \exp(\lambda^2 \sigma^2/2) \quad \forall \abs*{\lambda} \leq \frac{1}{\sigma} 
\end{align*}
\end{definition}

An alternative concentration property for random vectors revolving around its norm, known as norm-subGaussianity \citep{jin2019short}, is also relevant. 
\begin{definition}[norm-subGaussian random vectors]
A random vector $\vx \in \bbR^d$ is
$\sigma$-norm-subGaussian (nSG($\sigma$)), there exists $\sigma > 0$ such that
\begin{align*}
    \bbP(\norm*{\vx - \bbE \vx} \geq s) \leq 2e^{-\frac{s^2}{2 \sigma^2}} \quad \forall s \geq 0.
\end{align*}
\end{definition}

We recall the following result which provides several examples of nSG random vectors. In particular, it tells us a random vector $\vx \in \bbR^d$ that is $(\sigma/\sqrt{d})-$subGaussian is also $\sigma$-subGaussian. 
\begin{lemma}[Lemma 1 in \cite{jin2019short}]
There exists absolute constant $c$ such that the following random vectors are all nSG($c \sigma$).
\begin{enumerate}
    \item A bounded random vector $\vx \in \bbR^d$ so that $\norm*{\vx} \leq \sigma$.
    \item A random vector $\vx \in \bbR^d$, where $\vx = \xi \ve_1$ and the random variable $\xi \in \bbR$ is $\sigma$-subGaussian.
    \item A random vector $\vx \in \bbR^d$ that is $(\sigma/\sqrt{d})-$subGaussian
\end{enumerate}
\end{lemma}

In addition, if $\vx \in \bbR^d$ is zero-mean nSG($\sigma$), its component along a single direction is also subGaussian.

\begin{lemma}
\label{lemma:nSG_component_SG}
Suppose $\vx \in \bbR^d$ is zero-mean nSG($\sigma$). Then, for any fixed vector $\vv \in \bbR^{d}$, $\brac*{\vv, \vx}$ is zero-mean $\norm*{\vv}\sigma$-subGaussian.
\end{lemma}
\begin{proof}
Without loss of generality, we assume that $\vv \in \bbS^{d-1}$ is a unit vector. That $\brac*{\vv, \vx}$ is zero-mean follows directly from $\vx$ being zero-mean and $\vv$ being fixed. Meanwhile, since $\abs*{\brac*{\vv, \vx}} \leq \norm*{\vv} \norm*{\vx} = \norm*{\vx}$, for any $s \geq 0$, it follows that
\begin{align*}
    \bbP(\abs*{\brac*{\vv,\vx}} \geq s) \leq \bbP(\norm*{\vx} \geq s) \leq 2e^{-\frac{s^2}{2\sigma^2}},
\end{align*}
where the last inequality follows from the fact that $\vx$ is zero-mean and also nSG($\sigma$). Hence $\brac*{\vv,\vx}$ is zero-mean SG($\sigma$), as desired.
\end{proof}

\subsection{Concentration bounds for norm-subGaussian and subExponential random vectors}
\label{subsec:nSG_nSE_conce_bounds}

We begin by giving  some concentration bounds for norm-subGaussian random vectors. To do so, we introduce the following condition.

\begin{condition}
\label{condition:nSG}
Consider random vectors $\vx_1,\dots,\vx_n \in \bbR^d$, and corresponding filtrations $\gF_i$ generated by $(\vx_1,\dots,\vx_i)$. We assume $\vx_i \mid \gF_{i-1}$ is zero-mean, nSG($\sigma_i$), with $\sigma_i \in \gF_{i-1}$, i.e,
$$
\bbE \left[\vx_i\mid \gF_{i-1}\right]=0,
$$
and
$$
\bbP\left(
\|\vx_i\| \geq s\mid \mathcal{F}_{i-1}
\right)
\leq 2e^{-\frac{s^2}{2\sigma_i^2}} \quad \forall s \geq 0,
$$
where $\sigma_i$ is a measurable function of $(\vx_1,\ldots,\vx_{i-1})$ for each $i$.
\end{condition}

For norm subGaussian random vectors satisfying \cref{condition:nSG}, we first have the following bound.

\begin{lemma}
\label{lemma:nSG_brac}
Suppose $(\vx_1,\dots,\vx_n) \in \bbR^d$ satisfy \cref{condition:nSG}, i.e. each $\vx_i \mid \gF_{i-1}$ is mean-zero, nSG($\sigma_i$) with $\sigma_i \in \gF_{i-1}$. Let $\{\vu_i\}$ denote a sequence of random vectors such that $\vu_i \in \gF_{i-1}$ for every $i \in [n]$. Then, there exists an absolute constant $c$, such that for any $\delta \in(0,1)$ and $\lambda > 0$, with probability at least $1 - \delta$, 
\begin{align*}
\sum_{i=1}^n \brac*{\vu_i,\vx_i} \leq c \lambda \sum_{i=1}^n \norm*{\vu_i}^2 \sigma_i^2 + \frac{1}{\lambda} \log(1/\delta).    
\end{align*}
\end{lemma}
\begin{proof}
We note that if $\vx_i$ is mean-zero and nSG($\sigma_i$), then by \cref{lemma:nSG_component_SG}, $\brac*{\vu_i,\vx_i} \mid \gF_{i-1}$ is 
zero-mean and $\norm*{\vu_i} \sigma_i$-subGaussian. The rest of the proof follows from the proof of Lemma 39 in \cite{jin2019nonconvex} (key idea is exponentiate and then apply Markov's inequality). For completeness, we restate the proof here. Observe that for any $i$, since $\brac*{\vu_i,\vx_i}$ is $\norm*{\vu_i} \sigma_i$-subGaussian, for any $\lambda > 0$, we have that 
\begin{align*}
    \bbE\left[\exp(\lambda \brac*{\vu_i, \vx_i}) \mid \gF_{i-1} \right] \leq \exp(\lambda^2 \norm*{\vu_i}^2 \sigma_i^2/2) 
\end{align*}

For any $\lambda > 0$ and $s \geq 0$, observe that
\begin{align*}
    &\bbP\left(\sum_{i=1}^n \lambda \brac*{\vu_i,\vx_i} - \lambda^2 \norm*{\vu_i}^2 \sigma_i^2/2 \geq s\right) \\
    &= \bbP\left(\exp\left(\lambda \sum_{i=1}^n \brac*{\vu_i,\vx_i} - \lambda^2 \norm*{\vu_i}^2 \sigma_i^2/2\right) \geq \exp(\lambda s)\right) \\
    &\leq \bbE\left[\exp\left(\lambda \sum_{i=1}^n \brac*{\vu_i,\vx_i} - \lambda^2 \norm*{\vu_i}^2 \sigma_i^2/2\right)\right] \exp(-\lambda s) \\
    &= \bbE \left[\bbE\left[\exp\left(\lambda \sum_{i=1}^n \brac*{\vu_i,\vx_i} - \lambda^2 \norm*{\vu_i}^2 \sigma_i^2/2\right)\big| \gF_{n-1} \right] \right] \exp(-\lambda s) \\
    &= \bbE \left[\exp\left(\lambda \sum_{i=1}^{n-1} \brac*{\vu_i,\vx_i} - \lambda^2 \norm*{\vu_i}^2 \sigma_i^2/2\right)\bbE\left[\exp\left( \lambda \brac*{\vu_n,\vx_n} -\lambda^2 \norm*{\vu_n}^2 \sigma_n^2/2 \right)\big| \gF_{n-1} \right] \right] \exp(-\lambda s) \\
    &\labelrel\leq{eq:nSG_inner_prod_using_SG_MGF_property} \bbE \left[\exp\left(\lambda \sum_{i=1}^{n-1} \brac*{\vu_i,\vx_i} - \lambda^2 \norm*{\vu_i}^2 \sigma_i^2/2\right) \right] \exp(-\lambda s) \leq \dots \leq \exp(-\lambda s)
\end{align*}
Above, (\ref{eq:nSG_inner_prod_using_SG_MGF_property}) follows from the fact that $\brac*{\vu_i,\vx_i}\mid\gF_{i-1}$ is zero-mean and $\norm*{\vu_i} \sigma_i$-subGaussian for each $i \in [n]$. The final result then follows by picking $c = \frac{1}{2}$ and $s = \log(1/\delta)$.
\end{proof}

Assuming \cref{condition:nSG}, the following concentration result also holds for a sequence of nSG random vectors. 

\begin{lemma}[Lemma 6, Corollary 7 and Corollary 8 in \cite{jin2019short} combined]
\label{lemma:nSG_conc_with_corollaries}
Suppose $(\vx_1,\dots,\vx_n) \in \bbR^d$ satisfy \cref{condition:nSG}. Then, there exists an absolute constant $c$ such that for any fixed $\delta \in (0,1)$, $\theta > 0$, with probability at least $1 - \delta$,
\begin{align*}
    \norm*{\sum_{i=1}^n \vx_i} \leq c \theta \sum_{i=1}^n \sigma_i^2 + \frac{1}{\theta} \log(2d/\delta).
\end{align*}
Moreover, there are two corollaries.
\begin{enumerate}
    \item (Corollary 7 in \cite{jin2019short}) When $\{\sigma_i\}$ is deterministic, there exists an absolute constant $c$ such that for any fixed $\delta \in(0,1)$, with probability at least $1 - \delta$.
\begin{align*}
    \norm*{\sum_{i=1}^n \vx_i} \leq c \sqrt{\log(2d/\delta)\sum_{i=1}^n \sigma_i^2 }
\end{align*}
\item (Corollary 8 in \cite{jin2019short}) Suppose that the $\{\sigma_i\}$ sequence is random. Then, there exists an absolute constant $c$ such that for any fixed $\delta \in(0,1)$ and $B > b > 0$, with probability at least $1 - \delta$:
\begin{align*}
    \mbox{either } \sum_{i=1}^n \sigma_i^2 \geq B \quad \mbox{ or } \norm*{\sum_{i=1}^n \vx_i} \leq c \sqrt{\max\left\{\sum_{i=1}^n \sigma_i^2,b \right\} \cdot (\log(2d/\delta) + \log (\log (B/b)))}
\end{align*}
\end{enumerate}

\end{lemma}

We state here a Bernstein-type concentration inequality for sub-exponential random variables, which we also need.

\begin{lemma}[Bernstein concentration inequality]
\label{lemma:bernstein_exponential}
Consider a sequence of independently distributed $\sigma$-subexponential variables $\vx_1,\dots,\vx_n \in \bbR$, with mean $\bbE[\vx_i] \leq c' \sigma$ for some $c' > 0$ and each $i \in [n]$. Then, there exists an absolute constant $C > 0$, such that for any $\delta \in (0,1)$, with probability at least $1 - \delta$,
\begin{align}
\label{eq:bernstein_eq1}
\sum_{i=1}^n \vx_i \leq C \sigma(n + \log (1/\delta)).    
\end{align}
\end{lemma}
\begin{proof}
The result of \cref{eq:bernstein_eq1} follows by applying Bernstein's inequality to $\sum_{i=1}^n \vx_i - \bbE[\vx_i]$ (so each summand is mean-zero). Per Bernstein's inequality, (cf. Theorem 2.8.1 in \cite{vershynin2018high}), there exists an absolute constant $c > 0$ such that for any $s \geq 0$,
\begin{align*}
    \bbP\left(\sum_{i=1}^n (\vx_i - \bbE[\vx_i]) \geq s\right) \leq \exp\left(-c \min\left\{\frac{s^2}{n\sigma^2}, \frac{s}{\sigma} \right\} \right).
\end{align*}
Pick $s = \sigma\left(n + \frac{\log(1/\delta)}{c}\right)$. Then, 
\begin{align*}
\min\left\{\frac{s^2}{n\sigma^2}, \frac{s}{\sigma}\right\} = \min\left\{n + 2 \frac{\log(1/\delta)}{c} + \frac{(\log(1/\delta))^2}{c^2 n}, n + \frac{\log(1/\delta)}{c}\right\} = n + \frac{\log(1/\delta)}{c}.
\end{align*}
Continuing, we have that
\begin{align*}
    \bbP\left(\sum_{i=1}^n (\vx_i - \bbE[\vx_i]) \geq s\right) \leq \exp\left(-c \min\left\{\frac{s^2}{n\sigma^2}, \frac{s}{\sigma} \right\} \right) \leq \exp\left(-c\left(n + \frac{\log(1/\delta)}{c}\right) \right) \leq \delta.
\end{align*}
Thus, it follows that with probability at least $1 - \delta$,
\begin{align*}
    \sum_{i=1}^n (\vx_i - \bbE[\vx_i]) \leq \sigma \left(n + \frac{\log(1/\delta)}{c} \right) \implies \sum_{i=1}^n \vx_i \leq \sigma \left(n + \frac{\log(1/\delta)}{c} \right) + n c' \sigma,
\end{align*}
where implication holds since by assumption, $\bbE[\vx_i] \leq c' \sigma$ for some $c' > 0$. Then, by setting $C= \max\{1+c',1/c\}$, the desired result follows. 
\end{proof}

\subsection{A novel concentration inequality for the zeroth-order setting}
\label{subsec:novel_concentration_inequality}

In the zeroth-order setting, we will frequently have to bound the norms of terms of the form 
\begin{align}
\label{eq:W_tau_abstract_form}
    W_{\tau} = \sum_{t=0}^{\tau - 1} M_{t} (Z_t Z_t^\top - I) v_t, 
\end{align}
where $M_t$ is a known and fixed quantity, while $Z_t$ is random, and $v_t$ depends on $x_0$ and the history of previous $\{Z_j\}_{j=0}^{t-1}$'s, and is hence $\gF_{t-1}$-measurable. For our purposes, it suffices to consider $Z_t \sim N(0,I)$.

To see why such a bound will be useful, as mentioned in the main text and as we will see again later in the full proofs, in the analysis of escaping saddle points, we need to bound a term of the form 
\begin{align*}
    W_{g_0}(\tau) = \eta \sum_{t = 0}^{\tau-1} (I - \eta H)^{\tau-1 - t} (Z_tZ_t^\top - I) (\nabla f(x_t) - \nabla f(x_t')),
\end{align*}
where $H = \nabla^2 f(x_0)$ (assuming that $x_0$ is an $\ep$-saddle point), and $x_t$ and $x_t'$ are two coupled sequences.
Comparing with \cref{eq:W_tau_abstract_form}, we see that for the equation above, we can define $M_t= \eta (I - \eta H)^{\tau - 1 -t}$ (a fixed and known quantity) and $v_t = \nabla f(x_t) - \nabla f(x_t')$ (clearly, $\nabla f(x_t) - \nabla f(x_t')$ is $\gF_{t-1}$-measurable). This motivates why we wish to bound terms of the form \cref{eq:W_tau_abstract_form}.

Observe that each $(Z_t Z_t^\top - I) v_t \mid \gF_{t-1}$ term is subExponential rather than subGaussian. While it is possible to define norm-subExponential vectors in analogous way to norm-subGaussian vectors, the corresponding moment generating function (MGF) for subExponential random variables is not defined on the entirety of $\bbR$. When bounding a sum in the form of $\sum_{t=0}^{\tau-1} (Z_t Z_t^\top - I) v_t $, this creates a subtle but challenging technical issue. 

Following the intuition outlined in the main text, we bypass this difficulty by proving the following result. For notational simplicity, we introduce the function
\begin{equation}\label{eq:lr_def}
\operatorname{lr}(x) \coloneqq \log\left( x\log(x)\right).
\end{equation}

We now recall Proposition \ref{proposition:nSG_adapted_to_ZO} which we first introduced in the main text.
\nSGadaptedToZO*

\begin{proof}
We will focus on proving the first point, since the second follows as a natural corollary of our proof of the first part and the proof of Corollary 8 in \cite{jin2019short}. For simplicity, we shall assume $v_t\neq 0$ in the intermediate steps; extension to the general case is straightforward.

First of all, for $0\leq \alpha<1$, let
$$
g(\alpha;\delta) = \sqrt{\frac{2}{\pi}}\int_\alpha^{\sqrt{2\operatorname{lr}(1/\delta)}}(x^2-1)e^{-x^2/2}\,dx
= \sqrt{\frac{2}{\pi}}\!\left(\alpha e^{-\alpha^2/2}
-\frac{\delta \sqrt{2\operatorname{lr}(1/\delta)}}{\log(1/\delta)}\right).
$$
It's not hard to see that for a fixed $\delta\in(0,1/e]$, $g(\alpha;\delta)$ is continuous and strictly increasing over $\alpha\in[0,1)$. Then, since $\frac{\log x}{x}+1\leq x$ for $x\geq 1$, by plugging in $x=\log(1/\delta)$, we get
$$
\frac{\operatorname{lr}(1/\delta)}{(\log(1/\delta))^2}
=\frac{\log\log(1/\delta)+\log(1/\delta)}{(\log(1/\delta))^2}
=\frac{1}{\log(1/\delta)}\left(
\frac{\log\log(1/\delta)}{\log(1/\delta)}
+1
\right)\leq 1,
$$
which leads to
$$
g(2\delta;\delta)
=\sqrt{\frac{2}{\pi}}\left(
2\delta e^{-2\delta^2}
-\frac{\delta \sqrt{2\operatorname{lr}(1/\delta)}}{\log(1/\delta)}
\right)
\geq \sqrt{\frac{2}{\pi}}\left(
2e^{-2/e^2}\delta - \sqrt{2}\delta
\right)>0
$$
for $\delta\in(0,1/e]$. Furthermore, we obviously have $g(0;\delta)<0$. Therefore $g(\alpha;\delta)=0$ has a unique solution in $(0,2\delta)$, which we denote by $\alpha(\delta)$.\footnote{
By letting $W_0(x)$ denote the the principal branch of the Lambert $W$ function, it can be shown that
$$
\alpha(\delta) = \sqrt{-W_0\!\left(-\frac{2\delta^2\operatorname{lr}(1/\delta)}{(\log(1/\delta))^2}\right)}.
$$
}
These results imply that, for a random variable $Z$ following the standard normal distribution, we have
$$
\mathbb{E}\!\left[
(Z^2-1)\ind_{\alpha(\delta)\leq |Z|\leq \sqrt{2\operatorname{lr}(1/\delta)}}
\right]
=
\sqrt{\frac{2}{\pi}}\int_{\alpha(\delta)}^{\sqrt{2\operatorname{lr}(1/\delta)}}(x^2-1)e^{-x^2/2}\,dx
=
g(h(\delta);\delta)=0
$$
and
\begin{align*}
\bbP(\alpha(\delta)\leq|Z|\leq\sqrt{2\operatorname{lr}(1/\delta)})
\geq\ &
1-2\left(\frac{1}{\sqrt{2\pi}}\int_{\sqrt{2\operatorname{lr}(1/\delta)}}^{\infty}
e^{-x^2/2}\,dx
+
\frac{1}{\sqrt{2\pi}}\int_0^{\alpha(\delta)}e^{-x^2/2}\,dx
\right) \\
\geq\ &
1-2\left(\frac{1}{2}\exp\!\left(-\frac{2\operatorname{lr}(1/\delta)}{2}\right)
+\frac{\alpha(\delta)}{\sqrt{2\pi}}
\right)
=1-2\left(
\frac{\delta}{2\log(1/\delta)}
+\frac{\alpha(\delta)}{\sqrt{2\pi}}
\right) \\
\geq\ &
1-2\left(\frac{\delta}{2}+\frac{2}{\sqrt{2\pi}}\delta\right)
\geq 1-C\delta
\end{align*}
for any $\delta\in(0,1/e]$, where we define the absolute constant $C\coloneqq 2(1/2 +2/\sqrt{2\pi})$.

Now we let $A_t$ denote the event
$$
A_t = \left\{\alpha(\delta)\leq
\frac{\abs*{Z_t^\top v_t}}{\|v_t\|}
\leq\sqrt{2\operatorname{lr}(1/\delta)}
\right\}.
$$
Since $Z_t^\top v_t/\|v_t\|$ conditioned on $\gF_{t-1}$ follows the standard normal distribution, we have
\begin{align}
\bbP_{\gF_{t-1}}(A_t)\geq 1-C\delta,
\label{eq:Zt_vt_bounded_in_range_prob_bdd}
\end{align}
and
$$
\bbE_{\gF_{t-1}}\!\left[
v_t^\top\!\left(
Z_tZ_t^\top-I
\right)v_t\ind_{A_t}
\right]=0.
$$
Moreover, for any random vector $u\in\gF_{t-1}$ that is orthogonal to $v_t$, we have
$$
\bbE_{\gF_{t-1}}\!\left[
u^\top\!\left(Z_tZ_t^\top-I\right)v_t\ind_{A_t}
\right]
=\bbE_{\gF_{t-1}}\!\left[
u^\top Z_t\right]\cdot\bbE_{\gF_{t-1}}\!\left[Z_t^\top v_t\ind_{A_t}
\right]=0,
$$
where we used the fact that $Z_t^\top u$ is independent of $Z_t^\top v_t$ conditioned on $\gF_{t-1}$. Therefore
$$
\bbE_{\gF_{t-1}}\!\left[(Z_tZ_t^\top-I)v_t\ind_{A_t}\right]=0.
$$
Consider defining then the random variable $Q_t$ by
\begin{align*}
    Q_t \coloneqq
    (Z_t Z_t^\top - I) v_t\cdot\ind_{A_t}.
\end{align*}
We now show that $Q_t \mid  \gF_{t-1}$ is norm-subGaussian. Let $u\in\mathbb{R}^d$ with $\|u\|=1$ be arbitrary. We have
\begin{align*}
u^\top Q_t
=\ &
u^\top (Z_tZ_t^\top - I)v_t\cdot\mathbbm{1}_{A_t} \\
=\ &
u^\top\left(\frac{v_t v_t^\top}{\|v_t\|^2}
+ I - \frac{v_t v_t^\top}{\|v_t\|^2}
\right)(Z_tZ_t^\top - I)v_t\cdot\mathbbm{1}_{A_t} \\
=\ &
u^\top v_t
\left(\frac{|Z_t^\top v_t|^2}{\|v_t\|^2}-1\right)\cdot\mathbbm{1}_{A_t}
+ u^\top\!\left(I-\frac{v_tv_t^\top}{\|v_t\|^2}\right)(Z_tZ_t^\top-I)v_t\cdot\mathbbm{1}_{A_t} \\
=\ &
u^\top v_t
\left(\frac{|Z_t^\top v_t|^2}{\|v_t\|^2}-1\right)\cdot\mathbbm{1}_{A_t}
+ u_\perp^\top Z_tZ_t^\top v_t\cdot\mathbbm{1}_{A_t},
\end{align*}
where we denote $u_\perp=\left(I-\frac{v_tv_t^\top}{\|v_t\|^2}\right)u$. Since
\begin{align*}
\left|u^\top v_t
\left(\frac{|Z_t^\top v_t|^2}{\|v_t\|^2}-1\right)\cdot\mathbbm{1}_{A_t}
\right|
\leq\ &
|u^\top v_t|
(2\operatorname{lr}(1/\delta)-1),
\end{align*}
we see that $u^\top v_t
\left(\frac{|Z_t^\top v_t|^2}{\|v_t\|^2}-1\right)\cdot\mathbbm{1}_{A_t}$ conditioned on $\gF_{t-1}$ is $|u^\top v_t|
(2\operatorname{lr}(1/\delta)-1)$-subGaussian. Furthermore, since $|u_\perp^\top Z_tZ_t^\top v_t\cdot\mathbbm{1}_{A_t}| \leq |Z_t^\top u_\perp| \sqrt{2\operatorname{lr}(1/\delta)}\|v_t\|$, we have
\begin{align*}
\mathbb{P}_{\gF_{t-1}}
\!\left(
|u_\perp^\top Z_tZ_t^\top v_t\cdot\mathbbm{1}_{A_t}|
\geq s
\right)
\leq\ &
\mathbb{P}_{\gF_{t-1}}
\!\left(
|Z_t^\top u_\perp| \sqrt{2\operatorname{lr}(1/\delta)}\|v_t\|
\geq s
\right),
\end{align*}
and since $Z_t u_\perp/\|u_\perp\|\mid \gF_{t-1}$ follows the standard normal distribution, we see that $u_\perp^\top Z_tZ_t^\top v_t\cdot\mathbbm{1}_{A_t}$ is a $\sqrt{2\operatorname{lr}(1/\delta)}\|u_\perp\|\|v_t\|$-subGaussian variable. Note that $u^\top Q_t$ is just the sum of $u^\top v_t
\left(\frac{|Z_t^\top v_t|^2}{\|v_t\|^2}-1\right)\cdot\mathbbm{1}_{A_t}$ and $u_\perp^\top Z_tZ_t^\top v_t\cdot\mathbbm{1}_{A_t}$, we can conclude that $u^\top Q_t$ is subGaussian with parameter
\begin{align*}
& (2\operatorname{lr}(1/\delta)-1)|u^\top v_t|
+\sqrt{2\operatorname{lr}(1/\delta)}\|u_\perp\|\|v_t\| \\
\leq\ &
2\operatorname{lr}(1/\delta)(|u^\top v_t|
+\|u_\perp\|\|v_t\|)
\leq
2\sqrt{2}\operatorname{lr}(1/\delta)
\sqrt{|u^\top v_t|^2 + \|u_\perp\|^2\|v_t\|^2} \\
=\ &
2\sqrt{2}\operatorname{lr}(1/\delta)\|v_t\|,
\end{align*}
whenever $\delta\in(0,1/e]$. Consequently, by Lemma 1 in \cite{jin2019short}, we see that $Q_t\mid\gF_{t-1}$ is $8\operatorname{lr}(1/\delta)\sqrt{d}\|v_t\|$-norm-subGaussian.

It follows easily that $M_t Q_t \mid \gF_{t-1}$ is mean-zero and $8\operatorname{lr}(1/\delta)\norm*{M_t}_2 \norm*{v_t}\sqrt{d}$-norm-subGaussian. Hence, by Lemma 6 in \cite{jin2019nonconvex}, we know that there exists an absolute constant $c > 0$ such that for any $\theta > 0$ and $\delta > 0$, we have that with probability at least $1 - \delta$,
\begin{align*}
    \norm*{\sum_{t=0}^{\tau-1} M_t Q_t} \leq c \theta \sum_{t=0}^{\tau - 1} d(\operatorname{lr}(1/\delta))^2 \norm*{M_t}_2^2 \norm*{v_t}^2 + \frac{1}{\theta}  \log(2d/\delta). 
\end{align*}

Now, consider denoting the event 
\begin{align*}
    A \coloneqq
    \bigcup_{t=0}^{\tau-1}A_t
    =\left\{\abs*{Z_t^\top v_t} \in \left(\alpha(\delta) \norm*{v_t}, \sqrt{2\operatorname{lr}(1/\delta)}) \norm*{v_t}\right),\ \forall t \in \{0,\ldots,\tau-1\} \right\}
\end{align*}
By the union bound and \cref{eq:Zt_vt_bounded_in_range_prob_bdd}, we note that 
\begin{align*}
    \bbP(A) \geq 1 - \tau C \delta.
\end{align*}
Moreover, note that on the event $A$, $\sum_{t=0}^{\tau - 1} M_t Q_t = \sum_{t=0}^{\tau - 1} M_t (Z_t Z_t^\top - I) v_t$. Hence, 
\begin{align*}
    &\bbP\left(\norm*{\sum_{t=0}^{\tau-1} M_t (Z_t Z_t^\top - I)v_t } \leq c \theta \sum_{t=0}^{\tau - 1} d(\operatorname{lr}(1/\delta))^2 \norm*{M_t}_2^2 \norm*{v_t}^2 + \frac{1}{\theta}  \log(2d/\delta) \right) \\
    &\geq  \bbP\left(\norm*{\sum_{t=0}^{\tau-1} M_t Y_t} \leq c \theta \sum_{t=0}^{\tau - 1} d(\operatorname{lr}(1/\delta))^2 \norm*{M_t}_2^2 \norm*{v_t}^2 + \frac{1}{\theta}  \log(2d/\delta) , \mbox{ and } A \mbox{ happens }\right) \\
    &\geq 1 -  \left(\bbP\left(\norm*{\sum_{t=0}^{\tau-1} M_t Y_t} \geq c \theta \sum_{t=0}^{\tau - 1} d(\operatorname{lr}(1/\delta))^2 \norm*{M_t}_2^2 \norm*{v_t}^2 + \frac{1}{\theta}  \log(2d/\delta)\right) + \bbP(A^c)\right) \\
    &\geq 1 - (\delta + \tau C \delta).
\end{align*}
Now, by rescaling $\delta$ to $\delta/(C\tau + 1)$, we get the desired result. Note this $C$ is different from the $C$ in the statement of the lemma by an absolute multiplicative factor.
\end{proof}

\subsection{sub-Weibull random variables}
\label{subsec:sub_weibull}
In our work, we occasionally require bounding sums of heavy-tailed distribution, e.g. higher powers of $\norm*{Z}$ where $Z \sim N(0,I)$. To this end, we consider the following definition of sub-Weibull random variables. 
\begin{definition}
\label{definition:sub_weibull_defn}
We say that a random variable $X \in \bbR$ is \emph{sub-Weibull($K,\alpha$)} for some $K, \alpha > 0$, 
\begin{align*}
    \bbP(\abs{X} \geq s) \leq 2 \exp(- (s/K)^{1/\alpha}) \quad \forall s \geq 0.
\end{align*}
\end{definition}
For instance, the standard normal distribution is sub-Weibull($1, \frac{1}{2}$). From the way we define the tail parameter $\alpha$, the larger the $\alpha$, the heavier the tail of the distribution.


    



In our work, we need to show that the sum of sub-Weibull random variables is again sub-Weibull, which is ensured by the following result
\begin{lemma}
\label{lemma:sub-Weibull_closed}
Suppose $X$ and $Y$ are sub-Weibull($K_X, \alpha$) and sub-Weibull($K_Y,\alpha$) respectively. Then, $XY$ is sub-Weibull($C(K_X \cdot K_Y), 2\alpha$)  and $X+Y$ is sub-Weibull($C(K_X + K_Y), \alpha$) for some absolute constant $C > 0$.
\end{lemma}

A helpful result is the following, which bounds the sum of identically distributed sub-Weibull random variables. 

\begin{lemma}[Corollary 3.1 in \cite{vladimirova2020sub}]
\label{lemma:sub_weibull_concentration}
Suppose $X_1,\dots,X_n$ are identically distributed $(K',\alpha)$ sub-Weibull random variables. Then, for some absolute constant $c > 0$, for all $s \geq n c K'$, we have
\begin{align*}
    \bbP\left(\abs*{\sum_{i=1}^n X_i} \geq s \right) \leq \exp\left(-\left(\frac{s}{n c K'}\right)^{1/\alpha} \right) 
\end{align*}
\end{lemma}

In our work, we frequently need to bound sums of the $k$-th power of the norm of a standard $d$-dimensional Gaussian. We do so using \cref{lemma:sub_weibull_concentration}.

\begin{lemma}
\label{lemma:X_norm_kth_power_sum_bound}
Suppose $X_i \stackrel{i.i.d}{\sim} N(0,I_d)$ for $i \in [n]$. Then, for any $k \in \mathbb{Z}^+$, there exists absolute constants $c, C > 0$ such that for any $\delta \in (0,1)$, with probability at least $1 - \delta$,
\begin{align*}
    \abs*{\sum_{i=1}^n \norm*{X_i}^{2k}} \leq n C c^k d^k (1 +(\log(1/\delta))^k).
\end{align*}
In particular, for any $\delta\in(0, 1/e)$ such that $\log(1/\delta) \geq 1$, it follows that
\begin{align*}
    \abs*{\sum_{i=1}^n \norm*{X_i}^{2k}} \leq 2n C c^k d^k (\log(1/\delta))^k.
\end{align*}
\end{lemma}
\begin{proof}
First, observe that for any $j \in [d]$, $(X_i)_j^2$, being subExponential, is $(1,1)$-subWeibull. Then, by \cref{lemma:sub-Weibull_closed}, $\norm*{X_i}^2 = \sum_{j=1}^d (X_i)_j^2$ is $(c d, 1)$ for some absolute constant $c$. Now, it follows from definition of sub-Weibullness in \cref{definition:sub_weibull_defn} that
$\norm*{X_i}^{2k}$ is $(c^k d^k, k)$-subWeibull. Hence, applying \cref{lemma:sub_weibull_concentration}, we have that there exists absolute constant $C > 0$ such that for any $s \geq n C c^k d^k$,
\begin{align*}
    \bbP\left(\abs*{\sum_{i=1}^n \norm*{X_i}^{2k}} \geq s \right) \leq \exp\left(-\left(\frac{s}{n C c^k d^k}\right)^{1/k} \right) 
\end{align*}
Choosing $s = (1 +(\log(1/\delta))^k) n C c^k d^k$, we arrive then at the desired result.
\end{proof}

\subsection{Supermartingale concentration inequalities}

We first state and prove a supermartingale-type concentration inequality of the form we later require.
\begin{lemma}
\label{lemma:special_supermartingale_inequality}
Consider a filtration of sigma-algebras $\gF_0 \subset \gF_1 \subset \dots \subset \gF_{n-1} \subset \gF_n$ and a sequence of random variables $X_1,\dots,X_n$ such that $X_i\in\gF_i$. Suppose that 
\begin{align}
\label{eq:two_point_supermartingale_condition}
\bbP_{\gF_{i-1}}(X_i\leq a) = 1
\qquad\text{and}\qquad
\bbP_{\gF_{i-1}}(X_i\leq -b)
\geq p
\end{align}
for some $a , b  > 0$ and $0<p\leq \frac{1}{2}$. Then, for any $0 < \mu \leq b$ such that $\abs*{- b + \mu} \geq \frac{1-p}{p} \left(a + \mu\right)$, we have
\begin{align*}
    \bbP\left(\sum_{i=1}^n X_i \geq -n \mu + s\right) \leq \exp\left(-\frac{s^2}{4n (b- \mu)^2}\right),
    \qquad\forall s>0.
\end{align*}
\end{lemma}
\begin{proof}
Observe that by Markov's inequality, for any $\lambda > 0$,
\begin{align*}
    \bbP\left(\sum_{i=1}^n X_i \geq -n \mu + s\right) = \bbP\left(\exp\left(\lambda \sum_{i=1}^n (X_i + \mu)\right) \geq \exp(\lambda s)\right) \leq \frac{\bbE\left[\exp\left(\lambda \sum_{i=1}^n (X_i + \mu)\right)\right]}{\exp(\lambda s)}.
\end{align*}
Now, observe that 
\begin{align}
 \bbE\left[\exp\left(\lambda \sum_{i=1}^n (X_i + \mu)\right)\right] &= \bbE\left[\bbE_{\gF_{n-1}}\!\left[\exp\left(\lambda \sum_{i=1}^n (X_i + \mu)\right)\right]\right] \nonumber \\
 &= \bbE\left[\exp\left(\lambda \sum_{i=1}^{n-1} (X_i + \mu)\right) \bbE_{\gF_{n-1}}\!\left[ \exp(\lambda (X_n + \mu)) \right]\right]. \label{eq:supermartingale_markov_exponentiation_intermediate}
\end{align}
Let us now compute $\bbE_{\gF_{n-1}}\!\left[ \exp(\lambda (X_n + \mu))\right]$:
\begin{align*}
    & \bbE_{\gF_{n-1}}\!\left[ \exp(\lambda (X_n + \mu)) \right] \\
    =\ &
    \int_{(-\infty,-b]} \exp(\lambda(x+\mu))\,
    \bbP_{\gF_{n-1}}(X_n\in dx)
    +
    \int_{(-b,a]} \exp(\lambda(x+\mu))\,
    \bbP_{\gF_{n-1}}(X_n\in dx) \\
    \leq\ &
    \bbP_{\gF_{n-1}}(X_n\leq -b) \exp(\lambda(-b + \mu)) + 
    \bbP_{\gF_{n-1}}(-b<X_n\leq a) \exp(\lambda(a + \mu)) \\
    \leq\ &
    p \exp(\lambda(-b + \mu))
    + (1-p) \exp(\lambda(a + \mu)).
\end{align*}
Then observe that by our choice of $\mu$, $- b + \mu < 0$, and that $\abs*{-b + \mu} \geq (a + \mu) \frac{1-p}{p}$. Since we assumed $p \leq \frac{1}{2}$, this means that $\frac{1-p}{p} \geq 1$ and so for any $k\geq 1$,
\begin{align*}
    \abs*{-b + \mu} \geq (a + \mu) \frac{1-p}{p} \implies \abs*{-b + \mu} \geq (a + \mu) \left(\frac{1-p}{p}\right)^{1/k} \implies p \abs*{-b + \mu}^k \geq (1-p)(a + \mu)^k.
\end{align*}
Consequently, by Taylor expansion,
\begin{align*}
& p \exp(\lambda(-b + \mu)) + (1-p) \exp(\lambda(a + \mu)) \\
=\ & 1+\sum_{k=1}^\infty \frac{\lambda^k(p(-b+\mu)^k
+ (1-p) (a+\mu)^k)}{k!}
\leq
1+\sum_{k=1}^\infty
\frac{\lambda^k(p(-b+\mu)^k
+ p\,|{-b}+\mu|^k)}{k!} \\
=\ &
1 + \sum_{k=1}^\infty
\frac{\lambda^{2k}\cdot 2p\,|{-b}+\mu|^{2k}}{(2k)!}
\leq
1 + \sum_{k=1}^\infty
\frac{\lambda^{2k}|{-b}+\mu|^{2k}}{(k)!} \\
=\ &
\exp(\lambda^2(-b+\mu)^2),
\end{align*}
which leads to
$$
\bbE_{\gF_{n-1}}\!\left[ \exp(\lambda (X_n + \mu)) \right]
\leq \exp(\lambda^2(-b+\mu)^2).
$$
Now, continuing from \cref{eq:supermartingale_markov_exponentiation_intermediate}, we have that
\begin{align*}
    \bbE\left[\exp\left(\lambda \sum_{i=1}^n (X_i + \mu)\right)\right] &\leq \bbE\left[\exp\left(\lambda \sum_{i=1}^{n-1} (X_i + \mu)\right) \bbE_{\gF_{n-1}}\1\left[ \exp(\lambda (X_n + \mu)) \right]\right] \\
    &\leq \bbE\left[\exp\left(\lambda \sum_{i=1}^{n-1} (X_i + \mu)\right) \exp(\lambda^2(b-\mu)^2)\right] \\
    &\leq \dots \\
    &\leq \exp(n\lambda^2 (b-\mu)^2).
\end{align*}
Thus, for any $\lambda > 0$ and $s \geq 0$,
\begin{align*}
    \bbP\left(\sum_{i=1}^n X_i \geq -n\mu + s\right) &\leq \frac{\bbE\left[\exp(\lambda(\sum_{i=1}^n (X_i + \mu))) \right]}{\exp(\lambda s)} \\
    &\leq \exp(n\lambda^2 (b - \mu)^2 - \lambda s)
\end{align*}
By finding the minimizing $\lambda$, we find that
\begin{align*}
    \bbP\left(\sum_{i=1}^n X_i \geq -n\mu + s\right) \leq \exp\left(-\frac{s^2}{4n (b- \mu)^2}\right),
\end{align*}
which completes the proof.
\end{proof}

We will later require a weakened form of a supermartingale concentration inequality, as stated and proven below.

\begin{proposition}[Weakened supermartingale concentration inequality]
\label{proposition:weakened_supermartingale_concentration_inequality}
Consider a filtration of sigma-algebras $\gF_0 \subset \gF_1 \dots \subset \gF_{n}$ and a sequence of random variables $X_1,\dots,X_{n}$ such that $X_i\in\gF_i$. Consider for each $i \in \{1,\ldots,n\}$ a bad set $B_i$ where $\ind_{B_i} \in \gF_{i-1}$, and suppose
\begin{align*}
\bbP_{\gF_{i-1}}(X_i\ind_{B_i^c} \leq a)=1
\qquad\text{and}\qquad
\bbP_{\gF_{i-1}}(X_i \ind_{B_i^c} \leq -b)
\geq p
\end{align*}
for some $a, b > 0$ and $0 \leq p \leq 1/2$. Then, for any $0 < \mu \leq b$ such that $\abs*{-b + \mu} \geq \frac{1-p}{p}(a + \mu)$, we have
\begin{align*}
\bbP\left(\sum_{i=1}^n X_i \geq -n \mu + s\right) \leq \exp\left(-\frac{s^2}{4n (b- \mu)^2}\right) + \sum_{i=1}^n \bbP(X_i\in B_i),
\qquad\forall s>0.
\end{align*}
\end{proposition}
\begin{proof}
We define 
$ Q_i \coloneqq X_i\ind_{B_i^c}$. 
We can then apply \cref{lemma:special_supermartingale_inequality} and get
\begin{align*}
    \bbP\left(\sum_{i=1}^n Q_i \geq -n \mu + s\right) \leq \exp\left(-\frac{s^2}{4n (b- \mu)^2}\right).
\end{align*}
Since $\bbP\left(X_i \neq Q_i \mbox{ for some } i \in [n]\right) \leq \sum_i \bbP(X_i\in B_i)$, it follows that 
\begin{align*}
    \bbP\left(\sum_{i=1}^n X_i \geq -n \mu + s\right) \leq \exp\left(-\frac{s^2}{4n (b- \mu)^2}\right) + \sum_{i=1}^n \bbP(X_i\in B_i),
\end{align*}
which completes the proof.
\end{proof}

\section{Function decrease in large gradient regime}
\label{appendix:function_decrease}
In this section, we show that sufficient function decrease can be made across the iterations with large gradients. We first restate and prove the function decrease lemma (\cref{lemma:function_decrease_tighter_decomp}), first introduced in the main text. We then provide a detailed roadmap of our proof in the subsequent discussion following the proof of \cref{lemma:function_decrease_tighter_decomp}.

\funcDecreaseMain*

\begin{proof}

First, for each $t \in \{-1,\dots,\tau\}$, we define $\gF_{t}$ to be the sigma-algebra generated by $$x_0, \quad (\{Z_{0,i}\}_{i=1}^m, \dots, \{Z_{t,i}\}_{i=1}^m), \quad (Y_0,\dots,Y_{t}).$$ 
Note that $\gF_{-1}$ is the sigma-algebra generated only by $x_0$. 

By Taylor expansion, for any $x,y\in\mathbb{R}^{d}$, there exists $\alpha\in[0,1]$ such that
$
f(x+y) = f(x)
+\langle\nabla f(x),y\rangle
+\frac{1}{2}y^\top \nabla^2 f(x+\alpha y)\,y
$.
Therefore
$$
\frac{f(x_t+uZ_{t,i})
-f(x_t-uZ_{t,i})
}{2u}
=\langle \nabla f(x), Z_{t,i}\rangle
+\frac{u}{2}Z_{t,i}^\top \tilde{H}_{t,i}Z_{t,i}
$$
with
$$
\tilde{H}_{t,i}
=\frac{\nabla^2 f(x+\alpha_{i,+}u Z_{t,i})
-\nabla^2 f(x-\alpha_{i,-}u Z_{t,i})}{2}
$$
for some $\alpha_{i,\pm}\in[0,1]$, and
\begin{align}
    x_{t+1} &= x_t - \eta\left( \frac{1}{m}\sum_{i=1}^m \left(Z_{t,i}Z_{t,i}^\top \nabla f(x_t) + \frac{u}{2}Z_{t,i}Z_{t,i}^\top \tilde{H}_{t,i} Z_{t,i}\right) + Y_t\right) \label{eq:x_delta_expression}
\end{align}
By the $\rho$-Hessian Lipschitz property of $f$, it follows that $\norm*{\tilde{H}_{t,i}} \leq \rho u\norm*{Z_{t,i}}$

Observe that
\begin{align}
f(x_{t+1})
\labelrel\leq{eq:f_L_smoothness}\ &
f(x_t) + \brac*{x_{t+1}-x_t, \nabla f(x_t)} + \frac{L}{2} \norm{x_{t+1} -x_t}^2  \nonumber \\
\labelrel={eq:use_x_delta_expression}\ &
f(x_t) - \eta \frac{1}{m} \sum_{i=1}^m \left|Z_{t,i}^\top \nabla f(x_t)\right|^2 - \eta \frac{1}{m} \sum_{i=1}^m
\frac{u}{2} Z_{t,i}^\top\nabla f(x_t)\cdot Z_{t,i}^\top \tilde{H}_{t,i} Z_{t,i} -  \eta \brac*{\nabla f(x_t), Y_t} \nonumber \\
&
+\frac{L\eta^2}{2} \norm*{\frac{1}{m} \sum_{i=1}^m \left(Z_{t,i}Z_{t,i}^\top \nabla f(x_t) + \frac{u}{2}Z_{t,i}Z_{t,i}^\top \tilde{H}_{t,i} Z_{t,i}\right) + Y_t}^2 \nonumber \\
\labelrel\leq{eq:use_peter_paul_and_more}\ &
f(x_t) - \frac{\eta}{m} \sum_{i=1}^m \abs*{Z_{t,i}^\top \nabla f(x_t)}^2
+\frac{\eta}{m}\sum_{i=1}^m \!\left(\! \frac{\abs*{Z_{t,i}^\top \nabla f(x_t)}^2}{4}
\!+\!
\frac{u^2 \big|Z_{t,i}^\top \tilde{H}_{t,i} Z_{t,i}\big|^2}{4} \!\right)\!
-  \eta \brac*{\nabla  f(x_t), Y_t} \nonumber \\
&
+ \frac{L\eta^2}{2} \left(2 \norm*{\frac{1}{m} \sum_{i=1}^m Z_{t,i}Z_{t,i}^\top \nabla f(x_t)}^2 + u^2 \norm*{\frac{1}{m} \sum_{i=1}^mZ_{t,i}Z_{t,i}^\top \tilde{H}_{t,i} Z_{t,i}}^2 + 4 \norm*{Y_t}^2 \right) \nonumber \\
\labelrel\leq{eq:use_H_tilde_i_bound}\ &
f(x_t) - \frac{3\eta}{4m} \sum_{i=1}^m \abs*{Z_{t,i}^\top \nabla f(x_t)}^2  +  \frac{\eta u^2}{m}\sum_{i=1}^m \frac{u^2 \rho^2 \norm*{Z_{t,i}}^6}{4}  -  \eta \brac*{\nabla  f(x_t), Y_t} \nonumber \\
&
+ \frac{L\eta^2}{2} \left(2 \norm*{\frac{1}{m} \sum_{i=1}^m Z_{t,i}Z_{t,i}^\top \nabla f(x_t)}^2 
+ \frac{u^2}{m} \sum_{i=1}^m u^2 \rho^2 \norm*{Z_{t,i}}^8
+ 4 \norm*{Y_t}^2 \right) \nonumber \\
\leq\ &
f(x_t) - \frac{3\eta}{4m} \sum_{i=1}^m \abs*{Z_{t,i}^\top \nabla f(x_t)}^2  + \frac{\eta u^4 \rho^2}{4m} \sum_{i=1}^m \norm*{Z_{t,i}}^6 + \frac{L\eta^2 u^4 \rho^2}{2m} \sum_{i=1}^m \norm*{Z_{t,i}}^8 -  \eta \brac*{\nabla  f(x_t), Y_t}  \nonumber \\
    &\quad + \frac{L\eta^2}{2} \left(2 \norm*{\frac{1}{m} \sum_{i=1}^m Z_{t,i}Z_{t,i}^\top \nabla f(x_t)}^2 + 4 \norm*{Y_t}^2 \right) \label{eq:func_decrease_derivation_tbc}
\end{align}
Above, to derive (\ref{eq:f_L_smoothness}), we used the $L$-smoothness of $f$. To derive (\ref{eq:use_x_delta_expression}), we used the expression for $(x_{t+1} - x_t)$ shown in \cref{eq:x_delta_expression}. To derive (\ref{eq:use_peter_paul_and_more}), we used the fact that $ab \leq (a^2 + b^2)/2$ for any $a,b \in \bbR_{\geq 0}$, as well as two applications of the fact that $\norm*{a +b}^2 \leq 2(\norm*{a}^2 + \norm*{b}^2)$ for any two vectors $a,b \in \bbR^d$. To derive (\ref{eq:use_H_tilde_i_bound}), we used the fact that $\norm*{\tilde{H}_{t,i}} \leq \rho u \norm*{Z_{t,i}}$.

To continue from \cref{eq:func_decrease_derivation_tbc}, we first observe that we can rewrite 
\begin{align*}
    Z_{t,i}Z_{t,i}^\top \nabla f(x_t) = (Z_{t,i}Z_{t,i}^\top - I) \nabla f(x_t) + \nabla f(x_t),
\end{align*}
so that
\begin{align*}
   \norm*{\frac{1}{m} \sum_{i=1}^m Z_{t,i}Z_{t,i}^\top \nabla f(x_t)}^2 \leq 2 \norm*{\frac{1}{m} \sum_{i=1}^m (Z_{t,i}Z_{t,i}^\top -I) \nabla f(x_t)}^2 + 2\norm*{\nabla f(x_t)}^2.
\end{align*}
Observe that we can apply the bound in \cref{proposition:nSG_adapted_to_ZO} to $\norm*{\sum_{i=1}^m (Z_{t,i}Z_{t,i}^\top - I) \nabla f(x_t)}$, and since $Z_{t,i}$ is independent of $\gF_{t-1}$ for all $i$, we know there exist absolute constants $\mathfrak{c}_1>0, C_1 \geq 1$ such that for any $\delta \in(0, 1/e]$ and $\theta > 0$,  with probability at least $1 - \delta$ conditioned on $\gF_{t-1}$,
\begin{align}
\label{eq:sum_m_ZZt_grad_bdd_before_theta}
    \norm*{\sum_{i=1}^m (Z_{t,i}Z_{t,i}^\top - I) \nabla f(x_t)}
    &\leq \mathfrak{c}_1 \theta \sum_{i=1}^{m} d (\operatorname{lr}(C_1 m/\delta))^2 \norm*{\nabla f(x_t)}^2 + \frac{1}{\theta} \log(C_1 d m/\delta) \nonumber \\
    &= \mathfrak{c}_1 \theta m d (\operatorname{lr}(C_1 m/\delta))^2 \norm*{\nabla f(x_t)}^2 + \frac{1}{\theta} \log(C_1 d m/\delta).
\end{align}
Moreover, since $C_1 \geq 1$, $\log(C_1 dm/\delta)$ and $\operatorname{lr}(C_1 m/\delta)$ both are at least $1$ as long as $\delta\in(0,1/e]$. Observe that conditioned on $\gF_{t-1}$, $\nabla f(x_t)$ is fixed. Hence, we can pick
$$
\theta = \frac{1}{\sqrt{\mathfrak{c}_1 md\,\operatorname{lr}(C_1d m/\delta)} \norm*{\nabla f(x_t)}}
$$
which is $\gF_{t-1}$-measurable, and plug it into \cref{eq:sum_m_ZZt_grad_bdd_before_theta} to find that the probability conditioned on $\gF_{t-1}$ of the following event
\begin{align}
    \norm*{\sum_{i=1}^m (Z_{t,i}Z_{t,i}^\top - I) \nabla f(x_t)}
    \leq 2\sqrt{\mathfrak{c}_1} (\operatorname{lr}(C_1dm/\delta))^{3/2} \sqrt{md} \norm*{\nabla f(x_t)}
\end{align}
is at least $1-\delta$.
By taking the total expectation, it follows that the event has a total probability at least $1 - \delta$. Thus, with probability at least $1 - \delta$,
\begin{align}
    \norm*{\frac{1}{m} \sum_{i=1}^m Z_{t,i}Z_{t,i}^\top \nabla f(x_t)}^2 &\leq 2 \norm*{\frac{1}{m} \sum_{i=1}^m (Z_{t,i}Z_{t,i}^\top -I) \nabla f(x_t)}^2 + 2\norm*{\nabla f(x_t)}^2 \nonumber \\
    &\leq 4\mathfrak{c}_1 (\operatorname{lr}(C_1 dm/\delta))^{3} \frac{d}{m} \norm*{\nabla f(x_t)}^2 + 2 \norm*{\nabla f(x_t)}^2 \nonumber \\
    &\leq \mathfrak{c}_2 (\operatorname{lr}(C_1 dm/\delta))^3 \frac{d}{m} \norm*{\nabla f(x_t)}^2, \label{eq:batch_Z_Z_nabla_bdd}
\end{align}
where the last inequality comes from the fact that $\operatorname{lr}(C_1 dm/\delta) \geq 1$, our assumption at the outset of the appendix that $d \geq m$, and denoting $\mathfrak{c}_2 := 4\mathfrak{c}_1 + 2$.

Denote the event $\tilde{H}_{0,\tau}(\delta)$ as the event that
\begin{align}
\label{eq:func_decrease_near_final}
f(x_{\tau}) - f(x_0)
\leq\ &
- \sum_{t=0}^{\tau-1} \frac{3\eta}{4m} \sum_{i=1}^m \abs*{Z_{t,i}^\top \nabla f(x_t)}^2
+ L \eta^2 \frac{\mathfrak{c}_2 d(\operatorname{lr}(C_1 dm/\delta))^3}{m} \sum_{t=0}^{\tau -1} \norm*{\nabla f(x_t)}^2 \nonumber\\
&\quad+  \frac{\eta u^4 \rho^2}{4m} \sum_{t=0}^{\tau-1}\sum_{i=1}^m \norm*{Z_{t,i}}^6 + \frac{L\eta^2 u^4 \rho^2}{2m} \sum_{t=0}^{\tau-1}\sum_{i=1}^m \norm*{Z_{t,i}}^8  \nonumber \\
&
-  \eta \sum_{t=0}^{\tau-1}\brac*{\nabla  f(x_t), Y_t}   + 2 L\eta^2 \sum_{t=0}^{\tau-1} \norm*{Y_t}^2 
\end{align}
holds.

Now, continuing from \cref{eq:func_decrease_derivation_tbc}, and using the bound in \cref{eq:batch_Z_Z_nabla_bdd}, summing over the iterations from $t=0$ to $\tau -1$, we find using the union bound that $\bbP(\cap_{\tau = 1}^{\tau'} \tilde{H}_{0,\tau}(\delta)) \geq 1 - \tau' \delta$, $\bbP(\tilde{H}_{0,\tau}(\delta)) \geq 1 - \tau \delta.$

Now, by \cref{lemma:nSG_brac}, for any $\delta\in(0,1), \alpha> 0$, with probability at least $1 - \delta$, there exists an absolute constant $\mathfrak{c}_3 > 0$ such that
\begin{align}
\label{eq:bound_nabla_dot_Y}
-\eta \sum_{t=0}^{\tau-1}\brac*{\nabla f(x_t),Y_t} \leq \eta\left(\frac{1}{\alpha}\sum_{t=0}^{\tau-1} \norm*{\nabla f(x_t)}^2 + \mathfrak{c}_3 \alpha  r^2 \log(1/\delta)\right).
\end{align}
Meanwhile, since $Y_t \sim N(0,(r^2/d)I)$, $\norm*{Y_t}^2$ is sub-exponential with sub-exponential norm $cr^2$ for some absolute constant $c > 0$, and by Bernstein's inequality (\cref{lemma:bernstein_exponential}), there exists some absolute constant $\mathfrak{c}_4 > 0$ such that
\begin{align}
\label{eq:Y_sq_bernstein_bdd}
    \sum_{t=0}^{\tau-1} \norm*{Y_t}^2 &\leq \mathfrak{c}_4 r^2 (\tau + \log(1/\delta))
\end{align}
with probability at least $1 - \delta$. 

To bound $\sum_{t=0}^{\tau-1} \frac{1}{m} \sum_{i=1}^m  \norm*{Z_{t,i}}^{6}$ and $\sum_{t=0}^{\tau-1} \frac{1}{m} \sum_{i=1}^m  \norm*{Z_{t,i}}^{8}$, both sums of heavy tailed Gaussian moments, we use \cref{lemma:X_norm_kth_power_sum_bound}, which states that for any $k \in \bbZ^+$ and $\delta \in(0,1)$, with probability at least $1 - \delta$, 
\begin{align}
\label{eq:X_norm_sixth_power_sum_bound_in_action}
    \frac{1}{m}\sum_{t=0}^{\tau-1}\sum_{i=1}^m  \norm*{Z_{t,i}}^{2k} \leq \mathfrak{c}_5\tau (\mathfrak{c}_6)^k d^k (1 +(\log(1/\delta))^k)
\end{align}
for some absolute constants $\mathfrak{c}_5,\mathfrak{c}_6 > 0$.
As in the statement of the proof, using $\chi := \operatorname{lr}(C_1 dm/\delta)$ to ease the notation, denote the event that 
\begin{align*}
f(x_{\tau}) -f(x_0)
\leq\ &
-\frac{3\eta}{4} \sum_{t=0}^{\tau-1} \frac{1}{m} \sum_{i=1}^m \abs*{Z_{t,i}^\top \nabla f(x_t)}^2  + \left(\frac{\eta}{\alpha} + \frac{\mathfrak{c}_2 L\eta^2 \chi^3 d}{m}\right) \sum_{t=0}^{\tau-1} \norm*{\nabla f(x_t)}^2 \\
&
+ \frac{\tau\eta u^4 \rho^2}{2}\cdot \mathfrak{c}_5 \mathfrak{c}_6^3 d^3 \left(\log\frac{1}{\delta}\right)^3
+ \tau L \eta^2 u^4 \rho^2 \cdot \mathfrak{c}_5\mathfrak{c}_6^4 d^4 \left(\log\frac{1}{\delta}\right)^4 \\
&
+ \eta (\mathfrak{c}_3\alpha r^2 + 2\mathfrak{c}_4 \eta L r^2) \log\frac{1}{\delta}  + 2\mathfrak{c}_4 L\eta^2 \tau r^2 
\end{align*}
holds as $\mathcal{H}_{0,\tau}(\delta)$. 

Plugging \cref{eq:bound_nabla_dot_Y}, \cref{eq:Y_sq_bernstein_bdd}, and \cref{eq:X_norm_sixth_power_sum_bound_in_action} into \cref{eq:func_decrease_near_final}, by union bound, we see that
\begin{align*}
    \bbP(\cap_{\tau=1}^{\tau'} \mathcal{H}_{0,\tau}(\delta)) \geq 1 - (\tau' + 4\tau')\delta = 1 - 5\tau' \delta, \qquad \bbP(\mathcal{H}_{0,\tau}) \geq 1 - (\tau + 4) \delta.
\end{align*}

The final result then follows by rescaling $\delta$ to $\frac{\delta}{T}$ and denoting $c_1 \coloneqq \max\{\mathfrak{c}_2,\mathfrak{c}_3,2\mathfrak{c}_4,\mathfrak{c}_5\mathfrak{c}_6^3/2,\mathfrak{c}_5\mathfrak{c}_6^4\}$.
\end{proof}

\textbf{Outline of proof approach.} Similar to the first-order setting, our goal is to show that we can arrive at a contradiction $f(x_T) < \min_x f(x)$ when there is a large number of steps at which $\norm*{\nabla f(x_t)} \geq \ep$. Roughly speaking, as \cref{eq:func_decrease_first_lemma} shows, we need to prove a lower bound of the form
\begin{align}
\label{eq:Z_dot_nabla_sq_geq_nabla_sq_roughly}
\sum_{t=0}^{T-1} \frac{1}{m} \sum_{i=1}^m \norm*{Z_{t,i}^\top \nabla f(x_t)}^2 \geq \Omega \left(\frac{1}{\alpha} + \frac{c_1 L\eta \chi^3 d}{m} \right) \sum_{t=0}^{T-1} \norm*{\nabla f(x_t)}^2
\end{align}
for some $\alpha$ which is not too large (an example would be picking $\alpha$ such that it only scales logarithmically in the problem parameters). 
However, it is tricky to prove such a lower-bound in the zeroth-order setting. In particular, for small batch-sizes $m$, $\frac{1}{m} \sum_{i=1}^m \norm*{Z_{t,i}^\top \nabla f(x_t)}^2$ could be small even as $\norm*{\nabla f(x_t)}^2$ is large; this is because for each $i \in [m]$,  $Z_{t,i}$ could have a negligible component in the $\nabla f(x_t)$ direction. This necessitates a more careful analysis to prove a bound similar to \cref{eq:Z_dot_nabla_sq_geq_nabla_sq_roughly}. We do so using the following approach.
\begin{enumerate}
    \item Intuitively, whilst for each individual iteration $t$, $\frac{1}{m} \sum_{i=1}^m \norm*{Z_{t,i}^\top \nabla f(x_t)}^2$ could be small even as $\norm*{\nabla f(x_t)}^2$ is large, in a small number of (consecutive) iterations $\{t_0, \dots, t_0 + t_f\}$, with high probability, there will be at least one iteration $t$ within $\{t_0, \dots, t_0 + t_f - 1\}$, such that $\frac{1}{m} \sum_{i=1}^m \norm*{Z_{t,i}^\top \nabla f(x_t)}^2 = \Omega(\norm*{\nabla f(x_t)}^2)$. We formalize this intuition in \cref{lemma:some_t_Z_bdd_away_from_0}. Thus, we consider breaking the time-steps into chunks where each chunk has $t_f$ consecutive iterations.

    \item Consider any such interval $\{t_0,\dots,t_0 + t_f - 1\}$. There are two cases to consider.

\begin{enumerate}
    \item The first case is when the gradient throughout all $t_f$ iterations is large enough to dominate the perturbation terms. Intuitively, in this case, it is not hard to see that given appropriate parameter choices, the gradient will change little throughout the $t_f$ iterations. In fact, as we formalize in \cref{lemma:gradient_changes_little_if_noise_small}, for an appropriate choice of $t_f$ and $\eta$, we can show that 
    $$\frac{1}{2}\norm*{\nabla f(x_{t_0})} \leq \norm*{\nabla f(x_t)} \leq 2\norm*{\nabla f(x_{t_0})} \qquad \forall t \in \{t_0,\dots,t_0 + t_f - 1\}.$$
    As a result, combined with point 1, we see that
    $$\sum_{t=t_0}^{t_0 + t_f - 1} \frac{1}{m} \sum_{i=1}^m \norm*{Z_{t,i}^\top \nabla f(x_t)}^2 \geq \Omega(\norm*{\nabla f(x_{t_0})}^2).$$
    Thus, by choosing $\alpha$ and $\eta$ judiciously, for such intervals, it is possible to show that 
    \begin{align*}
        \sum_{t=t_0}^{t_0 + t_f - 1} \frac{1}{m} \sum_{i=1}^m \norm*{Z_{t,i}^\top \nabla f(x_t)}^2 \geq \Omega(\norm*{\nabla f(x_{t_0})}^2) \geq & \Omega \left(\frac{1}{\alpha} + \frac{c_1 L\eta \chi^3 d}{m} \right) \sum_{t=t_0}^{t_0 + t_f -1} \norm*{\nabla f(x_t)}^2 \\
        =& \Omega \left(\frac{1}{\alpha} + \frac{c_1 L\eta \chi^3 d}{m} \right) \Omega\left(t_f\norm*{\nabla f(x_{t_0})}^2 \right)
    \end{align*}
    Thus, in these intervals, it is possible to obtain function improvement on the order of $\eta \Omega(\norm*{\nabla f(x_{t_0})}^2)$.
    \item The remaining case is when the gradient is small and dominated by the perturbation terms in any one of the $t_f$ iterations. In this case, as we show in \cref{lemma:gradient_norm_small_if_noise_larger}, for each of the $t_f$ iterations, the gradient will be small and on the same scale as the perturbation terms. In turn, by choosing $r, u$ and $\eta$ appropriately, we can make the perturbation terms small. Thus, whilst these intervals may not contribute to function decrease, they also contribute little in the way of function increase.
\end{enumerate}
\item When there are at least $T/4$ iterations with large gradient (i.e. $\norm*{\nabla f(x_t)} \geq \ep$), assuming $t_f$ divides $T$, it follows that there are at least $T/(4t_f)$ intervals of length $t_f$ where one iteration in the interval contains a large gradient. By choosing $u, r$ and $\eta$ appropriately such they are dominated by $\ep$, it is possible to show that with high probability, such an interval cannot belong to the second case above, and must instead be from the first case. Since $\norm*{\nabla f(x_t)} \approx \norm*{\nabla f(x_{t_0})}$ for each $t \in \{t_0,\dots,t_0 + t_f -1\}$ in this case, and we know that one of the iterations has a gradient with size at least $\ep$, it follows that we make function decrease progress of at least $
\eta \Omega(\ep^2)$ for such intervals. By appropriately choosing $\eta, u$ and $r$ to limit the effects of the intervals of the second form, we can then show a contradiction of the form $f(x_T) < f^*$. We demonstrate this formally in \cref{proposition:func_decrease_contradiction}.
\end{enumerate}

We formalize our approach in the following series of  results. First, for analytical convenience, we prove the following result showing that for any $t$, the perturbation terms $\norm*{Y_t}$ and $\frac{1}{m}\sum_{i=1}^m \norm*{Z_{t,i}}^4$ are bounded with high probability.

\begin{lemma}
\label{lemma:Y_t_bdd_throughout}
There exists an absolute constant $c_3 > 0$ such that, for any $t\in\mathbb{N}$, the event
$$
\mathcal{G}_t(\delta)
\coloneqq
\left\{\norm*{Y_t}^2 \leq c_3^2 r^2 \left(1 + \frac{\log(T/\delta)}{d}\right)
\ \text{and}\ \ 
\frac{1}{m}\sum_{i=1}^m \norm*{Z_{t,i}}^4 \leq 2c_3 d^2\left(\log\frac{T}{\delta}\right)^2\right\}
$$
has probability at least $1 - 2\delta/T$ for any $\delta \in (0,1/e]$.
\end{lemma}
\begin{proof}
Noting that $Y_t \sim N(0,(r^2/d)I)$, by applying Bernstein's inequality (\cref{lemma:bernstein_exponential}), it can be shown that with probability at least $\delta/T$,
$$
\|Y_t\|^2\leq c_3^2 r^2\left(1+\frac{\log(T/\delta)}{d}\right),
$$
where $c_3>0$ is some absolute constant. Then by using \cref{lemma:X_norm_kth_power_sum_bound}, applying the union bound, and redefining the constant $c_3$, we complete the proof.
\end{proof}

Next, in \cref{lemma:some_t_Z_bdd_away_from_0}, we show that in a small number of iterations, with high probability,  there exists some iteration $t$ such that $\frac{1}{m} \sum_{i=1}^m \abs*{Z_{t,i}^\top \nabla f(x_t)}^2 \geq \frac{1}{2} \norm*{\nabla f(x_t)}^2$.
 
\begin{lemma}
\label{lemma:some_t_Z_bdd_away_from_0}
There exists an absolute constant $c_2 \geq 1$ such that, upon defining
$$
t_f(\delta)=\left\lceil \frac{c_2}{m}\log\frac{T}{\delta}\right\rceil,
\qquad\delta>0,
$$
and defining the event
$$
\mathcal{B}_{t_0}(\delta;k)\coloneqq
\bigcup_{t=t_0}^{t_0+k-1}
\left\{
\frac{1}{m} \sum_{i=1}^m \abs*{Z_{t,i}^\top \nabla f(x_t)}^2 \geq \frac{1}{2} \norm*{\nabla f(x_t)}^2
\right\},
$$
we have
$$
\bbP\left(\mathcal{B}_{t_0}(\delta;k)\right)\geq 1-\frac{\delta}{T}.
$$
for any $\delta\in(0,1)$, $t_0\in\mathbb{N}$ and $k\geq t_f(\delta)$.
\end{lemma}
\begin{proof}
Denote the event
$$
E_t = \left\{\frac{1}{m}\sum_{i=1}^m |Z_{t,i}^\top \nabla f(x_t)|^2
< \frac{1}{2}\|\nabla f(x_t)\|^2\right\}.
$$
Observe that, conditioned on $\gF_{t-1}$, the set of random variables $\left\{\|\nabla f(x_t)\|^2-|Z_{t,i}^\top\nabla f(x_t)|^2\right\}_{i=1}^m$ are independent, mean-zero, and subexponential with subexponential norm $\leq c\|\nabla f(x_t)\|^2$ for some absolute constant $c>0$. Hence
\begin{align*}
\bbP_{\gF_{t-1}}\!\left(E_t\right)
=\ 
& \bbP_{\gF_{t-1}}\!\left(
\frac{1}{m}\sum_{i=1}^m |Z_{t,i}^\top \nabla f(x_t)|^2
< \frac{1}{2}\|\nabla f(x_t)\|^2
\right) \\
=\ &
\bbP_{\gF_{t-1}}\!\left(
\sum_{i=1}^m
\left(
\|\nabla f(x_t)\|^2 - \left|Z_{t,i}^\top\nabla f(x_t)\right|^2
\right)
>\frac{m}{2}\|\nabla f(x_t)\|^2
\right) \\
\leq\ &
\exp\left(-c'm\right),
\end{align*}
where $c'$ is some positive absolute constant. Then, for any $t_0,k\in\mathbb{N}$,
\begin{align*}
& \bbP\!\left(
\frac{1}{m}\sum_{i=1}^m \left|Z_{t,i}^\top \nabla f(x_t)\right|^2
<\frac{1}{2}\|\nabla f(x_t)\|^2
\text{ for every }t\in[t_0,t_0+k)
\right) \\
=\ &
\bbE\!\left[
\prod_{t=t_0}^{t_0+k-1}
\ind_{E_t}
\right]
=\bbE\!\left[
\prod_{t=t_0}^{t_0+k-2}
\ind_{E_t}\cdot
\bbE_{\gF_{t_0+k-2}}
\!\left[\ind_{E_{t_0+k-1}}\right]\right] \\
\leq\ &
\exp(-c' m)\cdot 
\bbE\!\left[
\prod_{t=t_0}^{t_0+k-2}
\ind_{E_t}
\right]
\leq \cdots
\leq
\exp(-c' m k).
\end{align*}
Therefore, by letting $c_2 = \max\{1,1/c'\}$ and
$$
k \geq t_f(\delta) = \left\lceil \frac{c_2}{m}\log\frac{T}{\delta}\right\rceil,
$$
we get
$$
\bbP\!\left(
\frac{1}{m}\sum_{i=1}^m \left|Z_{t,i}^\top \nabla f(x_t)\right|^2
<\frac{1}{2}\|\nabla f(x_t)\|^2
\text{ for every }t\in[t_0,t_0+k)
\right)
\leq \frac{\delta}{T},
$$
which completes the proof.
\end{proof}

The term $t_f(\delta)$ will frequently appear in the proofs to come; in the sequel we denote
\begin{align}
\label{eq:tf_defn}
t_f(\delta)\coloneqq
\left\lceil \frac{c_2}{m} \log \frac{T}{\delta}\right\rceil,
\qquad \delta\in(0, 1/e],
\end{align}
where $c_2\geq 1$ is the absolute constant defined in \cref{lemma:some_t_Z_bdd_away_from_0}.

We next show that with high probability, the norm difference term $\norm*{\nabla f(x_{t+1}) - \nabla f(x_t)}$ can be bounded in terms of $\norm*{\nabla f(x_t)}$ and the perturbation terms $\norm*{\frac{u}{2m} \sum_{i=1}^m Z_{t,i}Z_{t,i}^\top \tilde{H}_{t,i}Z_{t,i}}$ as well as  $\norm*{Y_t}$. 

\begin{lemma}
\label{lemma:bound_nabla_diff_At}
Define
\begin{align}
\mathcal{A}_t(\delta)
\coloneqq
\left\{
\|\nabla f(x_{t+1})\!-\!\nabla f(x_t)\|
\leq \frac{\|\nabla f(x_t)\|}{8t_f(\delta)}
+ \eta L
\left(\norm*{\frac{u}{2m} \sum_{i=1}^m Z_{t,i}Z_{t,i}^\top \tilde{H}_{t,i}Z_{t,i}} + \norm*{Y_t}\right)
\!\right\}
\end{align}
where $t_f(\delta)$ is defined in \cref{eq:tf_defn}, and let $C_1\geq 1$ be the corresponding absolute constants defined in \cref{lemma:function_decrease_tighter_decomp}.
Then there exists an absolute constant $c_4>0$ such that, whenever $\eta$ satisfies
\begin{align}
\label{eq:eta_condition_gradient_changes_little_if_noise_small}
\eta L \frac{c_4 (\operatorname{lr}(C_1 dmT/\delta))^{3/2} \sqrt{d}}{\sqrt{m}} \leq \frac{1}{8 t_f(\delta)},
\end{align}
we have
$$
\bbP(\mathcal{A}_t(\delta))\geq 1-\frac{\delta}{T}
$$
for any $\delta\in(0,1/e]$ and $t\in\bbZ^+$.
\end{lemma}
\begin{proof}
Since $\nabla f$ is $L$-Lipschitz, following the zeroth-order update step, we see that 
\begin{align}
\norm*{\nabla f(x_{t+1}) - \nabla f(x_t)} 
\leq\ &
L \norm*{x_{t+1} - x_t} \\
=\ &
\eta L \norm*{\frac{1}{m} \sum_{i=1}^{m}Z_{t,i} Z_{t,i}^\top \nabla f(x_t) + \frac{u}{2m} \sum_{i=1}^m Z_{t,i}Z_{t,i}^\top \tilde{H}_{t,i}Z_{t,i} + Y_t}. \label{eq:grad_1_step_change_intermediate}
\end{align}
Now, it follows from \cref{eq:batch_Z_Z_nabla_bdd} (with a slight modification in the absolute constant terms since here the norm is not squared) that there exists some absolute constant $c_4 > 0$ such that for any $\delta\in(0, 1/e]$, we have that with probability at least $1 - \delta/T$, the event
\begin{align*}
    \norm*{\frac{1}{m} \sum_{i=1}^m Z_{t,i}Z_{t,i}^\top \nabla f(x_t)} \leq c_4 (\operatorname{lr}(C_1 dmT/\delta))^{3/2} \sqrt{\frac{d}{m}} \norm*{\nabla f(x_t)},
\end{align*}
Hence, continuing from \cref{eq:grad_1_step_change_intermediate}, it follows that with probability at least $1 - \delta/T$,
\begin{align*}
& \norm*{\nabla f(x_{t+1}) - \nabla f(x_t)} \\ \leq\ &
\eta L \left(c_4 (\operatorname{lr}(C_1 dmT/\delta))^{3/2} \sqrt{\frac{d}{m}} \norm*{\nabla f(x_t)} 
    + \norm*{\frac{u}{2m} \sum_{i=1}^m Z_{t,i}Z_{t,i}^\top \tilde{H}_{t,i}Z_{t,i}} + \norm*{Y_t} \right),
\end{align*}
and by plugging in the condition \cref{eq:eta_condition_gradient_changes_little_if_noise_small}, we see that the event
$$
\mathcal{A}_t(\delta)
=\left\{
\|\nabla f(x_{t+1})-\nabla f(x_t)\|
\leq \frac{\|\nabla f(x_t)\|}{8t_f(\delta)}
+ \eta L
\left(\norm*{\frac{u}{2m} \sum_{i=1}^m Z_{t,i}Z_{t,i}^\top \tilde{H}_{t,i}Z_{t,i}} + \norm*{Y_t}\right)
\right\}
$$
has probability at least $1-\delta/T$.
\end{proof}

We show now that if the norm of the gradient dominates the norm of the perturbation terms, and we choose the step-size $\eta$ sufficiently small, then in a small number of iterations, the norm of the gradient does not change very much. For notational simplicity, we denote the event
$$
\calE(t_1,t_2,\delta)
\coloneqq \bigcap_{t=t_1}^{t_1+t_2-1}\left\{
\norm*{\nabla f(x_t)}  >  8 t_f(\delta)\eta L \left( \frac{u}{2} \norm*{\frac{1}{m} \sum_{i=1}^m Z_{t,i} Z_{t,i}^\top \tilde{H}_{t,i} Z_{t,i} }+ \norm*{Y_t}\right)
\right\}.
$$
\begin{lemma}
\label{lemma:gradient_changes_little_if_noise_small}
Let $\delta\in(0,1/e]$ and $T\in\mathbb{Z}^+$ be such that $T>2t_f(\delta)+1$. Consider any positive integer $t_f' \leq 2 t_f(\delta)$, and any $t_0 \in \{0,\dots, T-1- t_f'\}$. Suppose $\eta$ satisfies the condition
\cref{eq:eta_condition_gradient_changes_little_if_noise_small}. Then, on the event
$$
\calE(t_0,t_f',\delta)\cap \left(\bigcap_{t=t_0}^{t_0+t_f'-1}\mathcal{A}_t(\delta)\right),
$$
we have
$$
\frac{1}{2}\|\nabla f(x_0)\|
\leq \|\nabla f(x_t)\|
\leq 2\|\nabla f(x_0)\|
$$
for all $t\in\{t_0,\ldots,t_0+t_f'-1\}$.
\end{lemma}
\begin{proof}
By plugging
$$
\norm*{\nabla f(x_t)}  >  8 t_f(\delta)\eta L \left( \frac{u}{2} \norm*{\frac{1}{m} \sum_{i=1}^m Z_{t,i} Z_{t,i}^\top \tilde{H}_{t,i} Z_{t,i} }+ \norm*{Y_t}\right)
$$
into the definition of $\mathcal{A}_t(\delta)$, we see that, on the event $\calE(t_0,t_f', \delta)\cap \left(\bigcap_{t=t_0}^{t_0+t_f'-1}\mathcal{A}_t(\delta)\right)$, we have
\begin{align*}
\|\nabla f(x_{t+1})-\nabla f(x_t)\|
\leq \frac{\|\nabla f(x_t)\|}{4t_f(\delta)},
\end{align*}
and consequently,
$$
\left(1-\frac{1}{4t_f(\delta)}\right)
\|\nabla f(x_t)\|
\leq \|\nabla f(x_{t+1})\|
\leq 
\left(1+\frac{1}{4t_f(\delta)}\right)
\|\nabla f(x_t)\|,
$$
which leads to
$$
\left(1-\frac{1}{4t_f(\delta)}\right)^{t-t_0}
\|\nabla f(x_0)\|
\leq \|\nabla f(x_{t})\|
\leq 
\left(1+\frac{1}{4t_f(\delta)}\right)^{t-t_0}
\|\nabla f(x_0)\|
$$
for all $t\in\{t_0,\ldots,t_0+t_f'\}$. Then, since $(1 + 1/(4x))^{2x} \leq 2$ and $(1 - 1/(4x))^{2x} \geq 1/2$ for any $x\geq 1$, noting that $t_f'\leq 2t_f(\delta)$, we get the desired result.
\end{proof}

Conversely, in the following result, we show that in a small number of consecutive iterations, if the gradient is smaller than the perturbation terms in any one of the iterations, then for each of the iterations in this range, the gradient will be small and be on the same scale as the size of the perturbation terms. 

\begin{lemma}
\label{lemma:gradient_norm_small_if_noise_larger}
Let $\delta\in(0,1/e]$ and $T\in\mathbb{Z}^+$ be such that $T>2t_f(\delta)+1$. Consider any positive integer $t_f' \leq 2 t_f(\delta)$, and any $t_0 \in \{0,\dots, T-1- t_f'\}$. Suppose $\eta$ satisfies the condition
\cref{eq:eta_condition_gradient_changes_little_if_noise_small}. Then, on the event
$$
\calE^c(t_0,t_f',\delta)\cap \left(\bigcap_{t=t_0}^{t_0+t_f'-1}\mathcal{A}_t(\delta)\right)\cap\left(\bigcap_{t=t_0}^{t_0+t_f'-1}\mathcal{G}_t(\delta)\right),
$$
we have
$$
\norm*{\nabla f(x_t)} \leq c_5 t_f(\delta)\eta L \left(u^2 d^2 \rho\left(\log\frac{T}{\delta}\right)^2 + \sqrt{ 1 + \frac{\log(T/\delta)}{d}}  r \right) \quad \forall t \in \{t_0,t_0+1,\dots,t_0 + t_f' -1\},
$$
where $c_5$ is some absolute constant.
\end{lemma}
\begin{proof}
Let $t'$ be the first iteration in $\{t_0,t_0+1,\dots, t_0 + t_f'-1\}$ such that 
\begin{align}
\norm*{\nabla f(x_{t'})}  &\leq 8 t_f(\delta)\eta L \left( \frac{u}{2} \norm*{\frac{1}{m} \sum_{i=1}^m Z_{t',i} Z_{t',i}^\top \tilde{H}_{t',i} Z_{t',i} }+ \norm*{Y_{t'}}\right). \label{eq:gradient_smaller_than_noise}
\end{align}
Since we are working on an event which is a subset of $\calE^c(t_0,t_f',\delta)$, $t'$ is well-defined. By $\|\tilde{H}_{t',i}\|\leq \rho u \|Z_{t',i}\|$, we see that
\begin{align*}
\norm*{\nabla f(x_{t'})}  \leq\ &
8 t_f(\delta)\eta L \left( \frac{u^2 \rho}{2m} \sum_{i=1}^m \norm*{Z_{t',i}}^4 + \norm*{Y_{t'}}\right) \\
\leq\ &
8 t_f(\delta)\eta L \left(c_3 u^2 d^2 \rho\left(\log\frac{T}{\delta}\right)^2 + c_3 \sqrt{ 1 + \frac{\log(T/\delta)}{d}}  r\right),
\end{align*}
where we used the definition of $\mathcal{G}_t(\delta)$.

Recall that $t'$ is the first time step such that \cref{eq:gradient_smaller_than_noise} holds. By deriving similarly as in the proof of \cref{lemma:gradient_changes_little_if_noise_small}, we can show that for any $j \in \{t_0,t_0+1,\dots,t'-1\}$,
$$
\norm*{\nabla f(x_{j})} \leq 2 \norm*{\nabla f(x_{t'})} \leq 16 t_f(\delta)\eta L c_3 \left(u^2 d^2 \rho\left(\log\frac{T}{\delta}\right)^2 + \sqrt{ 1 + \frac{\log(T/\delta)}{d}}  r \right) .
$$
Meanwhile, for iterations $t \in [t',t_0 + t_f')$, by using the definitions of $\mathcal{A}_t(\delta)$ and $\mathcal{G}_t(\delta)$, we have
\begin{align*}
\norm*{\nabla f(x_{t+1})}
\leq\ &
\left(1 + \frac{1}{8t_f(\delta)}\right) \norm*{\nabla f(x_{t})} + \eta L c_3 \left(u^2 d^2 \rho\left(\log\frac{T}{\delta}\right)^2 + \sqrt{ 1 + \frac{\log(T/\delta)}{d}}  r \right) \\
=\ &
\left(1 + \frac{1}{8t_f(\delta)}\right)^{t+1 - t'} \norm*{\nabla f(x_{t'})} \\
&
+ \sum_{i=0}^{t - t'}
\left(1 + \frac{1}{8t_f(\delta)}\right)^{t - t' - i} \eta L c_3\left(u^2 d^2 \rho\left(\log\frac{T}{\delta}\right)^2 + \sqrt{ 1 + \frac{\log(T/\delta)}{d}}  r \right) \\
\leq\ & 
\left(1 + \frac{1}{8t_f(\delta)}\right)^{t_f'} \norm*{\nabla f(x_{t'})} \\
&
+ 8t_f(\delta)\left(
\left(1 + \frac{1}{8t_f(\delta)}\right)^{t_f'}-1
\right) \eta L c_3 \left(u^2 d^2 \rho\left(\log\frac{T}{\delta}\right)^2 + \sqrt{ 1 + \frac{\log(T/\delta)}{d}}  r \right) \\
\leq\ &
e^{1/4} \cdot 8 t_f(\delta)\eta L c_3 \left(u^2 d^2 \rho\left(\log\frac{T}{\delta}\right)^2 + \sqrt{ 1 + \frac{\log(T/\delta)}{d}}  r \right) \\
&
+ 8t_f(\delta)(e^{1/4}-1)\cdot \eta L c_3 \left(u^2 d^2 \rho\left(\log\frac{T}{\delta}\right)^2 + \sqrt{ 1 + \frac{\log(T/\delta)}{d}}  r \right) \\
\leq\ &
16 t_f(\delta)\eta L c_3 \left(u^2 d^2 \rho\left(\log\frac{T}{\delta}\right)^2 + \sqrt{ 1 + \frac{\log(T/\delta)}{d}}  r \right),
\end{align*}
where we used $t_f'\leq 2t_f(\delta)$ and the fact that $(1-1/(8x))^{2x}\leq e^{1/4}$ for all $x>0$. By defining $c_5 \coloneqq 16c_3$, we complete the proof.
\end{proof}

We next derive a useful result showing that the function change $f(x_{\tau}) - f(x_0)$ can be decomposed into one component arising from intervals when the gradient dominates noise (which improves function value) and another component arising from intervals with small gradient which may add to function value but whose contributions are bounded in terms of $\eta, u$ and $r$. For now, we focus on the case $\tau \geq t_f(\delta)$, since it will be useful to us in proving that there cannot be more than $T/4$ iterations with large gradient.

\begin{lemma}[Function change for large $\tau$]
\label{lemma:func_decrease_tau_bound}
Let $c_1>0, c_4>0, c_5>0, C_1\geq 1$ be the absolute constants defined in the statements of the previous lemmas. Let $\delta \in(0,1/e]$, and let $\tau \geq  t_f(\delta)$) be arbitrary.  Consider splitting $\{0,1\dots,\tau - 1\}$ into $K\coloneqq \floor{\tau/t_f(\delta)}$ intervals:
\begin{align*}
J_k =\ &
\{k t_f(\delta),\dots,(k+1)t_f(\delta)-1\},\ \ 0\leq k < K-1, \\
J_{K-1} =\ &
\{(K - 1) t_f(\delta),\dots, \tau - 1\}.
\end{align*}
Let $I_1$ denote the set of indices $k$ such that for every time-step $t$ in the interval $J_k$, the gradient dominates the noise terms as
\begin{align}
   \norm*{\nabla f(x_t)}  >  8 t_f(\delta)\eta L \left( \frac{u}{2} \norm*{\frac{1}{m} \sum_{i=1}^m Z_{t,i} Z_{t,i}^\top \tilde{H}_{t,i} Z_{t,i} }+ \norm*{Y_t}\right).
\label{eq:gradient_dominates_noise}
\end{align}
Suppose we choose $\eta$ such that 
\begin{align}
\eta \leq 
\frac{1}{L t_f(\delta)}\cdot\min \left\{\frac{\sqrt{m}}{8c_4 (\operatorname{lr}(C_1 dmT/\delta))^{3/2} \sqrt{d} }, \frac{m}{128 c_1(\operatorname{lr}(C_1 dmT/\delta))^3 d}\right\}.\label{eq:eta_smaller_than_1/8t_f}    
\end{align}
Then, on the event
$$
\mathcal{E}_{\tau}(\delta) \coloneqq \mathcal{H}_\tau(\delta)\cap
\left(\bigcap_{t=0}^{\tau-1}\mathcal{A}_t(\delta)\right)\cap
\left(\bigcap_{t=0}^{\tau-1}\mathcal{G}_t(\delta)\right)\cap
\left(\bigcap_{k=0}^{K-2}
\mathcal{B}_{kt_f(\delta)}(\delta;t_f(\delta))\right)
\cap
\mathcal{B}_{(K-1)t_f(\delta)}(\delta;\tau\!-\!(K\!-\!1)t_f(\delta)),
$$
we have the following upper bound on function value change: 
\begin{align}
f(x_{\tau}) - f(x_0) \leq\ &
-\sum_{k \in I_1}\frac{\eta}{2} \min_{t \in J_k} \norm*{\nabla f(x_t)}^2 + \tau \frac{c_5^2}{64} \eta^3 t_f(\delta)^2 L^2 \left(u^2 d^2 \rho\left(\log\frac{T}{\delta}\right)^2 + \sqrt{2\log(T/\delta)}  r \right)^2 \nonumber \\
&
+ \tau \eta u^4 \rho^2 \cdot c_1d^3\left(\log\frac{T}{\delta}\right)^3 + \tau L \eta^2 u^4 \rho^2\cdot c_1d^4 \left(\log\frac{T}{\delta}\right)^4
\nonumber \\
&
+ \eta c_1r^2(128t_f(\delta) + \eta L )\log\frac{T}{\delta} + \tau c_1 L\eta^2  r^2.  \label{eq:func_max_increase}
\end{align}

Moreover, $\bbP(\calE_{\tau}(\delta)) \geq 1 - \frac{(5\tau + 4)\delta}{T}$.
\end{lemma}
\begin{proof}

Without loss of generality, we may assume that $\tau$ is a multiple of $t_f(\delta)$.\footnote{To accommodate the last interval which has length at most $2t_f(\delta)- 1$, we note that the results we require for the proof, namely \cref{lemma:some_t_Z_bdd_away_from_0}, \cref{lemma:gradient_changes_little_if_noise_small} and \cref{lemma:gradient_norm_small_if_noise_larger}, all hold for any interval length $t_f' \leq 2t_f(\delta)$.}
Then, any interval $J_k=\{t_0,\dots, t_0+t_f(\delta)-1\}$ belongs to one of the following two cases:
\begin{enumerate}[label=\textbf{Case \arabic*}),itemindent=42pt,labelwidth=38pt,labelsep=4pt,leftmargin=0pt]
    \item (Gradient dominates noise): Recall that this means that for every $t \in  J_k$, we have
    \begin{align*}
        \norm*{\nabla f(x_t)}  > 8 t_f(\delta)\eta L \left( \frac{u}{2} \norm*{\frac{1}{m} \sum_{i=1}^m Z_{t,i} Z_{t,i}^\top \tilde{H}_{t,i} Z_{t,i} }+ \norm*{Y_t}\right).
    \end{align*}
By our choice of $\eta$ in \cref{eq:eta_smaller_than_1/8t_f}, we can apply \cref{lemma:gradient_changes_little_if_noise_small} to get
$$
\min_{t \in J_k} \norm*{\nabla f(x_t)} \geq \frac{1}{4} \max_{t \in J_k} \norm*{\nabla f(x_t)}.
$$

We now consider the two cases when $J$ has fewer than $t_f(\delta)$ iterations and when $J = J_k$ f

Note also that on the event $\mathcal{B}_{kt_f(\delta)}(\delta;t_f(\delta))$, there exists some $t \in J_k$ such that
\begin{align*}
    \frac{1}{m} \sum_{i=1}^m \abs*{Z_{t,i}^\top \nabla f(x_t)}^2 \geq \frac{1}{2} \norm*{\nabla f(x_t)}^2.
\end{align*}
This implies then that
\begin{align}
\frac{1}{4} \sum_{t\in J_k}\frac{1}{m}\sum_{i=1}^m \abs*{Z_{t,i}^\top \nabla f(x_t)}^2
\geq\ &
\frac{1}{4} \min_{t \in J_k} \norm*{\nabla f(x_t)}^2
\geq \frac{1}{64} \max_{t\in J_k} \norm*{\nabla f(x_t)}^2 \nonumber \\
\geq\ &
\frac{1}{64 t_f(\delta)} \sum_{t\in J_k} \norm*{\nabla f(x_t)}^2. \label{eq:Z_dot_nabla_sq_lower_bdd_nabla_sq}
\end{align}
Thus by setting $\alpha = 128t_f(\delta)$ in \cref{eq:func_decrease_first_lemma} and by choosing $\eta$ such that 
\begin{align*}
    \frac{c_1 L\eta^2 \chi^3 d}{m} \leq \frac{\eta}{\alpha} = \frac{\eta}{128 t_f(\delta)}
    \ \ \iff\ \ 
    \eta \leq \frac{m}{128 c_1 L t_f(\delta)d \chi^3},
\end{align*}
it follows that 
\begin{align}
&-\frac{3\eta}{4} \sum_{t\in J_k} \frac{1}{m}\sum_{i=1}^m \abs*{Z_{t,i}^\top \nabla f(x_t)}^2 + \left(\frac{\eta}{128t_f(\delta)} + \frac{c_1 L\eta^2 \chi^3 d}{m}\right)\sum_{t\in J_k} \norm*{\nabla f(x_t)}^2
\nonumber \\
=\ &
-\frac{3\eta}{4} \sum_{t\in J_k}\frac{1}{m}\sum_{i=1}^m \abs*{Z_{t,i}^\top \nabla f(x_t)}^2
+ \frac{\eta}{64t_f(\delta)} \sum_{t\in J_k}\norm*{\nabla f(x_t)}^2 \nonumber \\
\leq\ &
-\frac{\eta}{2} \sum_{t\in J_k}\frac{1}{m}\sum_{i=1}^m \abs*{Z_{t,i}^\top \nabla f(x_t)}^2 \nonumber \\
\leq\ &
-\frac{\eta}{2} \min_{t \in J_k} \norm*{\nabla f(x_t)}^2 \label{eq:func_decrease_case1_bdd}
\end{align}

\item (Gradient does not dominate noise): there exists some $t \in J_k$ such that 
    \begin{align*}
        \norm*{\nabla f(x_t)}  \leq 8 t_f(\delta)\eta L \left( \frac{u}{2} \norm*{\frac{1}{m} \sum_{i=1}^m Z_{t,i} Z_{t,i}^\top \tilde{H}_{t,i} Z_{t,i} }+ \norm*{Y_t}\right).
    \end{align*}
By our choice of $\eta$ in \cref{eq:eta_smaller_than_1/8t_f}, we can apply \cref{lemma:gradient_norm_small_if_noise_larger} to get
\begin{align*}
    \norm*{\nabla f(x_t)} \leq c_5 t_f(\delta)\eta L \left(u^2 d^2 \rho\left(\log\frac{T}{\delta}\right)^2 + \sqrt{ 1 + \frac{\log(T/\delta)}{d}}  r \right) \quad \quad \forall t \in J_k.
\end{align*}
Hence, by setting $\alpha = 128t_f(\delta)$ in \cref{eq:func_decrease_first_lemma} and choosing $\eta$ such that 
\begin{align*}
    \frac{c_1 L \eta^2 \chi^3 d}{m} \leq \frac{\eta}{\alpha} 
    = \frac{\eta}{128 t_f(\delta)},
\end{align*}
it follows that
\begin{align}
&\left( \frac{\eta}{128t_f(\delta)} + \frac{c_1 L \eta^2 \chi^3 d}{m}  \right)\sum_{t\in J_k} \norm*{\nabla f(x_t)}^2
\nonumber \\
\leq\ &
\frac{\eta}{64t_f(\delta)}
\sum_{t\in J_k}
\left(c_5 t_f(\delta)\eta L \left(u^2 d^2 \rho\left(\log\frac{T}{\delta}\right)^2 + \sqrt{ 1 + \frac{\log(T/\delta)}{d}}  r \right)\right)^2 
\nonumber \\
\leq\ &
\frac{c_5^2}{64}t_f(\delta)^2\eta^3 L^2 \left(u^2 d^2\rho \left(\log\frac{T}{\delta}\right)^2 + \sqrt{ 1 + \frac{\log(T/\delta)}{d}}  r \right)^2 \label{eq:func_decrease_case2_bdd}
\end{align}
\end{enumerate}

Without loss of generality, we may assume that $\tau$ is a multiple of $t_f(\delta)$.\footnote{To accommodate the last interval which has length at most $2t_f(\delta)- 1$, we note that the results we require for the proof, namely \cref{lemma:some_t_Z_bdd_away_from_0}, \cref{lemma:gradient_changes_little_if_noise_small} and \cref{lemma:gradient_norm_small_if_noise_larger}, all hold for any interval length $t_f' \leq 2t_f(\delta)$.}
Then, any interval $J_k=\{t_0,\dots, t_0+t_f(\delta)-1\}$ belongs to one of the following two cases:

Having studied the two cases, we may now proceed to use them to complete the proof.
Let $I_1^c$ denote the complement of $I_1$ in $\{0,1,\ldots,K-1\}$. Then,
\begin{align}
& -\frac{3\eta}{4} \sum_{t=0}^{\tau-1} \frac{1}{m} \sum_{i=1}^m \abs*{Z_{t,i}^\top \nabla f(x_t)}^2  + \left(\frac{\eta}{\alpha} + \frac{c_1 L\eta^2 \chi^3 d}{m}\right) \sum_{t=0}^{\tau-1} \norm*{\nabla f(x_t)}^2 \nonumber \\
=\ &
\sum_{k \in I_1} \left(-\frac{3\eta}{4} \sum_{t=\in J_k}
\frac{1}{m} \sum_{i=1}^m \abs*{Z_{t,i}^\top \nabla f(x_t)}^2 
+ \left(\frac{\eta}{128t_f(\delta)} + \frac{c_1 L\eta^2 \chi^3 d}{m}\right) \sum_{t\in J_k} \norm*{\nabla f(x_t)}^2\right)
\nonumber \\
& +
\sum_{k \in I_1^c} \left(-\frac{3\eta}{4} \sum_{t\in J_k}
\frac{1}{m} \sum_{i=1}^m \abs*{Z_{t,i}^\top \nabla f(x_t)}^2
+ \left(\frac{\eta}{128t_f(\delta)} + \frac{c_1 L\eta^2 \chi^3 d}{m}\right) \sum_{t\in J_k}
\norm*{\nabla f(x_t)}^2\right) \nonumber \\
\leq\ &
-\sum_{k \in I_1}\frac{\eta}{2} \min_{t \in J_k} \norm*{\nabla f(x_t)}^2
+ \sum_{k \in I_1^c} t_f(\delta)\left(\frac{c_5^2}{64} t_f(\delta)^2 \eta^3 L^2 \left(u^2 d^2 \rho\left(\log\frac{T}{\delta}\right)^2 + \sqrt{ 1 + \frac{\log(T/\delta)}{d}}  r \right)^2  \right)
\nonumber \\
\leq\ &
-\sum_{k \in I_1} \frac{\eta}{2} \min_{t \in J_k} \norm*{\nabla f(x_t)}^2
+ \tau \frac{c_5^2}{64} t_f(\delta)^2 \eta^3 L^2 \left(u^2 d^2 \rho\left(\log\frac{T}{\delta}\right)^2 + \sqrt{ 1 + \frac{\log(T/\delta)}{d}}  r \right)^2. \label{eq:S1_bdd}
\end{align}
and so by \cref{eq:func_decrease_first_lemma},
\begin{align*}
f(x_{\tau}) - f(x_0) \leq\ &
-\sum_{k \in I_1}\frac{\eta}{2} \min_{t \in J_k} \norm*{\nabla f(x_t)}^2 + \tau \frac{c_5^2}{64}t_f(\delta)^2 \eta^3 L^2 \left(u^2 d^2 \rho\left(\log\frac{T}{\delta}\right)^2 \!+\! \sqrt{ 1 \!+\! \frac{\log(T/\delta)}{d}}  r \right)^2 \nonumber \\
& +
\tau \eta u^4 \rho^2\cdot c_1 d^3 \left(\log\frac{T}{\delta}\right)^3
+ \tau L \eta^2 u^4 \rho^2\cdot c_1 d^4 \left(\log\frac{T}{\delta}\right)^4 
\\
&
+ \eta  c_1r^2(\alpha + \eta L)\log\frac{T}{\delta} + \tau c_1 L\eta^2  r^2.
\end{align*}
Note that we choose $\alpha = 128t_f(\delta)$. In addition, observe that by our choice of $\delta$ (such that $\delta \leq \frac{1}{e}$), it follows that $\sqrt{1 + \frac{\log(T/\delta)}{d}} \leq \sqrt{2\log(T/\delta)}$. 

We can now complete our proof by using the union bound (suppressing the dependence of some of the events on $\delta$ for notational simplicity) to derive
\begin{align*}
\bbP(\mathcal{E}_{\tau}^c)
\leq\ &
\bbP(\mathcal{H}_\tau^c)
+\sum_{t=0}^{\tau-1}\bbP(\mathcal{A}_t^c)
+\sum_{t=0}^{\tau-1}\bbP(\mathcal{G}_t^c)
+\sum_{k=0}^{K-1}\bbP(\mathcal{B}^c_{kt_f(\delta)}(\delta;t_f(\delta))) \\
\leq\ &
\frac{(\tau+4)\delta}{T}
+\frac{\tau}{T}\delta+2\frac{\tau}{T}\delta
+\frac{K\delta}{T}
\leq \frac{(5\tau + 4)}{T}\delta.
\qedhere
\end{align*}
\end{proof}

We are now ready to show that if sufficiently many iterations have a large gradient, then with high probability, the function value of the last iterate $f(x_T)$, will be less than $\min_x f(x)$, a contradiction. Hence this limits the number of iterations that can have a large gradient.
\begin{proposition}
\label{proposition:func_decrease_contradiction}
Let $c_1>0, c_2\geq 1, c_4>0, c_5>0, C_1\geq 1$ be the absolute constants defined in the statements of the previous lemmas, and let $\delta\in(0,1/e]$ be arbitrary. Suppose we choose $u$, $r$, $\eta$ and $T$ such that
\begin{align*}
&u \leq
\frac{\sqrt{\ep}}{d \sqrt{\rho} \log(T/\delta)}
\cdot \min\left\{
\frac{1}{64c_5^2c_2},\frac{1}{2048 c_1c_2}
\right\}^{\! 1/4},
\qquad
r \leq \epsilon\cdot\min\left\{\frac{1}{8c_5\sqrt{2c_2}},
\frac{1}{32\sqrt{c_1}}\right\}, \\
&\eta \leq
\frac{1}{Lt_f(\delta)}\min \left\{\frac{1}{\log(T/\delta)}, \frac{\sqrt{m}}{8c_4 (\operatorname{lr}(C_1 dmT/\delta))^{3/2} \sqrt{d}}, \frac{m}{128 c_1 (\operatorname{lr}(C_1 dmT/\delta))^3 d} \right\}, \\
& T \geq
\max\left\{\frac{256 t_f(\delta)\left(\left(f(x_0) - f^*\right) + \ep^2/L)  \right)}{\eta \ep^2},4\right\}.
\end{align*}
Then, with probability at least $1 - 6\delta$, there are at most $T/4$ iterations for which $\norm*{\nabla f(x_t)} \geq \ep$.
\end{proposition}
\begin{proof}
Without loss of generality, we assume that $T$ is a multiple of $t_f(\delta)$, and we similarly split $\{0,1,\ldots,T\}$ into $K=\floor{T/t_f(\delta)}$ intervals $J_0,\ldots,J_{K-1}$. Let $I_1$ denote the set of indices $k$ such that for every $t\in J_k$,
\begin{align}
   \norm*{\nabla f(x_t)}  >  8 t_f(\delta)\eta L \left[\left( \frac{u}{2} \norm*{\frac{1}{m} \sum_{i=1}^m Z_{t,i} Z_{t,i}^\top \tilde{H}_{t,i} Z_{t,i} }\right)+ \norm*{Y_t}\right].
\end{align}
We let $I_1^c$ denote the complement of $I_1$ in $\{0,1,\ldots,K-1\}$. We denote
$$
\mathcal{E}_{T}(\delta)
\coloneqq \mathcal{H}_T(\delta)\cap
\left(\bigcap_{t=0}^{T-1}\mathcal{A}_t(\delta)\right)\cap
\left(\bigcap_{t=0}^{T-1}\mathcal{G}_t(\delta)\right)\cap
\left(\bigcap_{k=0}^{K-1}
\mathcal{B}_{kt_f(\delta)}(\delta;t_f(\delta))\right).
$$
In the remaining part of the proof, unless otherwise stated, we shall always assume that we are working on the event $\mathcal{E}_{T}(\delta)$.

By \cref{lemma:func_decrease_tau_bound} with $\tau=T$ and our choices of $\eta$ and $\delta$ in the statement of the lemma, we have
\begin{align}
f(x_{T}) - f(x_0) \leq\ &
-\sum_{k \in I_1}\frac{\eta}{2} \min_{t \in J_k} \norm*{\nabla f(x_t)}^2 +
T \frac{c_5^2}{64} t_f(\delta)^2 \eta^3 L^2 \left(u^2 d^2 \rho\left(\log\frac{T}{\delta}\right)^2 + \sqrt{2\log(T/\delta)}  r \right)^2 \nonumber \\
&
+ T \eta u^4 \rho^2 \cdot c_1d^3\left(\log\frac{T}{\delta}\right)^3 + T L \eta^2 u^4 \rho^2\cdot c_1d^4 \left(\log\frac{T}{\delta}\right)^4
\nonumber \\
&
+ \eta c_1r^2(128t_f(\delta) + \eta L )\log\frac{T}{\delta} + T c_1 L\eta^2  r^2.  \label{eq:func_decrease_T}
\end{align}
Suppose that there are at least $T/4$ iterations where $\norm*{\nabla f(x_t)} \geq \ep$. Let $I_\epsilon$ denote the set of indices $k$ for which there exists some $t \in J_k$ with $\norm*{\nabla f(x_t)} \geq \ep$. Then, by the pigeonhole principle, the set $I_\epsilon$ has at least $\lceil T/(4t_f(\delta))\rceil$ members. Note that, by our choices of the parameters $\eta,u,r$, it can be shown that
\begin{align}
\label{eq:nabla_f_norm_larger_than_ep_larger_than_noise}
c_5 t_f(\delta)\eta L \left(u^2 d^2 \rho\left(\log\frac{T}{\delta}\right)^2 + \sqrt{ 1 + \frac{\log(T/\delta)}{d}}  r \right)
<\epsilon,
\end{align}
while by \cref{lemma:gradient_norm_small_if_noise_larger}, if $k$ is in $I_1^c$, we have
\begin{align*}
\norm*{\nabla f(x_t)} \leq c_5 t_f(\delta)\eta L \left(u^2 d^2 \rho\log(T/\delta) + \sqrt{ 1 + \frac{\log(T/\delta)}{d}}  r \right),
\qquad\forall t\in J_k.
\end{align*}
This implies that $I_\epsilon\subseteq I_1$.

Observe that by \cref{lemma:gradient_changes_little_if_noise_small}, for any $k\in I_1$, we have
$$
\frac{1}{2}\norm*{\nabla f(x_{kt_f(\delta)})} \leq \norm*{\nabla f(x_t) }\leq 2 \norm*{\nabla f(x_{kt_f(\delta)})},
\qquad\forall t \in J_k.
$$
This implies in particular that for any $k\in I_\epsilon$, we have
$
\min_{t \in J_k} \norm*{\nabla f(x_t)}^2 \geq \frac{1}{16}\ep^2
$,
and consequently
$$
-\sum_{k\in I_1}\frac{\eta}{2}\min_{t\in J_k}\|\nabla f(x_t)\|^2
\leq 
-\sum_{k\in I_\epsilon}\frac{\eta}{2}\cdot\frac{\epsilon^2}{16}
\leq 
-\frac{T}{4t_f(\delta)}\cdot\frac{\eta}{2}\cdot \frac{\epsilon^2}{16}
=-\frac{T\eta \epsilon^2}{128t_f(\delta)}.
$$
Hence, by \cref{eq:func_decrease_T},
\begin{align}
f(x_{T}) - f(x_0) \leq\ &
-\frac{T\eta \epsilon^2}{128t_f(\delta)}
+ T \frac{c_5^2}{64} t_f(\delta)^2 \eta^3 L^2 \left(u^2 d^2 \rho\left(\log\frac{T}{\delta}\right)^2 + \sqrt{2\log(T/\delta)}  r \right)^2 \nonumber \\
&
+ T \eta u^4 \rho^2 \cdot c_1d^3\left(\log\frac{T}{\delta}\right)^3 + T\eta\cdot(\eta L) u^4 \rho^2\cdot c_1d^4 \left(\log\frac{T}{\delta}\right)^4
\nonumber \\
&
+ \eta c_1r^2(128t_f(\delta) + \eta L )\log\frac{T}{\delta} + T \eta\cdot c_1 \eta L  r^2. \label{eq:func_decrease_T_use_large_gradient}
\end{align}
Now, by our choices of $u$, $r$ and $\eta$, we have
\begin{align*}
& T \frac{c_5^2}{64} t_f(\delta)^2 \eta^3 L^2 \left(u^2 d^2 \rho\left(\log\frac{T}{\delta}\right)^2 + \sqrt{2\log(T/\delta)}  r \right)^2 \\
\leq\ &
T\eta\cdot \frac{c_5^2}{32} t_f(\delta)^2  (\eta L)^2 \left(u^4 d^4 \rho^2 \left(\log\frac{T}{\delta}\right)^4 + 2\log(T/\delta) r^2 \right) \\
\leq\ &
T\eta
\cdot\left(\frac{\epsilon^2}{2048c_2\left(\log\frac{T}{\delta}\right)^2}
+
\frac{\epsilon^2}{2048c_2\log(T/\delta)}\right)
\leq
\frac{T\eta\epsilon^2}{512
t_f(\delta)},
\end{align*}
where we used $\log(T/\delta)\geq 1$ and $2c_2\log(T/\delta)\geq t_f(\delta)$. We also have
\begin{align*}
& T \eta u^4 \rho^2 \cdot c_1d^3\left(\log\frac{T}{\delta}\right)^3 + T\eta\cdot (\eta L) u^4 \rho^2\cdot c_1d^4 \left(\log\frac{T}{\delta}\right)^4 + Tc_1L\eta^2r^2 \\
\leq\ &
T\eta\cdot \frac{\epsilon^2}{2048 c_2 d \log(T/\delta)}  
+ T\eta\cdot\frac{\epsilon^2}{2048c_2 t_f(\delta)\log(T/\delta)}
+ T\eta\cdot \frac{\epsilon^2}{1024 t_f(\delta)\log(T/\delta)} \\
\leq\ &
\frac{T\eta\epsilon^2}{512t_f(\delta)},
\end{align*}
where we used $c_2d\log(T/\delta)\geq t_f(\delta)$, $c_2\geq 1$ and $\log(T/\delta)\geq 1$. Finally,
\begin{align*}
\eta c_1r^2(128t_f(\delta) + \eta L )\log\frac{T}{\delta}
\leq
\frac{(128t_f(\delta)+1)\epsilon^2}{1024Lt_f(\delta)}
<\frac{\epsilon^2}{L}.
\end{align*}
By plugging these bounds into \cref{eq:func_decrease_T_use_large_gradient}, we get
\begin{align*}
f(x_T) - f(x_0)
<
-\frac{T\eta \ep^2}{128 t_f(\delta)}
+\frac{T\eta\epsilon^2}{512t_f(\delta)}+\frac{T\eta\epsilon^2}{512t_f(\delta)}
+\frac{\epsilon^2}{L}
\leq 
-\frac{T\eta \ep^2}{256 t_f(\delta)}
+\frac{\epsilon^2}{L}.
\end{align*}
Therefore, as long as
\begin{align*}
T \geq \frac{256 t_f(\delta)\left(\left(f(x_0) - f^*\right) + \ep^2/L)  \right)}{\eta \ep^2},
\end{align*}
we will get $f(x_T)< f^\ast$, which is a contradiction. Thus, we can conclude that on the event $\mathcal{E}_{T}(\delta)$, there are at most $T/4$ iterations for which $\|\nabla f(x_t)\|\geq\epsilon$.

We can now complete our proof by using the union bound (suppressing the dependence of some of the events on $\delta$ for notational simplicity) to derive
\begin{align*}
\bbP(\mathcal{E}_{T}^c)
\leq\ &
\bbP(\mathcal{H}_T^c)
+\sum_{t=0}^{T-1}\bbP(\mathcal{A}_t^c)
+\sum_{t=0}^{T-1}\bbP(\mathcal{G}_t^c)
+\sum_{k=0}^{K-1}\bbP(\mathcal{B}^c_{kt_f(\delta)}(\delta;t_f(\delta))) \\
\leq\ &
\frac{(T+4)\delta}{T}
+\delta+2\delta
+\frac{K\delta}{T}
\leq 6\delta.
\qedhere
\end{align*}
\end{proof}

\section{Escaping saddle point}\label{sec: escaping saddle appendix}
In this section, we first show that the travelling distance of the iterates can be bounded in terms of the function value improvement (\cref{appendix:improve_or_localize}). Utilizing this result, as well as \cref{proposition:nSG_adapted_to_ZO} in \cref{subsec:novel_concentration_inequality} which provides a concentration bound on the the zeroth-order noise, we then prove that sufficient function value decrease can be made near a saddle point in \cref{appendix:saddle_point_decrease}.
\subsection{Key quantities and notation}

We will use $\gamma$ to denote $-\lambda_{\min}(\nabla^2 f(x_0))$, where we know that $\gamma \geq \sqrt{\rho \ep}.$

\subsection{Improve or Localize}
\label{appendix:improve_or_localize}

In this subsection, we aim to bound the movement of the iterates across a number of steps in terms of the function value improvement made during these number of steps. 

We first state a simple result separating the norm of the difference between $x_{t_0 + \tau}$ and $x_{t_0}$ into a few different terms. 

\begin{lemma}
Consider the perturbed zeroth-order update \cref{algorithm:ZOPGD}. Then, for any $t_0 \in \bbN$ and $\tau \in \bbN$,
\begin{align}
    &\norm*{x_{t_0+\tau} - x_{t_0}}^2 \leq V_1(t_0,\tau) + V_2(t_0,\tau) + V_3(t_0,\tau) + V_4(t_0,\tau), \label{eq:x_tau_x_0_norm_diff_V1234}
\end{align}
where
\begin{equation}
\begin{aligned}
\begin{array}{l l}
 V_1(t_0,\tau) \coloneqq 8\eta^2 \tau \sum_{t=t_0}^{t_0 +\tau - 1} \norm*{\nabla f(x_t)}^2, & V_2(t_0,\tau) \coloneqq 8\eta^2 \norm*{\sum_{t=t_0}^{t_0 +\tau - 1} \frac{1}{m} \sum_{i=1}^m (Z_{t,i}Z_{t,i}^\top -I) \nabla f(x_t)}^2 \\
 V_3(t_0,\tau) \coloneqq 4\eta^2 \norm*{\sum_{t=t_0}^{t_0 +\tau - 1} Y_t}^2, & V_4(t_0,\tau) \coloneqq 4\eta^2 \norm*{\sum_{t=t_0}^{t_0 +\tau - 1} \frac{1}{m} \sum_{i=1}^m uZ_{t,i}Z_{t,i}^\top \tilde{H}_{t,i} Z_{t,i}}^2. 
\end{array}
\end{aligned}
\end{equation}
\end{lemma}
\begin{proof}
For notational convenience, let $t_0 := 0$. Then, applying the form of the perturbed zeroth-order update in \cref{algorithm:ZOPGD}, we get
\begin{align*}
    &\norm*{x_{\tau} - x_0}^2 \\
    &=  \norm*{\sum_{t=0}^{\tau - 1} x_{t+1} - x_t}^2 \\
    &= \eta^2 \norm*{\sum_{t=0}^{\tau - 1} \frac{1}{m} \sum_{i=1}^m Z_{t,i}Z_{t,i}^\top \nabla f(x_t) + \frac{1}{m} \sum_{i=1}^m uZ_{t,i}Z_{t,i}^\top \tilde{H}_{t,i} Z_{t,i} + Y_t}^2 \\
    &\leq 4\eta^2\norm*{\sum_{t=0}^{\tau - 1} \frac{1}{m} \sum_{i=1}^m Z_{t,i}Z_{t,i}^\top \nabla f(x_t)}^2 + 4\eta^2 \norm*{\sum_{t=0}^{\tau-1} \frac{1}{m} \sum_{i=1}^m uZ_{t,i}Z_{t,i}^\top \tilde{H}_{t,i} Z_{t,i}}^2 + 4 \eta^2 \norm*{\sum_{t=0}^{\tau-1} Y_t}^2  \\
    &\leq 4\eta^2\norm*{\sum_{t=0}^{\tau - 1} \frac{1}{m} \sum_{i=1}^m (Z_{t,i}Z_{t,i}^\top -I) \nabla f(x_t) + \sum_{t=0}^{\tau - 1} \nabla f(x_t)}^2 + 4\eta^2 \norm*{\sum_{t=0}^{\tau-1} \frac{1}{m} \sum_{i=1}^m uZ_{t,i}Z_{t,i}^\top \tilde{H}_{t,i} Z_{t,i}}^2 + 4 \eta^2 \norm*{\sum_{t=0}^{\tau-1} Y_t}^2 \\
    &\leq \underbrace{8\eta^2 \tau \sum_{t=0}^{\tau - 1} \norm*{\nabla f(x_t)}^2}_{V_1(0,\tau)} + \underbrace{8\eta^2 \norm*{\sum_{t=0}^{\tau - 1} \frac{1}{m} \sum_{i=1}^m (Z_{t,i}Z_{t,i}^\top -I) \nabla f(x_t)}^2}_{V_2(0,\tau)} + \underbrace{4\eta^2 \norm*{\sum_{t=0}^{\tau-1} Y_t}^2}_{V_3(0,\tau)} + \underbrace{4\eta^2 \norm*{\sum_{t=0}^{\tau-1} \frac{1}{m} \sum_{i=1}^m uZ_{t,i}Z_{t,i}^\top \tilde{H}_{t,i} Z_{t,i}}^2}_{V_4(0,\tau)}.
\end{align*}
\end{proof}

We now proceed to bound the terms $V_1(t_0,\tau),V_2(t_0,\tau),V_3(t_0,\tau)$ and $V_4(t_0,\tau)$.

First, we have the following result bounding $V_1(t_0, \tau)$.
\begin{lemma}
\label{lemma:bounding_V1}
Let $c_1 > 0, c_2 \geq 1, c_4 > 0, c_5 > 0, C_1 \geq 1$ be the absolute constants defined in the statements of the previous lemmas, and let $\delta \in (0,1/e]$ be arbitrary.

Suppose we choose $\eta$ such that 
\begin{align*}
\eta \leq 
\frac{1}{L t_f(\delta)}\cdot\min \left\{\frac{\sqrt{m}}{8c_4 (\operatorname{lr}(C_1 dmT/\delta))^{3/2} \sqrt{d} }, \frac{m}{128 c_1(\operatorname{lr}(C_1 dmT/\delta))^3 d}\right\}.   
\end{align*}
There are two cases to consider. 
\begin{enumerate}
    \item The first is when $\tau \geq t_f(\delta)$. In this case, split $\{t_0,t_0 + 1,\dots,t_0 + \tau - 1\}$ into $K \coloneqq \floor{\tau/t_f(\delta)}$ intervals:
\begin{align*}
    &J_k = \{t_0 + kt_f(\delta),\dots,t_0 + (k+1)t_f(\delta) - 1 \}, \quad 0 \leq k < K-1, \\
    &J_{K-1} = \{t_0 + (K-1) t_f(\delta),\dots,t_0 + \tau - 1\}. 
\end{align*}
Then, on the event
\small
$$
\mathcal{E}_{t_0, \tau}(\delta)\! \coloneqq\! \mathcal{H}_{t_0, \tau}(\delta)\!\cap\!
\left(\!\bigcap_{t=t_0}^{t_0 + \tau-1}\mathcal{A}_t(\delta)\!\right)\!\cap
\!
\left(\!\bigcap_{t=t_0}^{t_0 + \tau-1}\mathcal{G}_t(\delta)\right)\!\cap\!
\left(\!\bigcap_{k=0}^{K-2}
\mathcal{B}_{t_0 + kt_f(\delta)}(\delta;t_f(\delta))\!\right)
\!\cap\!
\mathcal{B}_{t_0 + (K-1)t_f(\delta)}(\delta;\tau\!-\!(K\!-\!1)t_f(\delta)),
$$
\normalsize
we have that
\begin{align*}
    V_1(t_0, \tau)  = & \  8\eta^2 \tau \sum_{t=t_0}^{t_0 + \tau - 1} \norm*{\nabla f(x_t)}^2 \\
    \leq & \ 64 \eta \tau t_f(\delta)  \left(\left(f(x_{0}) - f(x_{\tau})\right) + N_{u,r}(\tau; \delta)\right),
\end{align*}
where 
\begin{align}
    N_{u,r}(\tau;\delta) \coloneqq & \  
    \tau \frac{c_5^2}{64} \eta^3 t_f(\delta)^2 L^2 \left(u^2 d^2 \rho\left(\log\frac{T}{\delta}\right)^2 + \sqrt{2\log(T/\delta)}  r \right)^2  \nonumber \\
\ & + \tau \eta u^4 \rho^2 \cdot c_1d^3\left(\log\frac{T}{\delta}\right)^3 + \tau L \eta^2 u^4 \rho^2\cdot c_1d^4 \left(\log\frac{T}{\delta}\right)^4  \nonumber \\
\ &
+ \ \eta c_1r^2(128t_f(\delta) + \eta L )\log\frac{T}{\delta} + \tau c_1 L\eta^2  r^2 \nonumber \\
\ & + \ c_5^2 t_f^3(\delta)\eta^3 L^2 \left(u^2 d^2 \rho\log(T/\delta) + \sqrt{ 2\log(T/\delta)}  r \right)^2. \label{eq:N_u,r_def}
\end{align}

\item The second is when $\tau < t_f(\delta)$.  
Suppose we choose $u$ and $r$ such that
\begin{align*}
&u \leq
\frac{\sqrt{\ep}}{d \sqrt{\rho} \log(T/\delta)}
\cdot \min\left\{
\frac{1}{64c_5^2c_2},\frac{1}{2048 c_1c_2}
\right\}^{\! 1/4},
\qquad
r \leq \epsilon\cdot\min\left\{\frac{1}{8c_5\sqrt{2c_2}},
\frac{1}{32\sqrt{c_1}}\right\}.
\end{align*}

Suppose the event $\cap_{t = t_0}^{t_0 + \tau - 1} (\mathcal{A}_t(\delta) \cap \mathcal{G}_t(\delta)$ holds. Suppose also that $\norm*{\nabla f(x_{t_0})} \leq \ep.$ Then,
\begin{align*}
    V_1(t_0,\tau) \leq 32\eta^2 \tau^2 \ep^2 \leq 32\eta^2 (t_f(\delta))^2 \ep^2
\end{align*}
\end{enumerate}

\end{lemma}
\begin{proof}
\begin{enumerate}
    \item We first consider the case where $\tau \geq t_f(\delta)$. Let $I_1$ denote the set of indices $k$ such that for every time-step $t$ in the interval $J_k$, the gradient dominates the noise terms as
\begin{align}
   \norm*{\nabla f(x_t)}  >  8 t_f(\delta)\eta L \left( \frac{u}{2} \norm*{\frac{1}{m} \sum_{i=1}^m Z_{t,i} Z_{t,i}^\top \tilde{H}_{t,i} Z_{t,i} }+ \norm*{Y_t}\right).
\label{eq:gradient_dominates_noise}
\end{align}
WLOG, we may assume that $t_0 := 0$, and denote $V_1(\tau) := V_1(0,\tau)$. WLOG, we also assume that $\tau$ is a multiple of $t_f(\delta)$. From \cref{lemma:func_decrease_tau_bound}, on the event that $\mathcal{E}_{\tau}(\delta)$ holds and by our choice of $\eta$, we have 
\begin{align*}
f(x_{\tau}) - f(x_0) \leq\ &
-\sum_{k \in I_1}\frac{\eta}{2} \min_{t \in J_k} \norm*{\nabla f(x_t)}^2 + \tau \frac{c_5^2}{64} \eta^3 t_f(\delta)^2 L^2 \left(u^2 d^2 \rho\left(\log\frac{T}{\delta}\right)^2 + \sqrt{2\log(T/\delta)}  r \right)^2 \nonumber \\
&
+ \tau \eta u^4 \rho^2 \cdot c_1d^3\left(\log\frac{T}{\delta}\right)^3 + \tau L \eta^2 u^4 \rho^2\cdot c_1d^4 \left(\log\frac{T}{\delta}\right)^4
\nonumber \\
&
+ \eta c_1r^2(128t_f(\delta) + \eta L )\log\frac{T}{\delta} + \tau c_1 L\eta^2  r^2. 
\end{align*}
By \cref{lemma:gradient_changes_little_if_noise_small} (and our choice of $\eta$), it follows that for any $k \in I_1$, on the event $\cap_{t \in J_k} \mathcal{A}_t(\delta)$, we have
$$\sum_{t \in J_k} \norm*{\nabla f(x_t)}^2 \leq 4 t_f \min_{t \in J_k} \norm*{\nabla f(x_t)}^2.$$
Thus, on the event that $\mathcal{E}_{\tau}(\delta)$ holds, for our choice of $\eta$, we have 
\begin{align*}
    \eta \sum_{k \in I_1} \sum_{t \in J_k} \norm*{\nabla f(x_t)}^2 \leq & \ 4t_f(\delta) \eta \sum_{k \in I_1} \min_{t \in J_k} \norm*{\nabla f(x_t)}^2 \\
    \leq & \ 8t_f(\delta) \sum_{k \in I_1} \frac{\eta}{2} \min_{t \in J_k} \norm*{\nabla f(x_t)}^2 \\
    \leq &\ 8t_f(\delta) \left((f(x_{0}) - f(x_{\tau})) + \tau \frac{c_5^2}{64} \eta^3 t_f(\delta)^2 L^2 \left(u^2 d^2 \rho\left(\log\frac{T}{\delta}\right)^2 + \sqrt{2\log(T/\delta)}  r \right)^2 \right) \\
\ & + \ 8t_f(\delta) \left(\tau \eta u^4 \rho^2 \cdot c_1d^3\left(\log\frac{T}{\delta}\right)^3 + \tau L \eta^2 u^4 \rho^2\cdot c_1d^4 \left(\log\frac{T}{\delta}\right)^4 \right)
\nonumber \\
\ &
+ \ 8t_f(\delta)\left(\eta c_1r^2(128t_f(\delta) + \eta L )\log\frac{T}{\delta} + \tau c_1 L\eta^2  r^2\right). 
\end{align*}
Similarly, for any $k \in I_1^c$ (where $I_1^c$ denotes the complement of $I_1$ in $\{0,1,\dots,K-1\}$, i.e. intervals where the gradient is smaller than than the perturbation terms in some iteration), on the event $\left(\cap_{t \in J_k} \mathcal{A}_t(\delta) \right) \cap \left(\cap_{t \in J_k} \mathcal{G}_t(\delta) \right)$, by \cref{lemma:gradient_norm_small_if_noise_larger} (and our choice of $\eta$), we have
\begin{align*}
\norm*{\nabla f(x_t)} \leq c_5 t_f(\delta)\eta L \left(u^2 d^2 \rho\log(T/\delta) + \sqrt{ 2\log(T/\delta)}  r \right),
\qquad\forall t\in J_k.
\end{align*}
On the event that $\mathcal{E}_{\tau}(\delta)$ holds, this gives us then
\begin{align*}
    \eta \sum_{k \in I_1^c} \sum_{t \in J_k} \norm*{\nabla f(x_t)}^2 &\leq \eta \tau \left(c_5^2 t_f^2(\delta)\eta^2 L^2 \left(u^2 d^2 \rho\log(T/\delta) + \sqrt{ 2\log(T/\delta)}  r \right)^2\right).
\end{align*}
Hence, on the event that $\mathcal{E}_{\tau}(\delta)$ holds, we have that
\begin{align*}
    \eta \sum_{t=0}^{\tau - 1} \norm*{\nabla f(x_t)}^2 = & \ \eta \sum_{k \in I_1} \sum_{t \in J_k} \norm*{\nabla f(x_t)}^2 + \eta \sum_{k \in I_1^c} \sum_{t \in J_k} \norm*{\nabla f(x_t)}^2 \\
    \leq & \ 8t_f(\delta) \left((f(x_{0}) - f(x_{\tau})) + \tau \frac{c_5^2}{64} \eta^3 t_f(\delta)^2 L^2 \left(u^2 d^2 \rho\left(\log\frac{T}{\delta}\right)^2 + \sqrt{2\log(T/\delta)}  r \right)^2 \right) \\
\ & + \ 8t_f(\delta) \left(\tau \eta u^4 \rho^2 \cdot c_1d^3\left(\log\frac{T}{\delta}\right)^3 + \tau L \eta^2 u^4 \rho^2\cdot c_1d^4 \left(\log\frac{T}{\delta}\right)^4 \right) \\
\ &
+ \ 8t_f(\delta)\left(\eta c_1r^2(128t_f(\delta) + \eta L )\log\frac{T}{\delta} + \tau c_1 L\eta^2  r^2\right) \\
\ & + \ 8t_f(\delta) \eta \tau \left(c_5^2 t_f^2(\delta)\eta^2 L^2 \left(u^2 d^2 \rho\log(T/\delta) + \sqrt{ 2\log(T/\delta)}  r \right)^2\right).
\end{align*}
This yields the final result for the case $\tau \geq t_f(\delta)$. 
\item We next consider the case where $1 \leq \tau < t_f(\delta)$. 
Recall the notation that
\begin{align*}
    \calE(t_0,t_0 + \tau,\delta) \coloneqq \cap_{t = t_0}^{t_0 + \tau - 1} \left\{ \norm*{\nabla f(x_t)} > 8t_f(\delta) \eta L \left(\frac{u}{2} \norm*{\frac{1}{m} \sum_{i=1}^m Z_{t,i}Z_{t,i}^\top \tilde{H}_{t,i}Z_{t,i}} + \norm*{Y_t}\right) \right\}
\end{align*}
There are two cases to consider.
\begin{enumerate}
    \item On the event $\calE(t_0,t_0 +\tau,\delta) \cap \left(\cap_{t = t_0}^{t_0 + \tau - 1} \mathcal{A}_t(\delta) \right),$ we have by \cref{lemma:gradient_changes_little_if_noise_small} that $\norm*{\nabla f(x_t)} \leq 2 \norm*{\nabla f(x_0)}$ for each $t \in \{0,1,\dots,\tau - 1\}$. Then,
    \begin{align*}
        V_1(t_0,\tau) = 8\eta^2 \tau \sum_{t=t_0}^{t_0+\tau - 1}\norm*{\nabla f(x_t)}^2  \leq 8\eta^2 \tau^2 \left( 4 \norm*{\nabla f(x_0)}^2 \right) \leq 32\eta^2 \tau^2 \ep^2,
    \end{align*}
    where the final inequality uses the assumption that $\norm*{\nabla f(x_0)} \leq \ep.$
    \item Suppose the event  $\calE^c(t_0,t_0 +\tau,\delta) \cap \left(\cap_{t = t_0}^{t_0 + \tau - 1} \mathcal{A}_t(\delta) \right) \cap \left(\cap_{t = t_0}^{t_0 + \tau - 1} \mathcal{G}_t(\delta) \right)$ holds. In this case, by \cref{lemma:gradient_norm_small_if_noise_larger}, we have that for each $t \in \{t_0,t_0 +1,\dots, t_0 + \tau - 1\}$
    \begin{align*}
    \norm*{\nabla f(x_t)} \leq &  c_5 t_f(\delta)\eta L \left(u^2 d^2 \rho\left(\log\frac{T}{\delta}\right)^2 + \sqrt{1 + \frac{\log(T/\delta)}{d}}  r \right) \\
    \leq \ep,
    \end{align*}
    where the final inequality follows by our choice of $\eta, u$ and $r$ (cf. \cref{eq:nabla_f_norm_larger_than_ep_larger_than_noise}). Hence, 
    \begin{align*}
        V_1(t_0,\tau) = & \ 8\eta^2 \tau \sum_{t=t_0}^{t_0+\tau - 1} \norm*{\nabla f(x_t)}^2 \\ \leq & \ 8\eta^2 \tau^2 \left( c_5 t_f(\delta)\eta L \left(u^2 d^2 \rho\left(\log\frac{T}{\delta}\right)^2 + \sqrt{1 + \frac{\log(T/\delta)}{d}}  r \right) \right)^2 \\
        \leq & \ 8\eta^2 \tau^2 \ep^2 < 32\eta^2 \tau^2 \ep^2.
    \end{align*}
    The final result for the case $\tau < t_f(\delta)$ then follows.
\end{enumerate}
\end{enumerate}

\end{proof}

We proceed to bound $V_2(t_0,\tau)$.

\begin{lemma}
\label{lemma:bounding_V2}
Let $c_1 > 0, c_2 \geq 1, c_4 > 0, c_5 > 0, C_1 \geq 1$ be the absolute constants defined in the statements of the previous lemmas, and let $\delta \in (0,1/e]$ be arbitrary and $\tau > 0$ be arbitrary.
Suppose we choose $\eta$ such that 
\begin{align*}
\eta \leq 
\frac{1}{L t_f(\delta)}\cdot\min \left\{\frac{\sqrt{m}}{8c_4 (\operatorname{lr}(C_1 dmT/\delta))^{3/2} \sqrt{d} }, \frac{m}{128 c_1(\operatorname{lr}(C_1 dmT/\delta))^3 d}\right\}.
\end{align*}
Let $T_s$ denote an integer such that $T_s \geq \max\left\{\tau,t_f(\delta) \right\}$, and for any $F > 0$, define
\begin{align*}
    B(\delta; F) \coloneqq \frac{8 t_f(\delta) (F + N_{u,r}(T_s,\delta))}{\eta} \left(T_s + \frac{d}{m}\right)
    (\operatorname{lr}(CT^2/\delta))^2, \qquad b_{\tau}(\delta; F) \coloneqq \frac{t_f(\delta) \tau F}{\eta}.
\end{align*}
Let $c', C > 0$ denote the same constants as in the statement of \cref{proposition:nSG_adapted_to_ZO}. Denote the event that 
\small
    \begin{align*}
        \mbox{either }\  &
        \sum_{t=t_0}^{t_0 + \tau - 1} \frac{d}{m} (\operatorname{lr}(CT^2/\delta))^2 \norm*{\nabla f(x_t)}^2 \geq B(\delta; F) \\
        \mbox{or }\ &
        \sqrt{\frac{V_2(t_0,\tau)}{8\eta^2}}\! \leq \! c'\! \sqrt{\!\max\left\{\!\sum_{t=t_0}^{t_0 + \tau - 1} \frac{d}{m} (\operatorname{lr}(CT^2/\delta))^2 \norm*{\nabla f(x_t)}^2, b_{\tau}(\delta;F)\!\right\} \left(\log\left(\frac{CT^2}{\delta}\right) \!+ \!\log\left(\log\left(\frac{B(\delta;F)}{b_{\tau}(\delta; F)}\right)+ 1\right)\right)}
    \end{align*}
\normalsize
holds as $\mathcal{L}_{t_0, \tau}(\delta;F)$\footnote{We note that by construction, $B(\delta;F)\geq b_\tau(\delta;F)$}. We show that $\bbP(\mathcal{L}_{t_0,\tau}(\delta;F)) \geq 1 - \frac{\delta}{T}.$ Finally, denote the event $\mathcal{M}_{t_0,T_s}(F)$ as the event that $f(x_{t_0}) - f(x_{t_0 + T_s}) < F$.

Then, on the event $\mathcal{L}_{t_0, \tau}(\delta)  \cap \calE_{t_0,T_s}(\delta) \cap \mathcal{M}_{t_0,T_s}(F)$ (where $\calE_{0,T_s}(\delta)$ is as defined in \cref{lemma:bounding_V1}),
\begin{align}
    V_2(t_0,\tau) \leq 8c'^2 \beta_1(\delta;F)\eta t_f(\delta)\max\left\{\frac{8d}{m} (\operatorname{lr}(CT^2/\delta))^2 \left(F + N_{u,r}(T_s,\delta)\right), \tau F \right\},
\end{align}
where
\begin{align*}
    \beta_{1}(\delta;F) \coloneqq \log\left(\frac{CT^2}{\delta}\right) + \log\left(\log\left(\frac{B(\delta;F)}{b_{1}(\delta; F)}\right)+ 1\right)
    .
\end{align*}
\end{lemma}
\begin{proof}
We note that $\bbP(\mathcal{L}_{t_0,\tau}(\delta;F)) \geq 1 - \frac{\delta}{T}.$ is a direct consequence of \cref{proposition:nSG_adapted_to_ZO}. 
In the rest of the proof, without loss of generality, we assume that $t_0 = 0$ for notational simplicity.
On the event $\mathcal{L}_{0,\tau}(\delta;F) \cap \mathcal{E}_{0,T_s}(\delta)\cap \mathcal{M}_{t_0,T_s}(F)$, suppose that
\begin{align*}
    \quad & \sum_{t=0}^{\tau - 1} \frac{d}{m} (\operatorname{lr}(CT^2/\delta))^2 \norm*{\nabla f(x_t)}^2 \geq B(\delta ; F) = \frac{8 t_f(\delta) (F + N_{u,r}(T_s,\delta))}{\eta} \left(T_s + \frac{d}{m}\right)(\operatorname{lr}(CT^2/\delta))^2 \\ \implies & \eta \sum_{t=0}^{\tau - 1} \norm*{\nabla f(x_t)}^2 \geq 8t_f(\delta) (F + N_{u,r}(T_s,\delta)) \\
    \implies & \eta  \sum_{t=0}^{T_s - 1} \norm*{\nabla f(x_t)}^2 \geq 8t_f(\delta) (F + N_{u,r}(T_s,\delta)) \\
    \implies & 8\eta^2 T_s \sum_{t=0}^{T_s - 1} \norm*{\nabla f(x_t)}^2 \geq 64\eta T_s t_f(\delta)(F + N_{u,r}(T_s,\delta)) \\
    \implies & 8\eta^2 T_s \sum_{t=0}^{T_s - 1} \norm*{\nabla f(x_t)}^2 \geq 64\eta T_s t_f(\delta)(f(x_0) - f(x_{T_s}) + N_{u,r}(T_s,\delta)), \quad \mbox{since } f(x_0) - f(x_{T_s}) \leq F \\
    \iff & V_1(0,T_s) \geq 64\eta T_s t_f(\delta)(f(x_0) - f(x_{T_s}) + N_{u,r}(T_s,\delta)),
\end{align*}
where we note the last equation contradicts \cref{lemma:bounding_V1}. 
For notational simplicity, denote
\begin{align*}
    \beta_{\tau}(\delta;F) \coloneqq \log\left(\frac{CT^2}{\delta}\right) + \log\left(\log\left(\frac{B(\delta;F)}{b_{\tau}(\delta; F)}\right)+1\right).
\end{align*}
Observe that $\beta_{1}$ is larger than $\beta_{\tau}$ for every $\tau \geq 1$.  
Since $\mathcal{L}_{t_0,\tau}(\delta;F)$ holds, we must have then that
\begin{align*}
&\sqrt{\frac{V_2(0,\tau)}{8\eta^2}} \leq c' \sqrt{\max\left\{\sum_{t=0}^{\tau - 1} \frac{d}{m} (\operatorname{lr}(CT^2/\delta))^2 \norm*{\nabla f(x_t)}^2, b_{\tau}(\delta;F) \right\} \beta_1(\delta;F)}.
\end{align*}
Now, continuing, recalling the definition of $V_1(0,T_s) = 8\eta^2 T_s \sum_{t=0}^{T_s - 1} \norm*{\nabla f(x_t)}^2$
\begin{align*}
   V_2(0,\tau) \leq & \ c'^2 \beta_{1}(\delta;F)  \max\left\{8\eta^2\sum_{t=0}^{\tau - 1} \frac{d}{m} (\operatorname{lr}(CT^2/\delta))^2 \norm*{\nabla f(x_t)}^2, 8\eta^2 b_{\tau}(\delta;F) \right\} \\
   \leq & \ c'^2 \beta_{1}(\delta;F) \max\left\{8\eta^2\sum_{t=0}^{T_s - 1} \frac{d}{m} (\operatorname{lr}(CT^2/\delta))^2 \norm*{\nabla f(x_t)}^2, 8\eta^2 b_{\tau}(\delta;F) \right\}  \\
   \leq & \ c'^2 \beta_1(\delta;F) \max\left\{\frac{d}{m} (\operatorname{lr}(CT^2/\delta))^2 \frac{V_1(0,T_s)}{T_s}, 8\eta t_f(\delta) \tau F \right\} \\
   \labelrel\leq{eq:use_V1_T_s_for_V2} & \ c'^2 \beta_1(\delta;F)\max\left\{\frac{d}{m} (\operatorname{lr}(CT^2/\delta))^2 \left(64\eta t_f(\delta) (f(x_0) - f(x_{T_S}) + N_{u,r}(T_s,\delta))\right), 8\eta t_f(\delta) \tau F \right\} \\
   \labelrel\leq{eq:use_f_diff_bdd_by_F_for_V2} & \ c'^2 \beta_1(\delta;F)\max\left\{\frac{d}{m} (\operatorname{lr}(CT^2/\delta))^2 \left(64\eta t_f(\delta) (F + N_{u,r}(T_s,\delta))\right), 8\eta t_f(\delta) \tau F \right\} \\
   = & \ c'^2 \beta_1(\delta;F)(8\eta t_f(\delta))\max\left\{\frac{d}{m} (\operatorname{lr}(CT^2/\delta))^2 \left(8(F + N_{u,r}(T_s,\delta))\right), \tau F \right\}.
\end{align*}
We note that (\ref{eq:use_V1_T_s_for_V2}) is a consequence of \cref{lemma:bounding_V1}, while (\ref{eq:use_f_diff_bdd_by_F_for_V2}) comes from our assumption that the event $\mathcal{M}_{t_0,T_s}(F)$ holds, i.e. $f(x_{t_0}) - f(x_{t_0 + T_s}) \leq F$.

\end{proof}

We next bound $V_3(t_0,\tau)$ and $V_4(t_0,\tau)$. 

\begin{lemma}
\label{lemma:bounding_V3_V4}
Let $c > 0$ denote the same constant in \cref{lemma:nSG_conc_with_corollaries}. Consider any arbitrary $0 < \delta \leq 1/e,$ and let $\tau \geq t_f(\delta)$ be arbitrary. Let $\mathcal{N}_{t_0,\tau}(\delta)$ denote the event that
\begin{align*}
    V_3(t_0,\tau) \coloneqq 4\eta^2 \norm*{\sum_{t=t_0}^{t_0 + \tau - 1} Y_t}^2 \leq 4c_6\eta^2 \tau \log(2dT/\delta) r^2,
\end{align*}
where $c_6 > 0$ is an absolute constant.
Then, by \cref{lemma:nSG_conc_with_corollaries}, $\bbP(\mathcal{N}_{t_0,\tau}(\delta)) \geq 1 - \frac{\delta}{T}.$ Denote the event 
\begin{align*}
    \mathcal{O}_t(\delta) \coloneqq \left\{\frac{1}{m} \sum_{i=1}^m \norm*{Z_{t,i}}^8 \leq c_7 d^4 \left(\log\left(\frac{T}{\delta}\right) \right)^4 \right\},
\end{align*}
where $c_7 > 0$ is an absolute constant.
Then, on the event $\cap_{t=t_0}^{t_0 + \tau - 1} \mathcal{O}_t(\delta)$, we have 
\begin{align*}
    V_4(t_0,\tau) \leq 4c_7 \eta^2 \tau^2 \rho^2 u^4 d^4 \left(\log\left(\frac{T}{\delta}\right) \right)^4.
\end{align*}
Moreover, for each $t$, $\bbP(\mathcal{O}_t(\delta)) \geq 1 - \frac{\delta}{T}$.
\end{lemma}
\begin{proof}
The proof for $V_3(t_0,\tau)$ follows directly from \cref{lemma:nSG_conc_with_corollaries}, by picking $c_6$ to be the $c$ that appears in the statement of \cref{lemma:nSG_conc_with_corollaries}. Meanwhile, observe that
\begin{align*}
    V_4(t_0,\tau) = & \  4\eta^2 \norm*{\sum_{t=0}^{\tau-1} \frac{1}{m} \sum_{i=1}^m uZ_{t,i}Z_{t,i}^\top \tilde{H}_{t,i} Z_{t,i}}^2 \\
    \leq & \ 4\eta^2 \tau \left(\sum_{t=t_0}^{t_0+\tau - 1} \norm*{\frac{1}{m} \sum_{i=1}^m u Z_{t,i}Z_{t,i}^\top \tilde{H}_{t,i}Z_{t,i}}^2  \right) \\
    \labelrel\leq{eq:use_H_t,i_property_bdd_V4} & \ 4\eta^2 \tau \sum_{t=t_0}^{t_0 + \tau - 1} \frac{1}{m} \sum_{i=1}^m \rho^2 u^4 \norm*{Z_{t,i}}^8 \\
    \leq & \ 4c_7 \eta^2 \tau^2 \rho^2 u^4 d^4 \left(\log\left(\frac{T}{\delta}\right) \right)^4.
\end{align*}
Above, to derive (\ref{eq:use_H_t,i_property_bdd_V4}), we used the bound that $\norm*{\tilde{H}_{t,i}} \leq \rho u \norm*{Z_{t,i}}$. The final inequality is a consequence of our assumption that $\cap_{t=t_0}^{t_0 + \tau - 1} \mathcal{O}_t(\delta)$ holds. Finally, the result that $\bbP(\mathcal{O}_t(\delta)) \geq 1 - \frac{\delta}{T}$ holds due to \cref{lemma:X_norm_kth_power_sum_bound}, where we note that we may pick the absolute constant $c_7$ to be equal to $2C c^4$, where $c,C > 0$ are the absolute constants that appear in the statement of \cref{lemma:X_norm_kth_power_sum_bound}.
\end{proof}

Finally, combining the earlier results, we have the following technical result, which bounds the travelling distance of the iterates in terms of the decrease in function value decrease. 

\begin{lemma}[Improve or Localize]
\label{lemma:improve_or_localize_exact}
Consider the perturbed zeroth-order update \cref{algorithm:ZOPGD}. Let $c' > 0, c_1 > 0, c_2 \geq 1, c_4 > 0, c_5 > 0, c_6 > 0, c_7 > 0, C_1 \geq 1$ be the absolute constants defined in the statements of the previous lemmas, and let $\delta \in (0,1/e]$ be arbitrary. Consider any $T_s \geq t_f(\delta)$. For any $F > 0$, suppose $f(x_{T_s}) - f(x_0) > - F,$ i.e. $f(x_0) - f(x_{T_s}) < F$. Suppose that the event 
\begin{align*}
    \mathcal{P}_{t_0,T_s}(\delta,F) \coloneqq \cap_{\tau = 1}^{T_s} \left(\mathcal{L}_{t_0,\tau}(\delta;F) \cap \mathcal{N}_{t_0,\tau}(\delta) \right) \cap \left(\cap_{t=t_0}^{t_0 + T_s - 1} \mathcal{O}_t(\delta) \cap \mathcal{A}_t(\delta) \cap \mathcal{G}_t(\delta) \right)  \cap \left(\cap_{\tau = t_f(\delta)}^{T_s-1} \calE_{t_0,\tau}(\delta) \right)
\end{align*}
holds, where the events $\calE_{t_0,\tau}(\delta), \mathcal{L}_{t_0,\tau}(\delta), \mathcal{N}_{t_0,\tau}(\delta), \mathcal{O}_t(\delta)$ are as defined in \cref{lemma:bounding_V1}, \cref{lemma:bounding_V2} and \cref{lemma:bounding_V3_V4}, and $\mathcal{G}_t(\delta)$ and $\mathcal{A}_t(\delta)$ are as defined in \cref{lemma:Y_t_bdd_throughout} and \cref{lemma:bound_nabla_diff_At}. Suppose we choose $u$, $r$ and $\eta$ such that
\begin{align*}
&u \leq
\frac{\sqrt{\ep}}{d \sqrt{\rho} \log(T/\delta)}
\cdot \min\left\{
\frac{1}{64c_5^2c_2},\frac{1}{2048 c_1c_2}
\right\}^{\! 1/4},
\qquad
r \leq \epsilon\cdot\min\left\{\frac{1}{8c_5\sqrt{2c_2}},
\frac{1}{32\sqrt{c_1}}\right\}, \\
&\eta \leq
\frac{1}{Lt_f(\delta)}\min \left\{\frac{1}{\log(T/\delta)}, \frac{\sqrt{m}}{8c_4 (\operatorname{lr}(C_1 dmT/\delta))^{3/2} \sqrt{d}}, \frac{m}{128 c_1 (\operatorname{lr}(C_1 dmT/\delta))^3 d} \right\}.
\end{align*} 
Suppose $\eta \leq \min\left\{1, \frac{1}{t_f(\delta)}, \frac{1}{t_f\delta L}\right\}$. 
Suppose also we pick $u$ and $r$ small enough such that
\begin{align*}
    u \leq \frac{r^{1/2}}{d \log(T/\delta)\rho^{1/2}}, \quad r^2 \leq \min \left\{\frac{F }{\eta T_s \log(T/\delta) \left(\frac{65c_5^2}{8} + 132c_1 + 1  \right)}, \frac{F}{4c_6\log(2dT/\delta)+  4c_7 \eta T_s} \right\}.
\end{align*}

Then, for each $\tau \in \{0,1,\dots,T_s\}$, we have that
\begin{align*}
    \norm*{x_{t_0 + \tau} - x_{t_0}}^2 \leq \phi_{T_s}(\delta,F),
\end{align*}
where 
\begin{align*}
    \phi_{T_s}(\delta,F) \! \leq \! \max\left\{128\eta T_s t_f(\delta)F, 32\eta^2 (t_f(\delta))^2 \ep^2 \right\}\! +\!  8c'^2 \beta_1(\delta;F)\eta t_f(\delta)\max\left\{\frac{16d}{m} (\operatorname{lr}(CT^2/\delta))^2 F, T_s F \right\} +  T_s \eta t_f(\delta)F,
\end{align*}
where $\beta_1(\delta;F)$ is defined as in \cref{lemma:bounding_V2}.
Moreover, $\bbP(\mathcal{P}_{t_0,T_s}(\delta,F)) \geq 1 - \frac{12T_s \delta}{T}$.
\end{lemma}
\begin{proof}
We recall that
\begin{align*}
   \norm*{x_{t_0+\tau} - x_{t_0}}^2 \leq & \ \underbrace{8\eta^2 \tau \sum_{t=t_0}^{t_0 + \tau - 1} \norm*{\nabla f(x_t)}^2}_{V_1(t_0,\tau)} + \underbrace{8\eta^2 \norm*{\sum_{t=t_0}^{t_0 + \tau - 1} \frac{1}{m} \sum_{i=1}^m (Z_{t,i}Z_{t,i}^\top -I) \nabla f(x_t)}^2}_{V_2(t_0,\tau)} \\
    & + \ \underbrace{4 \eta^2 \norm*{\sum_{t=t_0}^{t_0 + \tau-1} Y_t}^2}_{V_3(t_0,\tau)} + \underbrace{4\eta^2 \norm*{\sum_{t=t_0}^{t_0 + \tau-1} \frac{1}{m} \sum_{i=1}^m uZ_{t,i}Z_{t,i}^\top \tilde{H}_{t,i} Z_{t,i}}^2}_{V_4(t_0,\tau)}\\
\end{align*}
By \cref{lemma:bounding_V1}, \cref{lemma:bounding_V2}, and \cref{lemma:bounding_V3_V4}, which bound $V_1(t_0,\tau)$, $V_2(t_0,\tau)$, and $V_3(t_0,\tau), V_4(t_0,\tau)$ respectively, on the event $\mathcal{P}_{t_0,T_s}(\delta,F)$, we have, for any $0 \leq \tau \leq T_s$,
\begin{align*}
    \norm*{x_{\tau} - x_0}^2 \leq & \ V_1(0,\tau) + V_2(0,\tau) + V_3(0,\tau) + V_4(0,\tau) \\
    \leq & \ \max\left\{64\eta \tau t_f(\delta)(F + N_{u,r}(\tau;\delta)), 32\eta^2 (t_f(\delta))^2 \ep^2\right\} \\
    & + \ 8c'^2 \beta_1(\delta;F)\eta t_f(\delta)\max\left\{\frac{8d}{m} (\operatorname{lr}(CT^2/\delta))^2 \left(F + N_{u,r}(T_s,\delta)\right), \tau F \right\} \\
    & + \ 4c_6 \eta^2 \tau \log(2dT/\delta) r^2 + 4c_7 \eta^2 \tau^2 \rho^2 u^4 d^4 \left(\log(T/\delta) \right)^4,
\end{align*}
where $N_{u,r}(\tau;\delta)$ is defined as in \cref{lemma:bounding_V1}.

For the simplified bound (which does not contain $N_{u,r}(\tau;\delta)$), it remains for us to show that our choice of $u$ and $r$ ensures that $N_{u,r}(T_s,\delta) \leq F$ and
\begin{align*}
   4c_6 \eta^2 T_s \log(2dT/\delta) r^2 + 4c_7 \eta^2 T_s^2 \rho^2 u^4 d^4 \left(\log(T/\delta) \right)^4 \leq \eta T_s t_f(\delta)F.
\end{align*}
First, our choice of $u$ ensures that 
$$u^4 d^4 \rho^2 (\log(T/\delta))^4 \leq r^2.$$
Next, recall that
\begin{align*}
    N_{u,r}(\tau;\delta) \coloneqq & \  
    \tau \frac{c_5^2}{64} \eta^3 t_f(\delta)^2 L^2 \left(u^2 d^2 \rho\left(\log\frac{T}{\delta}\right)^2 + \sqrt{2\log(T/\delta)}  r \right)^2  \nonumber \\
\ & + \tau \eta u^4 \rho^2 \cdot c_1d^3\left(\log\frac{T}{\delta}\right)^3 + \tau L \eta^2 u^4 \rho^2\cdot c_1d^4 \left(\log\frac{T}{\delta}\right)^4  \nonumber \\
\ &
+ \ \eta c_1r^2(128t_f(\delta) + \eta L )\log\frac{T}{\delta} + \tau c_1 L\eta^2  r^2 \nonumber \\
\ & + \ c_5^2 t_f^3(\delta)\eta^3 L^2 \left(u^2 d^2 \rho\log(T/\delta) + \sqrt{ 2\log(T/\delta)}  r \right)^2.
\end{align*}

Recalling our choice of $\eta$ such that $$\eta \leq \min\{1,\frac{1}{t_f(\delta)}, \frac{1}{t_f(\delta)L} \},$$
it follows that
\begin{align*}
    N_{u,r}(T_s;\delta) &\leq \eta T_s r^2 \left( \frac{8c_5^2}{64} \log(T/\delta) + 2c_1 + 2c_1 + (128c_1+1) \log(T/\delta) + c_1 + 8c_5^2 \log(T/\delta) \right) \\
    &\leq \eta T_s r^2 \log(T/\delta) \left(\frac{65c_5^2}{8} + 132c_1 + 1  \right) \leq F,
\end{align*}
where the last inequality follows choosing $r$ such that $r^2 \leq \frac{F }{\eta T_s \log(T/\delta) \left(\frac{65c_5^2}{8} + 132c_1 + 1  \right)}$. 
Similarly, we have 
\begin{align*}
    & \ 4c_6 \eta^2 T_s \log(2dT/\delta) r^2 + 4c_7 \eta^2 T_s^2 \rho^2 u^4 d^4 \left(\log(T/\delta) \right)^4  \\
    \leq & \ \eta T_s t_f(\delta) \left(4c_6 \eta \log(2dT/\delta)r^2 + 4c_7 \eta T_s \rho^2 u^4 d^4 (\log(T/\delta))^4 \right) \\
    \leq & \ \eta T_s t_f(\delta) \left(4c_6 \eta \log(2dT/\delta)r^2 + 4c_7 \eta T_s r^2 \right) 
\end{align*}
By choosing $r$ such that
$$r^2 \leq \frac{F}{4c_6\log(2dT/\delta)+ 4c_7 \eta T_s},$$
it follows that 
\begin{align*}
   4c_6 \eta^2 T_s \log(2dT/\delta) r^2 + 4c_7 \eta^2 T_s^2 \rho^2 u^4 d^4 \left(\log(T/\delta) \right)^4 \leq \eta T_s t_f(\delta)F,
\end{align*}
as desired.

We next lower bound the probability of 
$$    \mathcal{P}_{t_0,T_s}(\delta,F) \coloneqq \cap_{\tau = 1}^{T_s} \left(\mathcal{L}_{t_0,\tau}(\delta;F) \cap \mathcal{N}_{t_0,\tau}(\delta) \right) \cap \left(\cap_{t=t_0}^{t_0 + T_s - 1} \mathcal{O}_t(\delta) \cap \mathcal{A}_t(\delta) \cap \mathcal{G}_t(\delta) \right)  \cap \left(\cap_{\tau = t_f(\delta)}^{T_s} \calE_{t_0,\tau}(\delta) \right).$$ 
Observe that 
\begin{align*}
    & \ \cap_{\tau = t_f(\delta)}^{T_s} \calE_{t_0,\tau}(\delta) \\ 
    = & \  \cap_{\tau = t_f(\delta)}^{T_s}\!\! \left(\!\!\mathcal{H}_{t_0, \tau}(\delta)\!\cap\!
\left(\!\bigcap_{t=t_0}^{t_0 + \tau-1}\mathcal{A}_t(\delta) \cap \mathcal{G}_t(\delta)\!\right)\!\cap\!
\left(\bigcap_{k=0}^{K-2}
\mathcal{B}_{t_0 + kt_f(\delta)}(\delta;t_f(\delta))\!\right)
\!\cap \!
\mathcal{B}_{t_0 + (K-1)t_f(\delta)}(\delta;\tau\!-\!(K\!-\!1)t_f(\delta)) \right) \\
= & \ \cap_{\tau =  t_f(\delta)}^{T_s}\!\! \left(\!\!\mathcal{H}_{t_0, \tau}(\delta)
\!\cap\!
\left(\bigcap_{k=0}^{K-2}
\mathcal{B}_{t_0 + kt_f(\delta)}(\delta;t_f(\delta))\!\right)
\!\!\cap \!\!
\mathcal{B}_{t_0 + (K-1)t_f(\delta)}(\delta;\tau\!-\!(K\!-\!1)t_f(\delta)) \!\right) \cap \!\left(\!\bigcap_{t=t_0}^{T_s - 1}\mathcal{A}_t(\delta) \!\cap\! \mathcal{G}_t(\delta) \!\right).
\end{align*}
Note this implies that $\cap_{\tau = t_f(\delta)}^{T_s} \calE_{t_0,\tau}(\delta) \cap \left(\cap_{t = t_0}^{T_s - 1} \mathcal{A}_t(\delta) \cap \mathcal{G}_t(\delta) \right) = \cap_{\tau = t_f(\delta)}^{T_s} \calE_{t_0,\tau}(\delta) $
We note that by \cref{lemma:function_decrease_tighter_decomp},
\begin{align*}
    \bbP\left(\left(\cap_{\tau = t_f(\delta)}^{T_s} \mathcal{H}_{t_0,\tau}(\delta)\right)^c \right) \leq \frac{5 T_s \delta}{T}. 
\end{align*}
Meanwhile, we note that 
\begin{align*}
   \cap_{t = t_0}^{T_s -1} \mathcal{B}_{t}(\delta; t_f(\delta)) \subseteq \cap_{\tau = t_f(\delta)}^{T_s}\!\! \left(
\left(\bigcap_{k=0}^{K-2}
\mathcal{B}_{t_0 + kt_f(\delta)}(\delta;t_f(\delta))\!\right)
\!\cap \!
\mathcal{B}_{t_0 + (K-1)t_f(\delta)}(\delta;\tau\!-\!(K\!-\!1)t_f(\delta)) \right).
\end{align*}
Hence, by \cref{lemma:some_t_Z_bdd_away_from_0}, we have that 
\begin{align*}
    & \ \bbP\left( \left(\cap_{\tau =  t_f(\delta)}^{T_s}\!\! \left(
\left(\bigcap_{k=0}^{K-2}
\mathcal{B}_{t_0 + kt_f(\delta)}(\delta;t_f(\delta))\!\right)
\!\cap \!
\mathcal{B}_{t_0 + (K-1)t_f(\delta)}(\delta;\tau\!-\!(K\!-\!1)t_f(\delta)) \right) \right)^c\right)\\
\leq & \  \bbP\left(\left(\cap_{t = t_0}^{T_s -1} \mathcal{B}_{t}(\delta; t_f(\delta)) \right)^c\right) \leq \frac{T_s \delta}{T}.
\end{align*}
Meanwhile, by  \cref{lemma:Y_t_bdd_throughout} and \cref{lemma:bound_nabla_diff_At}, we may bound 
\begin{align*}
    \bbP\left( \left(\!\bigcap_{t=t_0}^{T_s - 1}\mathcal{A}_t(\delta) \cap \mathcal{G}_t(\delta) \!\right)^c \right) \leq \frac{T_s \delta}{T} + \frac{2T_s \delta}{T} = \frac{3T_s \delta}{T}.
\end{align*}
Hence, it follows that 
\begin{align*}
    \bbP\left( \left(\cap_{\tau = t_f(\delta)}^{T_s} \calE_{t_0,\tau}(\delta) \cap \left(\cap_{t = t_0}^{T_s - 1} \mathcal{A}_t(\delta) \cap \mathcal{G}_t(\delta) \right) \right)^c \right) \leq \frac{5T_s\delta}{T} + \frac{T_s \delta}{T} + \frac{3T_s \delta}{T} = \frac{9T_s \delta}{T}.
\end{align*}
Meanwhile, it follows from our results in the preceding lemmas that 
\begin{align*}
    \bbP\left( \left(\cap_{\tau = 1}^{T_s} \left(\mathcal{L}_{t_0,\tau}(\delta;F) \cap \mathcal{N}_{t_0,\tau}(\delta) \right) \cap \left(\cap_{t=t_0}^{T_s - 1} \mathcal{O}_t(\delta)\right) \right)^c \right) \leq \frac{3T_s \delta}{T}.
\end{align*}
Hence, it follows that $\bbP(\mathcal{P}_{t_0,T_s}(\delta,F)) \geq 1 - \frac{12T_s \delta}{T}$.

\end{proof}

\subsection{Proving function value decrease near saddle point}
\label{appendix:saddle_point_decrease}

We next build on the technical result earlier to prove that each time we are near the saddle point, there is a constant probability of making significant function value decrease. We briefly provide a high-level proof outline below. In our proof, we introduce a coupling argument connecting two closely-related sequences both starting from the saddle, differing only in the sign of their perturbative term along the minimum eigendirection of the Hessian at the saddle. Specifically, when function decrease from a saddle is not sufficiently large, due to the earlier technical result, we know that the coupled sequences will remain within a radius $\phi$ of the original saddle for a large number (which we will denote as $T_s$) of iterations. We then utilize this fact to show that the difference of the coupled sequence will (with some constant probability) grow exponentially large, eventually moving out of their specified radius $\phi$ within $T_s$ iterations, leading to a contradiction. 

Our first result formally introduces the coupling, setting the stage for the rest of our arguments. For notational convenience, in this section, unless otherwise specified, we will often assume that the initial iterate $x_0$ is an $\ep$-saddle point.

\saddleDiffExpression*

\begin{proof}
Observe that
\begin{align*}
    &\hat{x}_{t+1} := x_{t+1} - x_{t+1}' \\
    &= x_t - \eta \left(\nabla f(x_t) + \xi_{g_0}(t) + \xi_u(t) Y_t\right) - \left[x_t' - \eta \left(\nabla f(x_t') + \xi_{g_0}'(t) + \xi_u'(t) + Y_t'\right)\right] \\
    &= \hat{x}_t - \eta \left[ \left(\nabla f(x_t) - \nabla f(x_t')\right) + \left(\xi_{g_0}(t) - \xi_{g_0}'(t)\right) + \left(\xi_{u}(t) - \xi_{u}'(t)\right)  + \left(Y_t - Y_t'\right)\right] \\
    &= \hat{x}_t - \eta H\hat{x}_t - \eta (\bar{H}_t - H) \hat{x}_t -  \eta \hat{\xi}_{g_0}(t) -\eta \hat{\xi}_u(t) - \eta \hat{Y}_t \\
    &= -\underbrace{\eta \sum_{\tau = 0}^t (I - \eta H)^{t - \tau} \hat{\xi}_{g_0}(\tau)}_{W_{g_0}(t+1)} - \underbrace{\eta \sum_{\tau = 0}^t (I - \eta H)^{t - \tau} (\bar{H}_{\tau} - H)\hat{x}_{\tau}  }_{W_{H}(t+1)} - \underbrace{\eta \sum_{\tau = 0}^t (I - \eta H)^{t - \tau}\hat{\xi}_u(\tau)}_{W_u(t+1)} -\underbrace{\eta \sum_{\tau = 0}^t (I - \eta H)^{t - \tau} \hat{Y}_{\tau}}_{W_{p}(t+1)}
\end{align*}
where 
\begin{align*}
    &\xi_{g_0}(t) = \frac{1}{m}\sum_{i=1}^m (Z_{t,i}Z_{t,i}^\top - I) \nabla f(x_t), \quad \xi'_{g_0}(t) = \frac{1}{m} \sum_{i=1}^m (Z_{t,i}
(Z_{t,i})^{\top} - I) \nabla f(x_t'), \quad \hat{\xi}_{g_0}(t) = \xi_{g_0}(t) - \xi'_{g_0}(t), \\
& \xi_{u}(t) = \frac{1}{m} \sum_{i=1}^m \frac{u}{2} Z_{t,i}Z_{t,i} \tilde{H}_{t,i} Z_{t,i}, \quad \xi_{u}'(t) = \frac{1}{m} \sum_{i=1}^m \frac{u}{2} Z_{t,i}Z_{t,i} \tilde{H}_{t,i}' Z_{t,i}, \quad \hat{\xi}_u(t) = \xi_u(t) - \xi_u'(t), \\
& \hat{Y}_t = Y_t - Y_t', \quad \bar{H}_t = \int_{0}^1 \nabla^2 f(a x_t + (1 -a)x_t')  da.
\end{align*}
To derive the final equality, we utilized the fact that $x_{0}' = x_{0}$. This completes our proof.
\end{proof}

Suppose $x_{0}$ is an $\ep$-saddle point. Recall that $\gamma > 0$ denotes $-\lambda_{\min}(\nabla^2 f(x_{0}))$, where we know that \sout{$\gamma \geq \sqrt{\rho \ep}$.}

\begin{align*}
\gamma \geq \bar{\psi} := \begin{cases}
\min\{\psi,1,L\} & \mbox{if } f(\cdot) \mbox{ is } (\ep,\psi, \sqrt{\rho \ep})\mbox{-strict saddle} \mbox{ for any } \psi  > \sqrt{\rho \ep} \\
\sqrt{\rho \ep} & \mbox{otherwise}.
\end{cases}
\end{align*}

In the sequel, for any $t \geq 0$, it is helpful to define the quantities
\begin{align}
\label{eq:beta_alpha_definition}
    \beta(t)^2 := \frac{(1 + \eta \gamma)^{2t}}{(\eta \gamma)^2 + 2\eta \gamma}, \quad \alpha(t)^2 := \frac{(1 + \eta \gamma)^{2t} - 1}{(\eta \gamma)^2 + 2\eta \gamma}.
\end{align}

We next introduce some probabilistic events (and their implications) which, if true, can be used to bound the sizes of $\norm*{W_{g_0}(t+1)}, \norm*{W_{u}(t+1)}$,  $\norm*{W_{u}(t+1)}$ (and as we will see in the next result, indirectly bound $\norm*{W_H(t+1)}$. These bounds will be useful in the final proof of making function value progress near a saddle point. 

\begin{lemma}
\label{lemma:events_for_saddle_func_decrease}
We assume $\delta \in (0,1/e]$ throughout the lemma. Suppose that we pick $u,r$ and $\eta$ as specified in \cref{lemma:improve_or_localize_exact}. Suppose $T_s \geq t_f(\delta)$. Suppose also that 
$$f(x_{T_s}) - f(x_0) > -F, \quad f(x_{T_s}') - f(x_0) > -F.$$
Then, we have the following results.

\begin{enumerate}[leftmargin=0.5cm]
    \item Let $\mathcal{S}_\phi(\delta)$ denote the event 
    \begin{align*}
        \mathcal{S}_\phi(\delta) \coloneqq \left\{\max\{\norm*{x_{t} - x_0}^2, \norm*{x_t' - x_0}^2 \} \leq \phi_{T_s}(\delta,F), \quad \forall 0 \leq t \leq T_s\right\}.
    \end{align*}
    In addition, let $\mathcal{S}_{u}(\delta)$ denote the event 
    \begin{align*}
        \mathcal{S}_u(\delta) \coloneqq \left\{\norm*{W_u(t+1)} \leq \eta \beta(t+1) \frac{\sqrt{3}}{\sqrt{\eta \bar{\psi}}} \left(2c_3\rho  d^2 (\log(T/\delta))^2 \right) u^2 , \quad \forall 0 \leq t \leq T_s-1\right\},
    \end{align*}
    where $c_3$ is the same absolute constant as the $c_3$ in the preceding lemmas. Then,
 $$\bbP(\mathcal{S}_\phi(\delta) \cap \mathcal{S}_u(\delta)) \geq 1 - \frac{24T_s \delta}{T}.$$
    
\item 
Consider defining the event $\mathcal{R}_{t}(\delta)$, which is the event where
\small
\begin{align*}
    &\mbox{\textbf{\underline{either}} } \sum_{\tau=0}^{t} (1+\eta \gamma)^{2(t-\tau)} \frac{dL^2}{m}\norm*{x_{\tau} - x_{\tau}'}^2 (\operatorname{lr}(CT^2/\delta))^2 \geq G_{T_s}(\delta,F), \mbox{ \textbf{\underline{or}} }\\
    &\norm*{W_{g_0}(t+1)} \\
    &\leq c'\eta\sqrt{\!\max\!\left\{\!\left(\!\operatorname{lr}\left(\frac{CT^2}{\delta}\right)\right)^2 \!\sum_{\tau=0}^{t}\! \frac{dL^2}{m} \!(1 +\eta \gamma)^{2(t-\tau)} \norm*{x_{\tau}\! -\! x_{\tau}'}^2, g(t+1)\right\}\! \left(\log\left(\frac{CdT^2}{\delta}\!\right) \!+\! \log\left(\log\left(\!\frac{G_{T_s}(\delta,F)}{g(t+1)}\!\right)\right)\!+\!1\right)}
\end{align*}
normalsize
holds. Above, $c',C$ refer to the same constants as in \cref{proposition:nSG_adapted_to_ZO}, and
\begin{align*}
    G_{T_s}(\delta,F) \coloneqq 8 \sum_{\tau = 0}^{T_s-1} (1+\eta \gamma)^{2\tau} \frac{dL^2}{m} (\operatorname{lr}(CT^2/\delta))^2 \phi_{T_s}(\delta,F) + \left(\frac{\beta(T_s) \eta r}{60 \sqrt{d}}\right)^2,  \qquad g(t+1) \coloneqq \left(\frac{\beta(t+1) \eta r}{60 \sqrt{d}}\right)^2.
\end{align*}
Then, $\bbP(\mathcal{R}_t(\delta)) \geq 1 - \frac{\delta}{T}$.
Suppose the event
$$\left(\cap_{t=0}^{T_s-1} \mathcal{R}_t(\delta)\right)  \cap \mathcal{S}_\phi(\delta)$$
holds. Then, the event $\mathcal{S}_{g_0}(\delta)$ holds, where
$$\mathcal{S}_{g_0}(\delta) \coloneqq \cap_{t=0}^{T_s - 1} \mathcal{S}_{g_0,t}(\delta),$$
and $\mathcal{S}_{g_0,t}(\delta)$ is defined as
\begin{align*}
    \mathcal{S}_{g_0,t}(\delta) \coloneqq \left\{\norm*{W_{g_0}(t+1)} \leq \zeta_1(\delta,F) c' \eta\!\sqrt{\!\max\left\{\left(\!\operatorname{lr}\left(\!\frac{CT^2}{\delta}\!\right)\!\right)^2\! \sum_{\tau=0}^{t} \frac{dL^2}{m} (1 +\eta \gamma)^{2(t-\tau)} \norm*{x_{\tau} - x_{\tau}'}^2, g(t\!+\!1)\!\right\}\!}\right\}
\end{align*}
where
\begin{align*}
    \zeta_1(\delta,F) \coloneqq \left(\log\left(\frac{CdT^2}{\delta}\right) + \log\left(\log\left(\frac{G_{T_s}(\delta,F)}{g(1)}\right)\right)+1\right).
\end{align*}

\item In addition, let $\mathcal{S}_p(\delta)$ denote the event 
\begin{align*}
    \mathcal{S}_p(\delta) \coloneqq \left\{\norm*{W_p(t+1)} \leq \frac{2\sqrt{2\log(T/\delta)} \beta(t+1)\eta r}{\sqrt{d}} \quad \forall 0 \leq t \leq T_s - 1\right\}.
\end{align*}
Then, $\bbP(\mathcal{S}_p(\delta)) \geq 1 - \frac{T_s \delta}{T}$.
\end{enumerate}

\end{lemma}
\begin{proof}
We consider the three claims separately.
\begin{enumerate}
    \item Note that our assumptions satisfy the conditions required in \cref{lemma:improve_or_localize_exact}. Hence, by \cref{lemma:improve_or_localize_exact}, on the event $\mathcal{P}_{0,T_s}(\delta,F)$, we have that $\norm*{x_{\tau} - x_0}^2 \leq \phi_{T_s}(\delta,F)$. Simultaneously, on the event  $\mathcal{P}_{0,T_s}(\delta,F)$, we know that $\cap_{t=0}^{T_s -1}\mathcal{G}_t(\delta)$ holds, i.e.
    \begin{align}
        \frac{1}{m} \sum_{i=1}^m \norm*{Z_{t,i}}^4 \leq 2c_3 d^2 \left(\log(T/\delta)\right)^2, \quad \forall 0 \leq t \leq T_s-1. \label{eq:Z_ti_4th_power_bdd}
    \end{align}
    Thus, for $W_u(t+1)$, we have that
\begin{align*}
    \norm*{W_u(t+1)} &= \norm*{\eta \sum_{\tau=0}^t (I- \eta H)^{t-\tau} \hat{\xi}_u(\tau)} \\
    &\leq \norm*{\eta \sum_{\tau=0}^t (I- \eta H)^{t-\tau} \xi_u(\tau)} + \norm*{\eta \sum_{\tau=0}^t (I- \eta H)^{t-\tau} \hat{\xi}_u'(\tau)} \\
    &\leq \eta \sum_{\tau = 0}^t (1 + \eta \gamma)^{t - \tau} \left(\norm*{\frac{1}{m} \sum_{i=1}^m \frac{u}{2} Z_{t,i}Z_{t,i} \tilde{H}_{t,i}Z_{t,i}} + \norm*{\frac{1}{m} \sum_{i=1}^m \frac{u}{2} Z_{t,i}Z_{t,i} \tilde{H}_{t,i}'Z_{t,i}} \right) \\
    &\leq \eta \sum_{\tau = 0}^t (1 + \eta \gamma)^{t-\tau} \frac{\rho}{m} \sum_{i=1}^m \norm*{Z_{t,i}}^4 u^2 \\
    &\labelrel\leq{eq:Wu_Z^4_bound} \eta \sum_{\tau = 0}^t (1 + \eta \gamma)^{t-\tau} \rho (2c_3) d^2 (\log(T2/\delta))^2 u^2 \\
    &\leq \eta \frac{(1 +\eta \gamma)^{t+1}}{\eta \gamma} \left(2c_3 \rho C d^2 (\log(T/\delta))^2 \right) u^2 \\
    &\labelrel={eq:use_beta_expression_for_Wu} \eta \beta(t+1) \frac{\sqrt{(\eta \gamma)^2 + 2 \eta \gamma}}{\eta \gamma} \left(2c_3\rho  d^2 (\log(T/\delta))^2 \right) u^2 \\
    &\leq \eta \beta(t+1) \frac{\sqrt{3}}{\sqrt{\eta \gamma}} \left(2c_3\rho  d^2 (\log(T/\delta))^2 \right) u^2 \\ &\labelrel\leq{eq:gamma_leq_L} \eta \beta(t+1) \frac{\sqrt{3}}{\sqrt{\eta \bar{\psi}}} \left(2c_3\rho  d^2 (\log(T/\delta))^2 \right) u^2
\end{align*}
where the inequality in (\ref{eq:Wu_Z^4_bound}) holds due to \cref{eq:Z_ti_4th_power_bdd}, the equality in (\ref{eq:use_beta_expression_for_Wu}) holds due to the definition of $\beta(t+1)$, and the inequality in (\ref{eq:gamma_leq_L}) used the fact that $\gamma \geq \bar{\psi}$. 

Hence the event 
\begin{align*}
     \cap_{t=0}^{T_s}\left\{\norm*{x_t - x_0}^2 \leq \phi_{T_s}(\delta,F) \mbox{ and }\right\} \cap \mathcal{S}_u(\delta)
\end{align*}
holds with probability at least $1 - \frac{12T_s \delta}{T}$.
    
Note that by the coupling, the distribution of $x_{\tau}'$ is the same as that of $x_{\tau}$. Thus, by the assumption $f(x_{T_s}') - f(x_0) > -F$, it follows by a similar argument that the bound $\norm*{x_{\tau}' - x_0}^2 \leq \phi_{T_s}(\delta,F)$ also holds with probability at least $1 - \frac{12T_s\delta}{T}$. The claim then follows by an application of the union bound.
    
        \item     For the second claim, observe first that the claim $\bbP(\mathcal{R}_t(\delta)) \geq 1 - \frac{\delta}{T}$ is a consequence of \cref{proposition:nSG_adapted_to_ZO}. Suppose next that $f(x_{T_s}) - f(x_0) > -F$. Then, by definition of the event $\mathcal{S}_\phi(\delta)$, we know that 
    \begin{align*}
        \norm*{x_{\tau} -x_0}^2 \leq \phi_{T_s}(\delta,F), \qquad \norm*{x_{\tau}' -x_0}^2 \leq \phi_{T_s}(\delta,F)
    \end{align*}
    where $\phi_{T_s}(\delta,F)$ is as defined in \cref{lemma:improve_or_localize_exact}. 
    
    Suppose now that $\mathcal{R}_t(\delta)$ holds true, and suppose for contradiction that
\begin{align*}
    & \ \sum_{\tau = 0}^t (1+\eta \gamma)^{2(t-\tau)} \frac{dL^2}{m} \norm*{x_{\tau} -x_\tau'}^2(\operatorname{lr}(CT^2/\delta))^2 \\
    \geq & \  G_{T_s}(\delta,F) \\
    = & \ 8 \sum_{\tau = 0}^{T_s-1} (1+\eta \gamma)^{2\tau} \frac{dL^2}{m} (\operatorname{lr}(CT^2/\delta))^2 \phi_{T_s}(\delta,F) + \left(\frac{\beta(T_s) \eta r}{60 \sqrt{d}}\right)^2.
\end{align*}
This implies that there exists some $0 \leq \tau \leq t \leq T_s$ such that $\norm*{x_{\tau} -x_\tau'}^2 \geq 8 \phi_{T_s}(\delta,F)$. However, we also know that on the event $\mathcal{S}_\phi(\delta)$,
\begin{align*}
    \norm*{x_{\tau} -x_\tau'}^2 \leq 2 \norm*{x_{\tau} -x_0}^2 + 2 \norm*{x_{\tau}' -x_0}^2 \leq 4 \phi_{T_s}(\delta,F).
\end{align*}
This leads to a contradiction. We must then have that
\begin{align*}
    &\norm*{W_{g_0}(t+1)} \leq \zeta_1(\delta,F) c' \eta\!\sqrt{\!\max\left\{\left(\!\operatorname{lr}\left(\!\frac{CT^2}{\delta}\!\right)\!\right)^2\! \sum_{\tau=0}^{t} \frac{dL^2}{m} (1 +\eta \gamma)^{2(t-\tau)} \norm*{x_{\tau} - x_{\tau}'}^2, g(t\!+\!1)\!\right\}\!\! \!},
\end{align*}
where
\begin{align*}
    \zeta_1(\delta,F) \coloneqq \sqrt{\log\left(\frac{CdT^2}{\delta}\right) + \log\left(\log\left(\frac{G(\delta, F)}{g(1)}\right)+1 \right)}
\end{align*}

\item Observe that
\begin{align*}
    W_p(t+1) &= \eta \sum_{\tau = 0}^t (I - \eta H)^{t - \tau} \hat{Y}_{\tau} = \eta \sum_{\tau = 0}^t (1 + \eta \gamma)^{t - \tau} (2 (Y_{\tau})_1),
\end{align*}
which means that $W_p(t+1)$ is a 1-dimensional Gaussian with variance
\begin{align}
\label{eq:W_p_variance_calc_with_alpha}
    \eta^2 \sum_{\tau = 0}^{t} (1 + \eta \gamma)^{2(t - \tau)} \frac{4 r^2 }{d} = \frac{4 \eta^2 r^2}{d} \frac{(1 + \eta \gamma)^{2(t+1)} - 1}{2\eta \gamma + (\eta \gamma)^2} = \frac{4\eta^2 r^2 \alpha(t+1)^{2}}{d}.
\end{align}
Since $\alpha(t+1) \leq \beta(t+1)$, using the subGaussianity of a Gaussian distribution, it follows that for any $t$, with probability at least $1 - \delta/T$,
\begin{align*}
    \norm*{W_p(t+1)} \leq \frac{2\sqrt{2 \log(T/\delta)} \beta(t+1) \eta r}{\sqrt{d}}.
\end{align*}
\end{enumerate}
\end{proof}

For any $F > 0$, we are now ready to show that the algorithm makes a function decrease of $F$ with $\Omega(1)$ probability near an $\ep$-saddle point.

\begin{proposition}
\label{proposition:saddle_point_func_decrease}
Suppose that $x_{t_0}$ is an $\ep$-approximate saddle point. Let $c' > 0, c_1 > 0, c_2 \geq 1, c_4 > 0, c_5 > 0, c_6 > 0, c_7 > 0, C_1 \geq 1$ be the absolute constants defined in the statements of the previous lemmas, and let $\delta \in (0,1/e]$ be arbitrary. Consider any $F > 0$. As in the statement of \cref{lemma:improve_or_localize_exact}, suppose we choose $u$, $r$ and $\eta$ such that
\begin{align*}
&u \leq
\frac{\sqrt{\ep}}{d \sqrt{\rho} \log(T/\delta)}
\cdot \min\left\{
\frac{1}{64c_5^2c_2},\frac{1}{2048 c_1c_2}
\right\}^{\! 1/4},
\qquad
r \leq \epsilon\cdot\min\left\{\frac{1}{8c_5\sqrt{2c_2}},
\frac{1}{32\sqrt{c_1}}\right\}, \\
&\eta \leq
\frac{1}{Lt_f(\delta)}\min \left\{\frac{1}{\log(T/\delta)}, \frac{\sqrt{m}}{8c_4 (\operatorname{lr}(C_1 dmT/\delta))^{3/2} \sqrt{d}}, \frac{m}{128 c_1 (\operatorname{lr}(C_1 dmT/\delta))^3 d} \right\}.
\end{align*} 
Suppose we pick 
\begin{align}
\label{eq:T_s_definition}
    T_s = \max\left\{\ceil{\frac{\iota}{ \eta \bar{\psi}}}, t_f(\delta),4 \right\},
\end{align}
where 
\begin{align*}
    \iota = \max\left\{\log\left(2\sqrt{\phi_{T_s}(\delta,F)} \frac{20\sqrt{d} \sqrt{\eta^2 \gamma^2 + 2\eta \gamma}}{\eta r}\right), 1 \right\},
\end{align*}

\begin{align*}
\bar{\psi} := \begin{cases}
\min\{\psi,1,L\} & \mbox{if } f(\cdot) \mbox{ is } (\ep,\psi, \sqrt{\rho \ep})\mbox{-strict saddle} \mbox{ for any } \psi  > \sqrt{\rho \ep} \\
\sqrt{\rho \ep} & \mbox{otherwise}.
\end{cases}
\end{align*}
Suppose in addition that $u,\eta$ also satisfy the conditions
\begin{align*}
        &u \leq \sqrt{\frac{r \sqrt{\eta \bar{\psi}}}{120\sqrt{3}c_3\sqrt{d} \rho d^2 (\log(T/\delta))^2}}, 
                &\eta \leq \max\left\{\frac{1}{c' c_9 \zeta_1(\delta,F)}, \frac{m\bar{\psi} }{360 \iota (c')^2 c_9^2 dL^2 \left(\operatorname{lr}\left(\frac{CT^2}{
    \delta} \right)\right)^2 \zeta_1(\delta,F)^2 },\frac{1}{2\bar{\psi}}  \right\}, 
\end{align*}
where $\zeta_1(\delta,F)$ is as defined in \cref{lemma:improve_or_localize_exact}, $c',c_3, C > 0$ are the same constants as in the previous results, and $c_9 = 2\sqrt{2} + \frac{1}{20}$.
Suppose also that $\phi_{T_s}(\delta,F)$ satisfies the bound
\begin{align}
\label{eq:phi_Ts_upper_bound}
\phi_{T_s}(\delta,F) \leq \left(\frac{\bar{\psi}}{60c_9  \iota \rho \log(T/\delta)}\right)^2.    
\end{align}

Then, with probability at least $\frac{1}{3} - \frac{13T_s\delta}{T}$, $f(x_{t_0+T_s}) - f(x_{t_0}) \leq -F$.

\end{proposition}

\begin{proof}[Proof  of \cref{proposition:saddle_point_func_decrease}]
Without loss of generality, we assume that $t_0 =0$. By \cref{lemma:saddle_diff_term_expression}, we have
\begin{align*}
    & \ \hat{x}_{t+1} \\
    := & \ x_{t+1} - x_{t+1}' \\
    = & \ -\underbrace{\eta \sum_{\tau = t_0}^t (I - \eta H)^{t - \tau} \hat{\xi}_{g_0}(\tau)}_{W_{g_0}(t+1)} - \underbrace{\eta \sum_{\tau = t_0}^t (I - \eta H)^{t - \tau} (\bar{H}_{\tau} - H)\hat{x}_{\tau}  }_{W_{H}(t+1)} - \underbrace{\eta \sum_{\tau = t_0}^t (I - \eta H)^{t - \tau}\hat{\xi}_u(\tau)}_{W_u(t+1)} -\underbrace{\eta \sum_{\tau = t_0}^t (I - \eta H)^{t - \tau} \hat{Y}_{\tau}}_{W_{p}(t+1)}
\end{align*}
where 
\begin{align*}
    &\xi_{g_0}(t) = \frac{1}{m}\sum_{i=1}^m (Z_{t,i}Z_{t,i}^\top - I) \nabla f(x_t), \quad \xi'_{g_0}(t) = \frac{1}{m} \sum_{i=1}^m (Z_{t,i}
(Z_{t,i})^{\top} - I) \nabla f(x_t'), \quad \hat{\xi}_{g_0}(t) = \xi_{g_0}(t) - \xi'_{g_0}(t), \\
& \xi_{u}(t) = \frac{1}{m} \sum_{i=1}^m \frac{u}{2} Z_{t,i}Z_{t,i} \tilde{H}_{t,i} Z_{t,i}, \quad \xi_{u}'(t) = \frac{1}{m} \sum_{i=1}^m \frac{u}{2} Z_{t,i}Z_{t,i} \tilde{H}_{t,i}' Z_{t,i}, \quad \hat{\xi}_u(t) = \xi_u(t) - \xi_u'(t), \\
& \hat{Y}_t = Y_t - Y_t', \quad \bar{H}_t = \int_{0}^1 \nabla^2 f(a x_t + (1 -a)x_t')  da.
\end{align*}

Recall that we define for $t \geq 0$,
\begin{align*}
    \beta(t)^2 := \frac{(1 + \eta \gamma)^{2t}}{(\eta \gamma)^2 + 2\eta \gamma}, \quad \alpha(t)^2 := \frac{(1 + \eta \gamma)^{2t} - 1}{(\eta \gamma)^2 + 2\eta \gamma}.
\end{align*}

Throughout the proof, we suppose for contradiction that $$f(x_{T_s}) - f(x_0) > -F, \quad f(x_{T_s}') - f(x_0) > -F,$$
and assume the event
\begin{align*}
    \left(\cap_{t=0}^{T_s-1} \mathcal{R}_t(\delta)\right)  \cap \mathcal{S}_\phi(\delta)\cap \mathcal{S}_u(\delta) \cap \mathcal{S}_p(\delta)  
\end{align*}
holds, where the events intersected are defined in \cref{lemma:events_for_saddle_func_decrease}. Then, by \cref{lemma:events_for_saddle_func_decrease}, the event $\mathcal{S}_{g_0}(\delta)$ (also defined in \cref{lemma:events_for_saddle_func_decrease}) holds\footnote{We may also directly assume that $\mathcal{S}_{g_0}(\delta)$ also holds, but our way of reasoning prevents double counting of probabilities.}. 

Consider the following induction argument, where we seek to show that there exists an absolute constant $c_9> 0$ such that for every $t \in \{0,1,\dots,T_s\}$,
\begin{align}
\label{eq:saddle_inductive_claim_bdd}
\norm*{x_{t} - x_{t}'} \leq c_9 \log(T/\delta) \frac{\beta(t) \eta r}{\sqrt{d}}, \mbox{ and } \max\left\{\norm*{W_{g_0}(t)}, \norm*{W_H(t)}, \norm*{W_u(t)}\right\} \leq \frac{\beta(t+1)\eta r}{\sqrt{d}}
\end{align}
Combined with a lower bound on $\norm*{W_{p}(t+1)}$ (which makes use of the property that $W_p(t+1)$ is a 1-dimensional Gaussian), we will then use the inductive claim in \cref{eq:saddle_inductive_claim_bdd} to show that
\begin{align*} \norm*{W_p(T_s)} \geq 2\left(\norm*{W_{g_0(T_s)}} + \norm*{W_H(T_s)} + \norm*{W_u(T_s)}\right).
\end{align*}
Since $W_p(t+1)$ is a 1-dimensional Gaussian random variable with a standard deviation that grows exponentially with $t$, by our choice of $T_s$, we will see that $\norm*{x_{T_s} - x_{T_s}'}$ is larger than what expect (since our assumptions imply that $\max\left\{\norm*{x_{T_s} - x_0}^2, \norm*{x_{T_s}' - x_0}^2 \right\} \leq \phi_{T_s}(\delta,F)$, i.e. $x_{T_s}$ and $x_{T_s}'$ both remain close to $x_0$ and hence close to each other). This yields a contradiction, implying that on the event we assumed to hold, i.e. 
\begin{align*}
    \left(\cap_{t=0}^{T_s-1} \mathcal{R}_t(\delta)\right)  \cap \mathcal{S}_\phi(\delta) \cap \mathcal{S}_p(\delta)  
\end{align*}
the assumption 
\begin{align*}
    f(x_{T_s}) - f(x_0) > -F, \quad \mbox{ and } f(x_{T_s}') - f(x_0) > -F
\end{align*}
is not true, i.e. one of the sequences must have made function value progress of at least $F$. 

We proceed to prove \cref{eq:saddle_inductive_claim_bdd}. Observe that the claim holds for the base case $t= 0$; this is true since $x_0 = x_0'$. Now suppose that this holds for all $\tau \leq t$. We will seek to show that \cref{eq:saddle_inductive_claim_bdd} holds for $t+1$ as well. We do so by bounding the norms of $W_{g_0}(t+1), W_H(t+1), W_u(t+1)$ and $W_p(t+1)$ respectively.

\begin{enumerate}
    \item (Bounding $\norm*{W_{g_0}(t+1)}$) Since the event $\mathcal{S}_{g_0}(\delta)$ holds, it follows that for each $0 \leq t \leq T_s - 1$, we have that
    \begin{align*}
\norm*{W_{g_0}(t+1)} \leq \zeta_1(\delta,F) c' \eta\!\sqrt{\!\max\left\{\left(\!\operatorname{lr}\left(\!\frac{CT^2}{\delta}\!\right)\!\right)^2\! \sum_{\tau=0}^{t} \frac{dL^2}{m} (1 +\eta \gamma)^{2(t-\tau)} \norm*{x_{\tau} - x_{\tau}'}^2, g(t\!+\!1)\!\right\}\!}
\end{align*}
where
\begin{align*}
    \zeta_1(\delta,F) \coloneqq \left(\log\left(\frac{CdT^2}{\delta}\right) + \log\left(\log\left(\frac{G_{T_s}(\delta,F)}{g(1)}\right)\right)+1\right),
\end{align*}
and the terms $G_{T_s}(\delta,F)$ and $g(1)$ are defined as in \cref{lemma:events_for_saddle_func_decrease}. Recall by the inductive claim in \cref{eq:saddle_inductive_claim_bdd} that there exists $c_9> 0$ such that
\begin{align*}
\norm*{x_{\tau} - x_{\tau}'} \leq c_9 \log(T/\delta) \frac{\beta(t) \eta r}{\sqrt{d}} \quad \forall \  0 \leq \tau \leq t.
\end{align*}

Hence, it follows that
\begin{align*}
    &\norm*{W_{g_0}(t+1)} \leq c'  \zeta_1(\delta,F) \eta \max\left\{\sqrt{t+1}    \left(\!\operatorname{lr}\left(\!\frac{CT^2}{\delta}\!\right)\!\right) \frac{c_9 \sqrt{d}L}{\sqrt{m}} \frac{\beta(t) \eta r}{\sqrt{d}}, \frac{\beta(t+1) \eta r}{60 \sqrt{d}} \right\}.
\end{align*}
Hence, noting the choice of $T_s$ in \cref{eq:T_s_definition}, by choosing $\eta$ such that
\begin{align}
&c'c_9 \zeta_1(\delta,F) \eta \sqrt{T_s} \left(\!\operatorname{lr}\left(\!\frac{CT^2}{\delta}\!\right)\!\right) \frac{\sqrt{d}L}{\sqrt{m}} \leq \frac{1}{60} \iff \eta \leq \frac{m \bar{\psi}}{360\iota(c')^2 c_9^2 dL^2 \left(\operatorname{lr}\left(\frac{CT^2}{
    \delta} \right)\right)^2 \zeta_1(\delta,F)^2 },\mbox{ and } \label{eq:key_eta_psi_dependence}\\
&c'c_9 \zeta_1(\delta,F) \eta \leq 1. \nonumber
\end{align}
it follows that 
\begin{align*}
    \norm*{W_{g_0}(t+1)} \leq \frac{\beta(t+1) \eta r}{60\sqrt{d}}.
\end{align*}

\item Meanwhile, the term $W_H(t+1)$ can be bounded as follows. By the inductive assumption in \cref{eq:saddle_inductive_claim_bdd}, we have that
\begin{align*}
    \norm*{\hat{x}_{\tau}} = \norm*{x_{\tau} - x_{\tau}'} \leq c_9 \log(T/\delta) \frac{\beta(\tau) \eta r}{\sqrt{d}} \quad \forall \  0 \leq \tau \leq t.
\end{align*}
Moreover, on the event our proof assumes, we know that $$\max\left\{\norm*{x_{\tau} - x_0}^2, \norm*{x_{\tau}' - x_0}^2 \right\}\leq \phi_{T_s}(\delta,F).$$ Thus, using the $\rho$-Hessian Lipschitz property, we have
\begin{align*}
&\ \norm*{W_H(t+1)} = \eta \norm*{\sum_{\tau = 0}^t (I - \eta H)^{t -\tau} (\bar{H}_{\tau} - H) \hat{x}_{\tau}} \\
    \leq& \  \eta \sum_{\tau = 0}^t (1 + \eta \gamma)^{t - \tau} \rho \sqrt{\phi_{T_s}(\delta,F)} \frac{c_9 \log(T/\delta) \beta(\tau) \eta r}{\sqrt{d}} \\
    \leq & \  c_9 (t+1) \log(T/\delta)\eta  \rho \sqrt{\phi_{T_s}(\delta,F)} \frac{\beta(t) \eta r}{\sqrt{d}}  \\
    \leq & \  c_9 T_s \log(T/\delta) \eta\rho \sqrt{\phi_{T_s}(\delta,F)} \frac{\beta(t) \eta r}{\sqrt{d}}.
\end{align*}
Given our choice of $T_s$ in \cref{eq:T_s_definition}, if
$$c_9 T_s \log(T/\delta) \eta\rho \sqrt{\phi_{T_s}(\delta,F)} \leq \frac{1}{60} \iff \phi_{T_s}(\delta,F) \leq \left(\frac{\bar{\psi}}{60c_9  \iota \rho \log(T/\delta)}\right)^2$$
it follows that
\begin{align*}
\norm*{W_H(t+1)} \leq \frac{\beta(t+1) \eta r}{60\sqrt{d}}.
\end{align*}

\item Meanwhile, for $W_u(t+1)$, since the event $\mathcal{S}_u(\delta)$ holds, we have that 
\begin{align*}
    \norm*{W_u(t+1)} \leq \eta \beta(t+1) \frac{\sqrt{3}}{\sqrt{\eta \bar{\psi}}} \left(2c_3\rho  d^2 (\log(T/\delta))^2 \right) u^2.
\end{align*}
Now, by picking 
\begin{align*}
    \eta \beta(t+1) \frac{\sqrt{3}}{\sqrt{\eta \bar{\psi}}} \left(2c_3\rho  d^2 (\log(T/\delta))^2 \right) u^2 \leq \frac{\beta(t+1)\eta r}{60\sqrt{d}} \iff u \leq \sqrt{\frac{r \sqrt{\eta \bar{\psi}}}{120\sqrt{3}c_3\sqrt{d} \rho d^2 (\log(T/\delta))^2}},
\end{align*}
it follows that with probability $1 - \delta/T$, $\norm*{W_u(t+1)} \leq \frac{\beta(t+1)\eta r}{60 \sqrt{d}}$.

\item Meanwhile, observe that since $\mathcal{S}_p(\delta)$ holds, it follows that

\begin{align*}
    W_p(t+1) \leq \frac{2\sqrt{2\log(T/\delta)}\beta(t+1)\eta r}{\sqrt{d}}.
\end{align*}
\end{enumerate}

Combining the bounds for $W_{g_0}, W_p, W_H$ and $W_u$, it follows that 
\begin{align*}
\norm{\hat{x}_{t+1}} \leq  & \  \norm{W_{g_0}(t+1)} + \norm{W_p(t+1)} + \norm*{W_H(t+1)} + \norm*{W_u(t+1)} \\
\leq &  \ \frac{\beta(t+1)\eta r}{\sqrt{d}} \left(\frac{1}{60} + \frac{1}{60} + \frac{1}{60} + 2\sqrt{2\log(T/\delta)} \right) \\
\leq & \ \frac{\beta(t+1)\eta r}{\sqrt{d}}\left(\frac{1}{20} + 2\sqrt{2} \right)\log(T/\delta),
\end{align*}
where the final inequality uses the fact that $0 <\delta \leq 1/e$ (which implies $\log(T/\delta) \geq 1$). 
Hence, we see that the first part of the inductive claim of \cref{eq:saddle_inductive_claim_bdd} holds with the constant $c_9 \coloneqq \frac{1}{20} + 2\sqrt{2}$, and the second part follows naturally as a consequence of our argument above.

Meanwhile, observe that for any $\eta$ such that $\eta \bar{\psi} \leq \frac{1}{2}$, we have that $(1 + \eta \gamma)^{\frac{1}{\eta \bar{\psi}}} \geq 2$. Thus, by choosing $\eta$ such that $\eta \bar{\psi} \leq \frac{1}{2}$, we have that for any $t \geq \frac{1}{\eta \bar{\psi}}$, 
\begin{align*}
    \alpha(t+1)^2 \geq \frac{1}{2} \beta(t+1)^2.
\end{align*}
Hence, following \cref{eq:W_p_variance_calc_with_alpha}, by choosing $T_s \geq \frac{1}{\eta \bar{\psi}}$, $W_p(T_s)$ is a 1-dimensional Gaussian with variance at least $\frac{2\eta^2 r^2 \beta(T_s)}{d}$, such that with probability at least 2/3,
\begin{align*}
    \norm*{W_p(T_s)} \geq\frac{\beta(T_s) \eta r}{10 \sqrt{d}}.
\end{align*}
Simultaneously, we know that on the event 
\begin{align*}
    \left(\cap_{t=0}^{T_s-1} \mathcal{R}_t(\delta)\right)  \cap \mathcal{S}_\phi(\delta)\cap \mathcal{S}_u(\delta) \cap \mathcal{S}_p(\delta),
\end{align*}
we have
\begin{align*}
    \norm{W_{g_0}(T_s)} + \norm*{W_H(T_s)} + \norm*{W_u(T_s)} \leq \frac{3\beta(T_s)\eta r}{60\sqrt{d}} = \frac{\beta(T_s)\eta r}{20\sqrt{d}}.
\end{align*}
We note that by \cref{lemma:events_for_saddle_func_decrease}, we have 
\begin{align*}
    \bbP\left(\left(\cap_{t=0}^{T_s-1} \mathcal{R}_t(\delta)\right)  \cap \mathcal{S}_\phi(\delta)\cap \mathcal{S}_u(\delta) \cap \mathcal{S}_p(\delta) \right) \geq 1 - \left(\frac{24T_s\delta}{T} + \frac{T_s \delta}{T} + \frac{T_s \delta}{T} \right) = 1 - \frac{26T_s \delta}{T}.
\end{align*}

Thus, with probability at least $2/3 - \frac{26T_s\delta}{T}$, we have
\begin{align*}
    \norm{\hat{x}_{T_s}} \geq \frac{1}{2} \norm{W_p(T_s)} \geq \frac{\beta(T_s)\eta r}{20\sqrt{d}}
\end{align*}
Thus, choosing $T_s \geq \frac{\iota}\eta {\bar{\psi}}$, where 
\begin{align*}
    \iota = \max\left\{\log\left(2\sqrt{\phi_{T_s}(\delta,F)} \frac{20\sqrt{d} \sqrt{\eta^2 \gamma^2 + 2\eta \gamma}}{\eta r}\right), 1 \right\},
\end{align*}
noting that if $\eta \bar{\psi} \leq 1/2$, then $(1+\eta \gamma)^{\frac{1}{\eta \bar{\psi}}} \geq (1+\eta \bar{\psi})^{\frac{1}{\eta \bar{\psi}}} \geq 2$, 
we have that with probability at least $2/3 - \frac{26 T_s \delta}{T}$,
\begin{align*}
    \norm{\hat{x}_{T_s}} &\geq \frac{\beta(T_s) \eta r}{20\sqrt{d}} = \frac{\eta r}{20\sqrt{d}} \frac{(1+ \eta \gamma)^{T_s}}{\sqrt{2\eta \gamma + (\eta \gamma)^2}} \\
    &\geq \frac{\eta r}{20\sqrt{d}} \frac{(1+ \eta \gamma)^{\frac{\log\left(2\sqrt{\phi_{T_s}(\delta,F)} \frac{20\sqrt{d} \sqrt{\eta^2 \gamma^2 + 2\eta \gamma}}{\eta r}\right)}{\eta \bar{\psi}}}}{\sqrt{2\eta \gamma + (\eta \gamma)^2}} \\
    &\geq \frac{\eta r}{20\sqrt{d}{\sqrt{2\eta \gamma + (\eta \gamma)^2}}} 2^{\log\left(2\sqrt{\phi_{T_s}(\delta,F)} \frac{20\sqrt{d} \sqrt{\eta^2 \gamma^2 + 2\eta \gamma}}{\eta r}\right)} > 2 \sqrt{\phi_{T_s}(\delta,F)} > 2 \sqrt{\phi(T_s,\delta)}.
\end{align*}
Thus, at least one of $\norm*{x_{T_s} - x_0}$ and $\norm*{x_{T_s}' - x_0}$ is larger than $\sqrt{\phi(T_s,\delta)}$, a contradiction. Since the two sequences have the same distribution, it follows that with probability at least $1/3 - \frac{13T_s \delta }{T}$, $f(x_{T_s}) - f(x_0) \leq -F$.
\end{proof}

In the result above, we require an upper bound on the norm of $\phi_{T_s}(\delta,F)$ to hold (i.e. \eqref{eq:phi_Ts_upper_bound}), which in turn necessitates an upper bound on $F$, the function value improvement we can expect to make. Below, we show how to choose $F$ to be as large as possible (up to constants and logarithmic factors) whilst still satisfying \eqref{eq:phi_Ts_upper_bound}, assuming that $u,r$ and $\eta$ are chosen appropriately small such that the dominant term of $\norm*{\phi_{T_s}(\delta,F)}$ scales with $F$.

\begin{lemma}
\label{lemma:derive_F_lower_bound}

Consider choosing $F$ such that 
\begin{align*}
    F = \frac{1}{2} \left(\frac{\bar{\psi}}{60c_9 \iota \rho \log(T/\delta)}\right)^2 \frac{1}{\eta T_s t_f(\delta) \left(129 + 8c'^2\beta_1(\delta;F)\left(16(\operatorname{lr}(CT^2/\delta))^2 + 1 \right) \right)}.
\end{align*}
Suppose $\eta \leq \min\left\{1, \frac{1}{t_f(\delta)}, \frac{1}{t_f\delta L}\right\}$. 
Suppose we pick $u$ and $r$ small enough such that
\begin{align*}
    u \leq \frac{r^{1/2}}{d \log(T/\delta)\rho^{1/2}}, \quad r^2 \leq \min \left\{\frac{F \bar{\psi}}{2\iota \log(T/\delta) \left(\frac{65c_5^2}{8} + 6c_1 + 1  \right)}, \frac{F}{4c_6\log(2dT/\delta)+  \frac{8c_7\iota}{\bar{\psi}}} \right\}.
\end{align*}
Then, $N_{u,r}(T_s,\delta) \leq F$, and that 
\begin{align*}
   4c_6 \eta^2 T_s \log(2dT/\delta) r^2 + 4c_7 \eta^2 T_s^2 \rho^2 u^4 d^4 \left(\log(T/\delta) \right)^4 \leq \eta T_s t_f(\delta)F. 
\end{align*}
Suppose in addition $\eta$ is small enough so that
\begin{align*}
    32\eta^2 (t_f(\delta))^2 \ep^2 \leq \frac{1}{2}\left(\frac{\bar{\psi}}{60c_9 \iota \rho \log(T/\delta)}\right)^2.
\end{align*}
Suppose also that $\bar{\psi} \leq 1$\footnote{Without loss of generality, we may set $\bar{\psi} = 1$ if $f(\cdot)$ is $(\ep, \psi, \sqrt{\rho \ep})$-strict saddle for any $\psi > 1$.} and $\eta \leq \frac{m}{d}$, so that $T_s \geq \frac{\iota}{\eta \bar{\psi}} \geq \frac{d}{m}$. Then, the condition in \cref{eq:phi_Ts_upper_bound} will be satisfied.
\end{lemma}
\begin{proof}

We note that since $\frac{\iota}{\bar{\psi}} \leq T_s \leq \frac{2\iota}{\bar{\psi}}$, it follows by our choice of $r$ that $r$ also satisfies the condition 
$$r^2 \leq \min \left\{\frac{F}{\eta T_s \log(T/\delta) \left(\frac{65c_5^2}{8} + 6c_1 + 1  \right)}, \frac{F}{4c_6\log(2dT/\delta)+  4c_7 \eta T_s} \right\}.$$
Hence, our choice of $\eta, u$ and $r$ satisfies the conditions in \cref{lemma:improve_or_localize_exact}, and it follows then that
\begin{align*}
    \phi_{T_s}(\delta,F) \! \leq \! \max\left\{128\eta T_s t_f(\delta)F, 32\eta^2 (t_f(\delta))^2 \ep^2 \right\} + 8c'^2 \beta_1(\delta;F)\eta t_f(\delta)\max\left\{\frac{16d}{m} (\operatorname{lr}(CT^2/\delta))^2 F, T_s F \right\} +  T_s \eta t_f(\delta)F,
\end{align*}
where $\beta_1(\delta;F)$ is as defined in \cref{lemma:bounding_V2}.

The condition in \cref{eq:phi_Ts_upper_bound} requires that
$$\phi_{T_s}(\delta,F) \leq \left(\frac{\bar{\psi}}{60c_9 \iota \rho \log(T/\delta)}\right)^2.$$

By our choice of $\eta$ such that
\begin{align*}
    32\eta^2 (t_f(\delta))^2 \ep^2 \leq \frac{1}{2}\left(\frac{\bar{\psi}}{60c_9 \iota \rho \log(T/\delta)}\right)^2,
\end{align*}
it suffices for us to show that
\begin{align*}
\frac{1}{2}\left(\frac{\bar{\psi}}{60c_9 \iota \rho \log(T/\delta)}\right)^2 \geq & \ 128\eta T_s t_f(\delta)F +  8c'^2 \beta_1(\delta;F)\eta t_f(\delta)\max\left\{\frac{16d}{m} (\operatorname{lr}(CT^2/\delta))^2 F, T_s F \right\} + \eta T_s t_f(\delta)F \\
= & \ 129\eta T_s t_f(\delta)F +  8c'^2 \beta_1(\delta;F)\eta t_f(\delta)\max\left\{\frac{16d}{m} (\operatorname{lr}(CT^2/\delta))^2 F, T_s F \right\}.
\end{align*}
By our assumption, we know that $T_s \geq \frac{d}{m}$. Thus, further simplifying indicates that it suffices for us to show
\begin{align}
\label{eq:derive_F_final_condition}
\frac{1}{2}\left(\frac{\bar{\psi}}{60c_9 \iota \rho \log(T/\delta)}\right)^2 \geq & \ 129\eta T_s t_f(\delta)F +  8c'^2 \beta_1(\delta;F)\eta t_f(\delta)\max\left\{16T_s (\operatorname{lr}(CT^2/\delta))^2 F, T_s F \right\}.
\end{align}
By choosing $F$ such that
\begin{align*}
    F &\leq \frac{1}{2} \left(\frac{\bar{\psi}}{60c_9 \iota \rho \log(T/\delta)}\right)^2 \frac{1}{\eta T_s t_f(\delta) \left(129 + 8c'^2\beta_1(\delta;F)\left(16(\operatorname{lr}(CT^2/\delta))^2 + 1 \right) \right)},
\end{align*}
we see that \cref{eq:derive_F_final_condition} is satisfied. 

\end{proof}
\begin{remark}
Suppose without loss of generality that $T_s = \frac{\iota}{\eta \bar{\psi}}$. Then, as a consequence of \cref{lemma:derive_F_lower_bound}, we note that the amortized function value progress of decreasing function value by $F$ over $T_s$ iterations is
\begin{align*}
    \frac{F}{T_s} &= \frac{1}{2} \left(\frac{\bar{\psi}}{60c_9 \iota \rho \log(T/\delta)}\right)^2 \frac{1}{\eta T_s^2 t_f(\delta) \left(129 + 8c'^2\beta_1(\delta;F)\left(16(\operatorname{lr}(CT^2/\delta))^2 + 1 \right) \right)} \\
    &= \eta \frac{\bar{\psi}^4}{\rho^2  } \left(\frac{1}{2\iota^2} \frac{1}{(60c_9 \iota \log(T/\delta))^2 \left(t_f(\delta) \right)\left(129 + 8c'^2\beta_1(\delta;F)\left(16(\operatorname{lr}(CT^2/\delta))^2 + 1 \right) \right)}  \right) 
\end{align*}
\end{remark}

\section{Proving the main result (informal statement in \cref{theorem:convergence_main}, full statement in \cref{theorem:convergence_full})}
\label{appendix:main_result}

In this section, we prove our main result. First, we need an additional result (\cref{lemma:func_decrease_small_tau_bound}) showing that with high probability, we can bound the function value increase if a saddle appears within $t_f(\delta)$ iterations immediately after we have had $T_s$ iterations after the previous saddle. We note that such a bound is necessary because our earlier result upper bounding function increase in $\tau$ iterations (see \cref{lemma:func_decrease_tau_bound}) focused on the case where $\tau \geq t_f(
\delta)$. Next, we state and prove \cref{theorem:convergence_full}, which is the precise version of \cref{theorem:convergence_main} in the main text.

\begin{lemma}[Function change for small $\tau$]
\label{lemma:func_decrease_small_tau_bound}
Let $c_1>0, c_4>0, c_5>0, C_1\geq 1$ be the absolute constants defined in the statements of the previous lemmas. Let $\delta \in(0,1/e]$, and suppose $\tau < t_f(\delta)$.

Let $J$ denote the interval $\{0,1\dots,\tau-1\}$ where $\tau < t_f(\delta)$.

Suppose we choose $\eta$ such that 
\begin{align}
\eta \leq 
\frac{1}{L t_f(\delta)}\cdot\min \left\{\frac{\sqrt{m}}{8c_4 (\operatorname{lr}(C_1 dmT/\delta))^{3/2} \sqrt{d} }, \frac{m}{128 c_1(\operatorname{lr}(C_1 dmT/\delta))^3 d}\right\}.\label{eq:eta_smaller_than_1/8t_f_small_tau}    
\end{align}
Suppose also we pick $u,r$ and $\eta$ as prescribed in the statement of \cref{proposition:func_decrease_contradiction}.

Suppose that $\min_{t \in J} \norm*{\nabla f(x_t)} \leq \ep$. Then, on the event
$$
\mathcal{D}_{\tau}(\delta) \coloneqq \mathcal{H}_{0,\tau}(\delta)\cap
\left(\bigcap_{t=0}^{\tau-1}\mathcal{A}_t(\delta)\right)\cap
\left(\bigcap_{t=0}^{\tau-1}\mathcal{G}_t(\delta)\right),
$$
we have the following upper bound on function value change: 
\begin{align*}
\quad f(x_{\tau}) -f(x_0)
\leq\ &
\frac{\eta}{4}\ep^2 + t_f(\delta) \eta u^4 \rho^2\cdot c_1 d^3 \left(\log\frac{T}{\delta}\right)^3 +t_f(\delta) L \eta^2 u^4 \rho^2\cdot c_1 d^4 \left(\log\frac{T}{\delta}\right)^4 
\nonumber \\
&
+ \eta  c_1r^2(128t_f(\delta) + \eta L)\log\frac{T}{\delta} + t_f(\delta) c_1 L\eta^2  r^2.
\end{align*}

Moreover, $\bbP(\mathcal{D}_{\tau}(\delta)) \geq 1 - \frac{(4t_f(\delta)+4)\delta}{T}$.
\end{lemma}
\begin{proof}
Throughout the proof, we assume that the event $\mathcal{D}_\tau(\delta)$ holds.

Let $J$ denote $\{0,1\dots,\tau-1\}$ where $\tau < t_f(\delta)$. Then, $J$ belongs to one of the two following cases.

\begin{enumerate}[label=\textbf{Case \arabic*}),itemindent=42pt,labelwidth=38pt,labelsep=4pt,leftmargin=0pt]
    \item (Gradient dominates noise): Recall that this means that for every $t \in  J$, we have
    \begin{align*}
        \norm*{\nabla f(x_t)}  > 8 t_f(\delta)\eta L \left( \frac{u}{2} \norm*{\frac{1}{m} \sum_{i=1}^m Z_{t,i} Z_{t,i}^\top \tilde{H}_{t,i} Z_{t,i} }+ \norm*{Y_t}\right).
    \end{align*}
By our choice of $\eta$ in \cref{eq:eta_smaller_than_1/8t_f}, we can apply \cref{lemma:gradient_changes_little_if_noise_small} to get
$$
\min_{t \in J} \norm*{\nabla f(x_t)} \geq \frac{1}{4} \max_{t \in J} \norm*{\nabla f(x_t)}.
$$

Thus by setting $\alpha = 128t_f(\delta)$ in \cref{eq:func_decrease_first_lemma} and by choosing $\eta$ such that 
\begin{align*}
    \frac{c_1 L\eta^2 \chi^3 d}{m} \leq \frac{\eta}{\alpha} = \frac{\eta}{128 t_f(\delta)}
    \ \ \iff\ \ 
    \eta \leq \frac{m}{128 c_1 L t_f(\delta)d \chi^3},
\end{align*}
it follows that 
\begin{align}
&-\frac{3\eta}{4} \sum_{t\in J} \frac{1}{m}\sum_{i=1}^m \abs*{Z_{t,i}^\top \nabla f(x_t)}^2 + \left(\frac{\eta}{128t_f(\delta)} + \frac{c_1 L\eta^2 \chi^3 d}{m}\right)\sum_{t\in J} \norm*{\nabla f(x_t)}^2
\nonumber \\
=\ &
-\frac{3\eta}{4} \sum_{t\in J}\frac{1}{m}\sum_{i=1}^m \abs*{Z_{t,i}^\top \nabla f(x_t)}^2
+ \frac{\eta}{64t_f(\delta)} \sum_{t\in J}\norm*{\nabla f(x_t)}^2 \nonumber \\
\leq\ &
\frac{\eta}{64t_f(\delta)} \sum_{t\in J}\norm*{\nabla f(x_t)}^2 \nonumber \\
\leq\ &
\frac{\eta}{64t_f(\delta)} \sum_{t\in J} \max_{t \in J}\norm*{\nabla f(x_t)}^2 \nonumber \\
\leq \ & \frac{16\eta}{64t_f(\delta)} \sum_{t\in J} \min_{t \in J}\norm*{\nabla f(x_t)}^2 \leq \frac{\eta}{4}  \min_{t \in J}\norm*{\nabla f(x_t)}^2 \leq \frac{\eta}{4} \ep^2, \label{eq:func_decrease_case1_bdd_small_tau}
\end{align}
where the final bound holds since we assumed $\min_{t \in J}\norm*{\nabla f(x_t)} \leq \ep$.

\item (Gradient does not dominate noise): there exists some $t \in J$ such that 
    \begin{align*}
        \norm*{\nabla f(x_t)}  \leq 8 t_f(\delta)\eta L \left( \frac{u}{2} \norm*{\frac{1}{m} \sum_{i=1}^m Z_{t,i} Z_{t,i}^\top \tilde{H}_{t,i} Z_{t,i} }+ \norm*{Y_t}\right).
    \end{align*}
By our choice of $\eta$ in \cref{eq:eta_smaller_than_1/8t_f}, we can apply \cref{lemma:gradient_norm_small_if_noise_larger} to get
\begin{align*}
    \norm*{\nabla f(x_t)} \leq c_5 t_f(\delta)\eta L \left(u^2 d^2 \rho\left(\log\frac{T}{\delta}\right)^2 + \sqrt{ 1 + \frac{\log(T/\delta)}{d}}  r \right) \quad \quad \forall t \in J.
\end{align*}
Note that, by our choices of the parameters $\eta,u,r$, it can be shown that
\begin{align*}
c_5 t_f(\delta)\eta L \left(u^2 d^2 \rho\left(\log\frac{T}{\delta}\right)^2 + \sqrt{ 1 + \frac{\log(T/\delta)}{d}}  r \right)
<\epsilon,
\end{align*}

Hence, by setting $\alpha = 128t_f(\delta)$ in \cref{eq:func_decrease_first_lemma} and choosing $\eta$ such that 
\begin{align*}
    \frac{c_1 L \eta^2 \chi^3 d}{m} \leq \frac{\eta}{\alpha} 
    = \frac{\eta}{128 t_f(\delta)},
\end{align*}
it follows that
\begin{align}
&\left( \frac{\eta}{128t_f(\delta)} + \frac{c_1 L \eta^2 \chi^3 d}{m}  \right)\sum_{t\in J} \norm*{\nabla f(x_t)}^2
\nonumber \\
\leq\ &
\frac{\eta}{64t_f(\delta)}
\sum_{t\in J}
\left(c_5 t_f(\delta)\eta L \left(u^2 d^2 \rho\left(\log\frac{T}{\delta}\right)^2 + \sqrt{ 1 + \frac{\log(T/\delta)}{d}}  r \right)\right)^2 
\nonumber \\
\leq\ &
\frac{\eta}{64t_f(\delta)}
\sum_{t\in J}
\ep^2
\nonumber \\
\leq\ &
\frac{\eta}{64}\ep^2 \label{eq:func_decrease_case2_bdd_small_tau}
\end{align}
\end{enumerate}
Combining both cases above (\cref{eq:func_decrease_case1_bdd_small_tau} and \cref{eq:func_decrease_case2_bdd_small_tau}), we see that for the choice $\alpha = 128t_f(\delta)$, the bound 
\begin{align}
&-\frac{3\eta}{4} \sum_{t\in J} \frac{1}{m}\sum_{i=1}^m \abs*{Z_{t,i}^\top \nabla f(x_t)}^2 + \left(\frac{\eta}{128t_f(\delta)} + \frac{c_1 L\eta^2 \chi^3 d}{m}\right)\sum_{t\in J} \norm*{\nabla f(x_t)}^2
\leq  \frac{\eta}{4} \ep^2 \label{eq:func_decrease_grad_bdd_small_tau}
\end{align}
always holds.

Recall by \cref{eq:func_decrease_first_lemma} that we have
\begin{align*}
\quad f(x_{\tau}) -f(x_0)
\leq\ &
-\frac{3\eta}{4} \sum_{t=0}^{\tau-1} \frac{1}{m} \sum_{i=1}^m \abs*{Z_{t,i}^\top \nabla f(x_t)}^2  + \left(\frac{\eta}{\alpha} + \frac{c_1 L\eta^2 \chi^3 d}{m}\right) \sum_{t=0}^{\tau-1} \norm*{\nabla f(x_t)}^2
\nonumber \\
&
+ \tau \eta u^4 \rho^2\cdot c_1 d^3 \left(\log\frac{T}{\delta}\right)^3
+ \tau L \eta^2 u^4 \rho^2\cdot c_1 d^4 \left(\log\frac{T}{\delta}\right)^4 
\nonumber \\
&
+ \eta  c_1r^2(\alpha + \eta L)\log\frac{T}{\delta} + \tau c_1 L\eta^2  r^2.
\end{align*}
By plugging in \cref{eq:func_decrease_grad_bdd_small_tau} above, as well as the choice $\alpha = 128t_f(\delta)$, we see that 
\begin{align*}
\quad f(x_{\tau}) -f(x_0)
\leq\ &
\frac{\eta}{4}\ep^2 + t_f(\delta) \eta u^4 \rho^2\cdot c_1 d^3 \left(\log\frac{T}{\delta}\right)^3 +t_f(\delta) L \eta^2 u^4 \rho^2\cdot c_1 d^4 \left(\log\frac{T}{\delta}\right)^4 
\nonumber \\
&
+ \eta  c_1r^2(128t_f(\delta) + \eta L)\log\frac{T}{\delta} + t_f(\delta) c_1 L\eta^2  r^2.
\end{align*}

We can now complete our proof by using the union bound (suppressing the dependence of some of the events on $\delta$ for notational simplicity) to derive
\begin{align*}
\bbP(\mathcal{D}_{\tau}^c)
\leq\ &
\bbP(\mathcal{H}_\tau^c)
+\sum_{t=0}^{\tau-1}\bbP(\mathcal{A}_t^c)
+\sum_{t=0}^{\tau-1}\bbP(\mathcal{G}_t^c) \\
\leq\ &
\frac{(\tau+4)\delta}{T}
+\frac{\tau}{T}\delta+2\frac{\tau}{T}\delta
\leq \frac{(4t_f(\delta)+4)}{T}\delta
\qedhere
\end{align*}
\end{proof}

Armed with \cref{proposition:saddle_point_func_decrease} and \cref{lemma:derive_F_lower_bound}, we are now ready to show for $T$ sufficiently large, with high probability, there can be no more than $T/4$ $\ep$-saddle points. Combined with \cref{proposition:func_decrease_contradiction}, this yields the following result.
\begin{theorem}
\label{theorem:convergence_full}
Suppose we pick $u,r,\eta$ such that they satisfy the conditions in \cref{proposition:saddle_point_func_decrease} and \cref{lemma:derive_F_lower_bound}. Suppose $F$ is chosen as prescribed in \cref{lemma:derive_F_lower_bound}. Suppose that $\bar{\psi} \leq 1$, so that $T_s \geq \frac{\iota}{\eta \bar{\psi}} \geq \frac{d}{mL}$\footnote{Recall we focus on the case $\bar{\psi} \leq L$, since otherwise, by the $L$-Lipschitz assumption, $\lambda_{\min}(\nabla^2 f(x)) \geq -L$ for all $x \in \bbR^d$, i.e. $\ep$-first order stationary points are also $\ep$-second order stationary points.}. Suppose we pick $T_s$ as prescribed in \cref{proposition:saddle_point_func_decrease}. Suppose in addition we pick $r$ such that
\begin{align*}
r^2 \leq \min\left\{\frac{\ep^2}{4(130c_1 t_f(\delta) + c_1 \log(T/\delta) + c_1)},  \frac{F \bar{\psi} }{80\iota \log(T/\delta) \left(\frac{65c_5^2}{8} + 132c_1 + 1  \right)}\right\}.
\end{align*}
Suppose also that we choose $\eta$ such that

\begin{align*}
    \eta &\leq \frac{0.1}{2\ep^2} \frac{\bar{\psi}}{2 \iota} \frac{1}{2} \left(\frac{\bar{\psi}}{60c_9 \iota \rho \log(T/\delta)}\right)^2 \frac{1}{ t_f(\delta) \left(129 + 8c'^2\beta_1(\delta;F)\left(16(\operatorname{lr}(CT^2/\delta))^2 + 1 \right) \right)}
\end{align*}

Suppose 
\begin{align}
\label{eq:T_final_lower_bdd}
    T \geq \left\{\frac{256t_f(\delta) \left(\left(f(x_0) - f^*\right) + \ep^2/L)  \right)}{\eta \ep^2}, \frac{\varphi \rho^2 \left(f(x_0)-f^* \right)}{\eta \bar{\psi}^4}, 256 \ceil{\frac{\iota}{\eta \bar{\psi}}}, 256 t_f(\delta), 1024 \right\},
\end{align}
where 
\begin{align*}
    \varphi \coloneqq 20 \left(2\iota^2 (60c_9 \iota \log(T/\delta))^2 \left(t_f(\delta) \right)\left(129 + 8c'^2\beta_1(\delta;F)\left(16(\operatorname{lr}(CT^2/\delta))^2 + 1 \right) \right)  \right).
\end{align*}

Then, with probability at least $1 -22\delta$,  there are at least $T/2$ $\ep$-approximate second order stationary points. 
\end{theorem}
\begin{proof}
Consider defining the following sequence of stopping times:
\begin{align*}
    &\tau_1 = \inf_t \{t \leq T: \norm*{\nabla f(x_t)} < \ep, \lambda_{\min}(\nabla^2 f(x_t)) \leq -\sqrt{\rho \ep} \}, \\
    &\tau_{i+1} = \inf_t \{t \leq T: t > \tau_{i} + T_s, \norm*{\nabla f(x_t)} < \ep, \lambda_{\min}(\nabla^2 f(x_t)) \leq -\sqrt{\rho \ep} \}, \quad \forall 1 \leq i \leq \floor{T/T_s}.
    \end{align*}
\end{proof}
We note that if $\tau_i = T$, then $\tau_{j} = T$ for any $j > i$. Let $N_s$ denote the (random) number of saddle points encountered in $T$ iterations. 

We observe that we can decompose the function change as 
\begin{align*}
    & \ f(x_T) - f(x_0) \\
    = & \ (f(x_{\tau_{N_s}}) - f(x_0)) + (f(x_T) - f(x_{\tau_{N_s}})) \\
    = & \ (f(x_{\tau_1}) - f(x_0)) + \sum_{i=1}^{N_s} \left(f(x_{\tau_i + T_s}) - f(x_{\tau_i})\right) + \sum_{i=1}^{N_s-1} \left(f(x_{\tau_{i + 1}})  - f(x_{\tau_i + T_s })\right) + (f(x_T) - f(x_{\tau_{N_s}})) \\
    = & \  \underbrace{\sum_{i=1}^{N_s} \left(f(x_{\tau_i + T_s}) - f(x_{\tau_i})\right)}_{U_1} + \underbrace{(f(x_{\tau_1}) - f(x_0)) +\sum_{i=1}^{N_s-1} \left(f(x_{\tau_{i + 1}})  - f(x_{\tau_i + T_s })\right)}_{U_2} + (f(x_T) - f(x_{\tau_{N_s}})).
\end{align*}

We first consider $U_1$. Letting $x_{j} := x_T$ for any $j \geq T$, we have that 
\begin{align*}
    \sum_{i=1}^{N_s} f(x_{\tau_i + T_s}) - f(x_{\tau_i}) &= \sum_{i=1}^{\floor{T/T_s}}  \left(f(x_{\tau_i + T_s}) - f(x_{\tau_i})\right) 1_{\tau_i < T} 
\end{align*}
Now, by \cref{eq:func_max_increase}, observe that with probability at least $1 - \frac{(5T_s+4)\delta}{T} \geq 1  - \frac{6T_s\delta}{T}$ (note $T_s \geq 4$), for any $1 \leq i \leq T/T_s$, we have that
\begin{align*}
\left(f(x_{\tau_i + T_s}) - f(x_{\tau_i})\right) 1_{\tau_i < T} \leq\ &
\tau \frac{c_5^2}{64} \eta^3 t_f(\delta)^2 L^2 \left(u^2 d^2 \rho\left(\log\frac{T}{\delta}\right)^2 + \sqrt{2\log(T/\delta)}  r \right)^2 \nonumber \\
&
+ \tau \eta u^4 \rho^2 \cdot c_1d^3\left(\log\frac{T}{\delta}\right)^3 + \tau L \eta^2 u^4 \rho^2\cdot c_1d^4 \left(\log\frac{T}{\delta}\right)^4
\nonumber \\
&
+ \eta c_1r^2(128t_f(\delta) + \eta L )\log\frac{T}{\delta} + \tau c_1 L\eta^2  r^2 \\
&:= M_{u,r,T_s}.
\end{align*}

Suppose we pick $u, r$ such that $M_{u,r, T_s} \leq 0.1F$. Recall from \cref{proposition:saddle_point_func_decrease} that with probability at least $1/3 - \frac{13T_s\delta}{T}$, $\left(f(x_{\tau_i + T_s}) - f(x_{\tau_i})\right) 1_{\tau_i < T} \leq -F$. Choosing $\delta$ such that $ 1/3 - \frac{13T_s\delta}{T} \geq 0.3$, and letting $\mu = 0.1F$, we note that $\abs*{-F + \mu} = 0.9F \geq \frac{0.7}{0.3} 0.2F \geq \frac{0.7}{0.3}(M_{u,r,T_s} + \mu)$.

Now, let $\gE_{\tau_i}$ denote the bad event on which 
\begin{align*}
    \mbox{neither }\left(f(x_{\tau_i + T_s}) - f(x_{\tau_i})\right) 1_{\tau_i < T} \leq -F, \quad \mbox{nor } \left(f(x_{\tau_i + T_s}) - f(x_{\tau_i})\right) 1_{\tau_i < T} \leq M_{u,r, T_s} \leq 0.1F.
\end{align*}
We know that $\gE_{\tau_i}$ has probability at most $\frac{6T_s \delta}{T}$. Let $\gE_{\tau} := \cup_{i=1}^{\floor{T/T_s}} \gE_{\tau_i}$, such that $\bbP(\gE_{\tau}) \leq 6 \delta$. Then, by applying the weakened supermartingale inequality in \cref{proposition:weakened_supermartingale_concentration_inequality}, we have 
\begin{align*}
    \bbP\left(\sum_{i=1}^{T/T_s}  \left(f(x_{\tau_i + T_s}) - f(x_{\tau_i})\right) 1_{\tau_i < T}  \geq -N_s 0.9F + s\right) \leq \bbE\left[\exp\left(-\frac{s^2}{4N_sF^2} \right)\right] + \bbP(\gE_{\tau})  \leq \exp\left(-\frac{s^2}{4(T/T_s)F^2} \right) + 6\delta.
\end{align*}
Now, pick $s = 2F \sqrt{\sqrt{\log(1/\delta)}T/T_s }$, then 
\begin{align*}
    \bbP\left(\sum_{i=1}^{T/T_s}  \left(f(x_{\tau_i + T_s}) - f(x_{\tau_i})\right) 1_{\tau_i < T}  \geq -N_s 0.9F + 2F\sqrt{\sqrt{\log(1/\delta)}T/T_s }\right) \leq 7 \delta. 
\end{align*}
Note that supposing for contradiction that there are at least $T/4$ saddles, we must then have that $N_s \geq T/(4T_s)$, such that
\begin{align*}
    -N_s 0.9F + 2F\sqrt{\sqrt{\log(1/\delta)}T/T_s }) \leq F (-0.9 T/(4T_s) + (2\sqrt{\sqrt{\log(1/\delta)}T/T_s})) \leq F(-0.1 T/T_s),
\end{align*}
where we may ensure the last inequality by picking $T/T_s$ such that
$$T/T_s \geq \left(\frac{2}{0.125}\right)^2 \sqrt{\log(1/\delta)} = 256 \sqrt{\log(1/\delta)}.$$
Note that our choice of $T$ ensures this.

Thus, with probability at least $1 - 7\delta$,
\begin{align*}
    U_1 = \sum_{i=1}^{T/T_s}  \left(f(x_{\tau_i + T_s}) - f(x_{\tau_i})\right) 1_{\tau_i < T} \leq -(0.1T/T_s)F. 
\end{align*}

Next, we bound the summand $U_2$. Recall that
\begin{align*}
    U_2 = (f(x_{\tau_1}) - f(x_0)) +\sum_{i=1}^{N_s-1} \left(f(x_{\tau_{i + 1}})  - f(x_{\tau_i + T_s })\right).
\end{align*}
Without loss of generality, we may analyze each of the summands $f(x_{\tau_{i + 1}})  - f(x_{\tau_i + T_s })$ in the same way as we treat $(f(x_{\tau_1}) - f(x_0))$. Let us then consider the summand $f(x_{\tau_1}) - f(x_0)$. There are two cases to consider.

\begin{enumerate}
    \item The first is when $\tau_1 < t_f(\delta)$. In this case, since we know that $\norm*{\nabla f(x_{\tau_1})} \leq \ep$ (as $x_{\tau_1}$ is an $\ep$-saddle point), it follows 
    by \cref{lemma:func_decrease_small_tau_bound} that 
\begin{align*}
\quad f(x_{\tau_1}) -f(x_0)
\leq\ &
\frac{\eta}{4}\ep^2 + t_f(\delta) \eta u^4 \rho^2\cdot c_1 d^3 \left(\log\frac{T}{\delta}\right)^3 +t_f(\delta) L \eta^2 u^4 \rho^2\cdot c_1 d^4 \left(\log\frac{T}{\delta}\right)^4 
\nonumber \\
&
+ \eta  c_1r^2(128t_f(\delta) + \eta L)\log\frac{T}{\delta} + t_f(\delta) c_1 L\eta^2  r^2
\end{align*}
with probability at least $1 - \frac{(4t_f(\delta) +4)\delta}{T}$.
\item The second case is when $\tau_1 \geq t_f(\delta)$. In this case, by \cref{lemma:func_decrease_tau_bound}, we have that
\begin{align*}
f(x_{\tau_1}) - f(x_0) \leq\ &
\tau_1 \frac{c_5^2}{64} \eta^3 t_f(\delta)^2 L^2 \left(u^2 d^2 \rho\left(\log\frac{T}{\delta}\right)^2 + \sqrt{2\log(T/\delta)}  r \right)^2 \nonumber \\
&
+ \tau_1 \eta u^4 \rho^2 \cdot c_1d^3\left(\log\frac{T}{\delta}\right)^3 + \tau_1 L \eta^2 u^4 \rho^2\cdot c_1d^4 \left(\log\frac{T}{\delta}\right)^4
\nonumber \\
&
+ \eta c_1r^2(128t_f(\delta) + \eta L )\log\frac{T}{\delta} + \tau_1 c_1 L\eta^2  r^2.  
\end{align*}
with probability at least $1 - \frac{(5\tau_1 +4) \delta}{T}$.
\end{enumerate}
By our choice of $u$, we know that
\begin{align*}
& t_f(\delta) \eta u^4 \rho^2\cdot c_1 d^3 \left(\log\frac{T}{\delta}\right)^3 +t_f(\delta) L \eta^2 u^4 \rho^2\cdot c_1 d^4 \left(\log\frac{T}{\delta}\right)^4
+ \eta  c_1r^2(128t_f(\delta) + \eta L)\log\frac{T}{\delta} + t_f(\delta) c_1 L\eta^2  r^2 \\
\leq & \ t_f(\delta) r^2 c_1 + t_f(\delta) r^2 c_1 + c_1 r^2 (128 t_f(\delta) + 1) \log(T/\delta) + c_1 r^2 \\
= & \ r^2(130c_1 t_f(\delta) + c_1 \log(T/\delta) + c_1).
\end{align*}
Hence, by picking $r$ such that
$$r \leq \frac{\ep^2}{4(130c_1 t_f(\delta) + c_1 \log(T/\delta) + c_1)},$$
it follows that
\begin{align*}
\frac{\eta\ep^2}{4}    \geq &  \ t_f(\delta) \eta u^4 \rho^2\cdot c_1 d^3 \left(\log\frac{T}{\delta}\right)^3 +t_f(\delta) L \eta^2 u^4 \rho^2\cdot c_1 d^4 \left(\log\frac{T}{\delta}\right)^4 
\nonumber \\
&
+ \eta  c_1r^2(128t_f(\delta) + \eta L)\log\frac{T}{\delta} + t_f(\delta) c_1 L\eta^2  r^2.
\end{align*}
Then, if $\tau_1 < t_f(\delta)$, with probability at least $1 - \frac{(5t_f(\delta)+4)}{\delta}$, 
$$f(x_{\tau_1}) - f(x_0) \leq \frac{\eta \ep^2}{2}.$$
Suppose also that we pick $r$ 
such that
\begin{align*}
r^2 \leq \frac{F \sqrt{\rho \ep}}{80\iota \log(T/\delta) \left(\frac{65c_5^2}{8} + 132c_1 + 1  \right)} \leq \frac{F}{40 \eta T_s \log(T/\delta) \left(\frac{65c_5^2}{8} + 132c_1 + 1  \right)}.
\end{align*}
Then, it can be verified that
\begin{align*}
    \frac{F}{40} \frac{T}{T_s} \geq & \ T
\frac{c_5^2}{64} \eta^3 t_f(\delta)^2 L^2 \left(u^2 d^2 \rho\left(\log\frac{T}{\delta}\right)^2 + \sqrt{2\log(T/\delta)}  r \right)^2 \nonumber \\
&
+ T \eta u^4 \rho^2 \cdot c_1d^3\left(\log\frac{T}{\delta}\right)^3 + T  L \eta^2 u^4 \rho^2\cdot c_1d^4 \left(\log\frac{T}{\delta}\right)^4
\nonumber \\
&
+ \frac{T}{T_s}\eta c_1r^2(128t_f(\delta) + \eta L )\log\frac{T}{\delta} + T c_1 L\eta^2  r^2.  
\end{align*}
Then, by a union bound, it follows that with probability at least $1 - 9\delta$,
\begin{align*}
    U_2 &= (f(x_{\tau_1}) - f(x_0)) +\sum_{i=1}^{N_s-1} \left(f(x_{\tau_{i + 1}})  - f(x_{\tau_i + T_s })\right) \\
    &\leq \frac{T}{T_s}\frac{\eta \ep^2}{2} + \frac{F}{40} \frac{T}{T_s}
\end{align*}
Therefore, by the union bound,  with probability at least $1 - 16\delta$,
\begin{align*}
    f(x_{\tau_{N_s}}) - f(x_0) &= U_1 + U_2 \leq \frac{T}{T_s}\left(-0.1 F + \eta \ep^2/2 + \frac{F}{40}\right)  
\end{align*}
By recalling our choice of $F$ in \cref{lemma:derive_F_lower_bound}, by choosing $\eta$ such that 
\begin{align*}
    \eta &\leq \frac{0.1}{2\ep^2} \frac{\bar{\psi}}{2 \iota} \frac{1}{2} \left(\frac{\bar{\psi}}{60c_9 \iota \rho \log(T/\delta)}\right)^2 \frac{1}{ t_f(\delta) \left(129 + 8c'^2\beta_1(\delta;F)\left(16(\operatorname{lr}(CT^2/\delta))^2 + 1 \right) \right)} \\
    &\leq \frac{0.1}{2\ep^2} \frac{1}{2} \left(\frac{\sqrt{\rho\ep}}{60c_9 \iota \rho \log(T/\delta)}\right)^2 \frac{1}{ \eta T_s t_f(\delta) \left(129 + 8c'^2\beta_1(\delta;F)\left(16(\operatorname{lr}(CT^2/\delta))^2 + 1 \right) \right)} = \frac{0.1F}{2\ep^2},
\end{align*}
it follows that with probability at least $1 - 16\delta$,
\begin{align*}
    f(x_{\tau_{N_s}}) - f(x_0) &= U_1 + U_2 \\
    &\leq \frac{T}{T_s}\left(-0.1 F + \eta \ep^2/2 + \frac{F}{40}\right)  \\
    &\leq \frac{T}{T_s}(-0.1 F +0.1F/4 + 0.1F/4) = \frac{T}{T_s}(-0.05 F).
\end{align*}

Choose $T$ such that 
\begin{align*}
    -(0.05 T/T_s) F \leq -(f(x_0) - f^*) \iff T \geq \frac{20T_s(f(x_0) -f^*)}{F}  \geq \frac{\varphi \rho^2 \left(f(x_0)-f^* \right)}{\eta \bar{\psi}^4}
\end{align*}
yields a contradiction, where 
\begin{align*}
    \varphi \coloneqq 20 \left(2\iota^2 (60c_9 \iota \log(T/\delta))^2 \left(t_f(\delta) \right)\left(129 + 8c'^2\beta_1(\delta;F)\left(16(\operatorname{lr}(CT^2/\delta))^2 + 1 \right) \right)  \right)
\end{align*}

 Hence, with probability at least $1 - 16\delta$, there cannot be more than $T/4$ saddle points. In addition, with probability at least $1 - 6\delta$, by \cref{proposition:func_decrease_contradiction}, there cannot be more than $T/4$ iterates with $\norm*{\nabla f(x_t)} \geq \ep$. Hence, with probability at least $1 - 22\delta$, there are at least $T/2$ $\ep$-approximate second order stationary points.

\section{More complete discussion of simulations}
\label{appendix:simulations}

We test the performance of our proposed algorithm with two-point estimators (ZOPGD-2pt) against existing zeroth-order benchmarks using the \emph{octopus function} (proposed in \cite{du2017gradient}) of varying dimensions.\footnote{Our code can be found at \url{https://github.com/rafflesintown/escape-saddle-points-2pt}} It is known that the octopus function defined on $\bbR^d$, which chains $d$ saddle points sequentially, takes exponential (in $d$) time for exact gradient descent to escape; it has thus emerged as a popular benchmark to evaluate and compare the performance of algorithms that seek to escape saddle points. In our experiments, we compare the performance of our two-point estimator algorithm (ZOPGD-2pt) with PAGD (Algorithm 1 in \cite{vlatakis2019efficiently}) and ZO-GD-NCF (see \cite{zhang2022zeroth}), which are the only two existing zeroth-order algorithms that have (a) a $\tilde{O}(\nicefrac{d}{\ep^2})$ sample complexity for escaping saddle points (with the latter algorithm yielding the tightest bounds), and (b) performed the best empirically on escaping saddle points (see the simulation results in \cite{zhang2022zeroth}). We note that both PAGD and ZO-GD-NCF have to use $2d$ function evaluations per iteration to estimate the gradient while our algorithm only needs to use $2$ function evaluations. In our plots, we plot the function value against the number of function evaluations. For completeness, we also plot the performance of exact gradient descent (normalized such that its $x$-axis is also the number of function queries).

We tested the algorithms for $d = 10$ and $d = 30$. To account for the stochasticity in the algorithms, for each algorithm, we computed the average and standard deviation over 30 trials, and plotted the mean trajectory with an additional band that represents $1.5$ times the standard deviation. For our algorithm\'s hyperparameters, we picked
\begin{align}
    \eta = \frac{1}{4dL}, u = 10^{-2}, r = 0.05,  m = 1 (\mbox{ i.e. two-point estimator})
\end{align}
For PAGD, we used the hyperparameters listed in their paper, and for ZO-GD-NCF, we used the code from their Neurips submission. We note in particular that both methods used the step-size $\frac{1}{4L}$. For initialization, we chose a random $x_0$ near the saddle point at the origin, drawn from $N(0, 10^{-3} I_{d \times d})$\footnote{Using the random seed in our code, we note that $\norm*{\nabla f(x_0)} = 0.011$ for $d= 10$ and $\norm*{\nabla f(x_0)} = 0.030$ for $d = 30$.} (fixed for all trials and all algorithms).

As we can see in \cref{fig:comparison_octopus}, in both cases, our algorithm reaches the global minimum of the octopus function in significantly fewer function evaluations than PAGD and ZO-GD-NCF (approximately 2.5 times faster than ZO-GD-NCF, and approximately 3 times faster than PAGD), despite our algorithm only using $2$ function evaluations per iteration compared to $2d$ function evaluations per iteration for both PAGD and ZO-GD-NCF. As a sanity check, we note that the number of function evaluations required for PAGD and ZO-GD-NCF to reach the global minimum approximately matches that in Figure 1 of \cite{zhang2022zeroth}; here the correspondence is only approximate since \cite{zhang2022zeroth} only plots one trial while we compute the mean and standard deviation of 30 trials.

This result suggests that in addition to the theoretical convergence guarantees, there might also be empirical benefits to using two-point estimators versus existing $2d$-point estimators in the zeroth-order escaping saddle point literature.

\begin{figure}[t]
\captionsetup[subfigure]{justification=centering}
\centering
\begin{subfigure}{.5\textwidth}
  \centering
  \includegraphics[width=.9\linewidth]{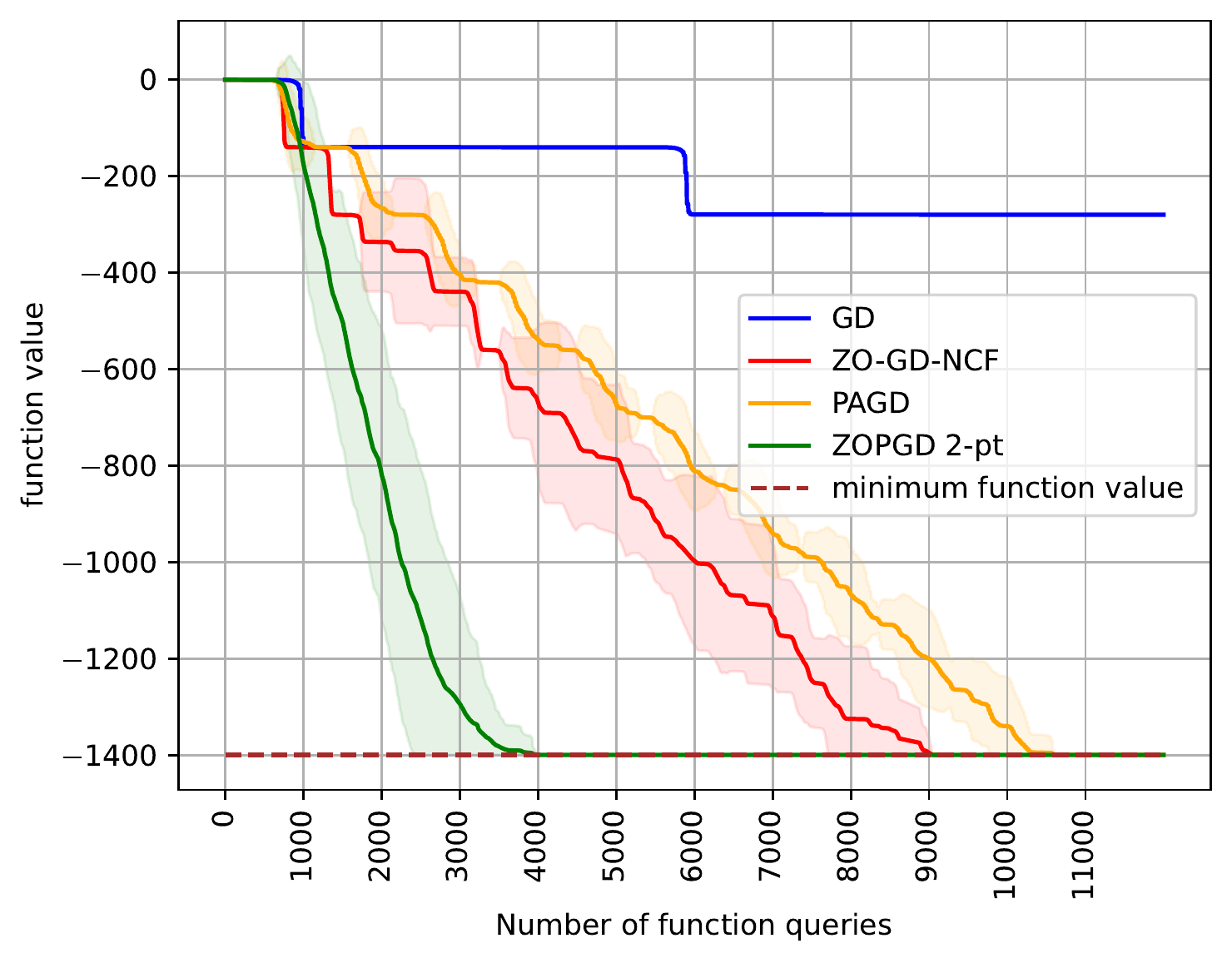}
  \caption{$d = 10$}
  \label{fig:d=15}
\end{subfigure}%
\begin{subfigure}{.5\textwidth}
  \centering
  \includegraphics[width=.9\linewidth]{images/pagd_vs_zopgd_vs_ncf_d_30_ntrials_30.pdf}
  \caption{$d = 30$}
  \label{fig:d=30}
\end{subfigure}
\caption{Performance on toy octopus function, with $\tau = e , L = e, \gamma = 1$ (Here, $\tau, L, \gamma$ are parameters determining the properties of $f$. Our parameter choice is consistent with that in \cite{zhang2022zeroth}. See \cite{du2017gradient} for details about the definitions of $\tau, L$ and $\gamma$.).}
\label{fig:comparison_octopus}
\end{figure}


\end{document}